\documentclass[oneside, 12pt]{amsart}
\usepackage{amscd, amssymb, amsmath, mathrsfs}
\usepackage[a4paper, top=3cm, left=2.8cm, right=2.8cm]{geometry}
\usepackage[english]{babel}
\usepackage{booktabs}
\usepackage{tikz-cd}
\usepackage{url}

\usepackage{lscape}
\usepackage{longtable}
\usepackage{makecell}
\usepackage{tabularx}
\usepackage{ltablex}
\newcolumntype{C}[1]{>{\centering\arraybackslash}p{#1}}

\usepackage[pdftex, colorlinks=true,  citecolor=blue, linkcolor=blue, linktocpage=true, backref]{hyperref}
\usepackage{mathdots}
\usepackage{enumitem}

\usepackage{etoolbox}

\numberwithin{equation}{section}

\newcommand\mylabel[1]{\label{#1}\marginpar{\vspace{-1ex}\medskip\medskip\footnotesize \tt #1}}
\renewcommand\mylabel[1]{\label{#1}}
\newcommand{\mydate}{
\number\day\space
\ifcase\month \or January\or February\or March\or April\or May\or June\or July\or August\or September\or October\or November\or December\fi 
\space\number\year}

\DeclareUrlCommand\arXiv{\urlstyle{tt}}

\newtheorem{theorem}{Theorem}[section]
\newtheorem{maintheorem}{Theorem}
\newtheorem{lemma}[theorem]{Lemma}
\newtheorem{proposition}[theorem]{Proposition}
\newtheorem{corollary}[theorem]{Corollary}

\theoremstyle{definition}
\newtheorem{maindefinition}[maintheorem]{Definition}
\newtheorem{definition}[theorem]{Definition}
\newtheorem{example}[theorem]{Example}

\newtheorem{remark}[theorem]{Remark}

\newtheorem*{acknowledgement}{Acknowledgement}

\theoremstyle{remark}

\newcommand{\NN}{\mathbb{N}}
\newcommand{\ZZ}{\mathbb{Z}}
\newcommand{\QQ}{\mathbb{Q}}
\newcommand{\RR}{\mathbb{R}}
\newcommand{\CC}{\mathbb{C}}
\newcommand{\FF}{\mathbb{F}}

\newcommand{\PP}{\mathbb{P}}
\renewcommand{\AA}{\mathbb{A}}
\newcommand{\GG}{\mathbb{G}}
\newcommand{\bbS}{\mathbb{S}}

\newcommand{\shE}{\mathscr{E}}
\newcommand{\shF}{\mathscr{F}}

\newcommand{\shI}{\mathscr{I}}

\newcommand{\shN}{\mathscr{N}}
\newcommand{\shL}{\mathscr{L}}

\newcommand{\catO}{\mathcal{O}}

\newcommand{\Aff}{\text{\rm Aff}}
\newcommand{\alg}{\text{\rm alg}}

\newcommand{\Alb}{\operatorname{Alb}}

\newcommand{\Aut}{\operatorname{Aut}}

\newcommand{\Br}{\operatorname{Br}}
\newcommand{\Bl}{\operatorname{Bl}}

\newcommand{\can}{\text{\rm  can}}

\newcommand{\Chev}{\text{\rm  Chev}}

\newcommand{\Der}{\operatorname{Der}}

\newcommand{\disc}{\operatorname{disc}}

\newcommand{\End}{\operatorname{End}}

\newcommand{\Ext}{\operatorname{Ext}}

\newcommand{\Flag}{\operatorname{Flag}}

\newcommand{\Frob}{\varphi}

\newcommand{\Gal}{\operatorname{Gal}}
\newcommand{\GL}{\operatorname{GL}}
\newcommand{\gl}{\operatorname{\mathfrak{gl}}}

\newcommand  {\Grass}{\operatorname{Grass}}

\newcommand{\Hom}{\operatorname{Hom}}
\newcommand{\Hilb}{\operatorname{Hilb}}

\newcommand{\id}{{\operatorname{id}}}

\newcommand{\inv}{\operatorname{inv}}

\newcommand{\I}{\text{\rm I}}
\newcommand{\II}{\text{\rm II}}

\newcommand{\IV}{\text{\rm IV}}

\newcommand{\Kernel}{\operatorname{Ker}}

\newcommand{\Lie}{\operatorname{Lie}}

\newcommand{\lra}{\longrightarrow}

\newcommand{\maxid}{\mathfrak{m}}

\newcommand{\MW}{\operatorname{MW}}

\newcommand{\Num}{\operatorname{Num}}
\newcommand{\NE}{\operatorname{\overline{NE}}}

\renewcommand{\O}{\mathscr{O}}

\newcommand{\ord}{\operatorname{ord}}

\newcommand{\Pic}{\operatorname{Pic}}
\newcommand{\PGL}{\operatorname{PGL}}

\newcommand{\pr}{\operatorname{pr}}
\newcommand{\Proj}{\operatorname{Proj}}

\newcommand{\quadand}{\quad\text{and}\quad}

\newcommand{\ra}{\rightarrow}

\newcommand{\Rad}{\operatorname{Rad}}

\newcommand{\Reg}{\operatorname{Reg}}

\newcommand{\Res}{\operatorname{Res}}

\newcommand{\sep}{{\operatorname{sep}}}

\newcommand{\Sing}{\operatorname{Sing}}
\newcommand{\SL}{\operatorname{SL}}

\newcommand{\Sp}{\operatorname{Sp}}
\newcommand{\Spec}{\operatorname{Spec}}

\newcommand{\SU}{\operatorname{SU}}

\newcommand{\Supp}{\operatorname{Supp}}
\newcommand{\Sym}{\operatorname{Sym}}

\newcommand{\U}{\operatorname{U}}

\newcommand{\val}{\operatorname{val}}

\newcommand\isomarrow{\stackrel{\cong}{\longrightarrow}}

\newcommand{\lieg}{\mathfrak{g}}
\newcommand{\bX}{{\bar{X}}}

\newcommand{\DLD}{\mathscr{D}}

\newcommand{\Sz}{\operatorname{Sz}}

\begin{document}

\title[Deligne--Lusztig varieties]
      {Deligne--Lusztig varieties whose canonical divisors have negativity}

\author[Ulrich G\"ortz]{Ulrich G\"ortz}
\address{Universität Duisburg-Essen, Fakultät für Mathematik, 45117 Essen,  Germany}
\curraddr{}
\email{ulrich.goertz@uni-due.de}
 
\author[Stefan Schr\"oer]{Stefan Schr\"oer}
\address{Heinrich Heine University D\"usseldorf, Faculty of Mathematics and Natural Sciences, Mathematical Institute, 40204 D\"usseldorf, Germany}
\curraddr{}
\email{schroeer@math.uni-duesseldorf.de}


\dedicatory{\mydate}

\begin{abstract} 
We investigate compactified Deligne--Lusztig varieties whose canonical divisor, when expressed as 
a linear combination of boundary divisors, has all coefficients strictly negative or zero. 
In dimension two we obtain explicit descriptions: The case $C_2$  gives the supersingular K3 surface with Artin 
invariant $\sigma=1$ in characteristic two. The arguments rely  on properties of the Tutte--Coxeter graph; in this connection we also gain some insight into the arithmetic of 
quasi-elliptic Weierstrass equations and rational double points. The twisted cases ${}^2A_2$ and ${}^2C_2$ 
and ${}^2G_2$ yield particular ruled surfaces, attached in a canonical way to supersingular elliptic curves in 
characteristic two, or the Ree curve  in characteristic three. The latter is a curve of genus fifteen with 
outstanding   symmetries. In the former cases, the surfaces can also be expressed as symmetric squares. 
To obtain such results, we  develop a new general framework for Deligne--Lusztig varieties that relies on the Isogeny Theorem, and works over   prime fields rather than their algebraic closures, and includes the notorious 
Suzuki--Ree cases without effort.
\end{abstract}

\maketitle
\tableofcontents

\section*{Introduction}
\mylabel{sec:introduction}

In their groundbreaking work \cite{Deligne; Lusztig 1976}, Deligne and Lusztig introduced certain smooth subschemes $X(w)$ inside the flag variety $G/B$
of a reductive group $G$ over the   field $k=\FF_p^\alg$ in characteristic $p>0$, depending on   elements $w$
from the Weyl group $W=N_G(T)/T$. Roughly speaking, the scheme  $X(w)$ comprises   flags $gB$ that reside in  fixed relative position $w$
to the flag   $\varphi(g)B$. Here $\varphi\colon G\ra G$ is the  isogeny stemming from 
the Frobenius power $F^s$  for a chosen reductive group $G_0$ over a finite field $k_0=\FF_{p^s}$ that  induces $G$.
Note that the isogeny is purely inseparable, but in general not the identity on the underlying topological space.
The fixed   set $G^\varphi$ is a \emph{finite group of Lie type}, and the main goal of Deligne and Lusztig was to understand its
representation theory in terms of the $\ell$-adic cohomology with compact support of $X(w)$.
Indeed, the so-called \emph{Deligne--Lusztig varieties} $X(w)$ quickly became an indispensable tool  
for representation theory in this area, and received an enormous
echo, compare    \cite{Lusztig 1976}, \cite{Lusztig 1977}, \cite{Lusztig 1984},  \cite{Digne; Michel 1991},  
\cite{Bonnafe; Rouquier 2003}, \cite{Bonnafe; Dat; Rouquier 2017} and many others.

All the schemes  $X(w)$ are affine. This was  already  established  by Deligne and Lusztig under some assumption (\cite{Deligne; Lusztig 1976}, Theorem 9.7),
and   recently   in full generality  by Brosnan, Hong and Lee (\cite{Brosnan; Hong; Lee 2025}, Corollary 3.9).
In any case, these schemes come with a naive compactification $\overline{X(w)}\subset G/B$ inside the flag variety, which usually 
has singularities. However,   each reduced expression $w=s_{i_1}\ldots s_{i_n}$ yields a   resolution of singularities, which by  abuse 
of notation is   often written as $\bX(w)$. The situation is parallel to Schubert cells, Schubert varieties and their resolutions à la Demazure and Bott--Samelson.
For example, the length $n=\ell(w)$ coincides with $\dim\bX(w)$. 

Summing up, the group-theoretical datum   $G$ and $\varphi:G\ra G$ and $w\in W$  yields a smooth proper scheme $X=\bX(w)$, which we also call \emph{Deligne--Lusztig variety}.
In marked contrast to  representation theory, to date the impact  of these scheme in algebraic geometry is comparatively small.
Noteworthy exceptions are J.\ Hansen's analysis of Deligne--Lusztig curves \cite{Hansen 1992},  Rodier's study  in dimension two \cite{Rodier 2000}, 
S.\ Hansen's   Canonical Bundle Formula  \cite{Hansen 1999}, and closely related work by Ekedahl \cite{Ekedahl 2003} and Langer \cite{Langer 2019}
on  non-liftable Calabi--Yau threefolds.

In this paper, we  view the smooth proper schemes $X=\bX(w)$ as natural objects of   algebraic   geometry waiting to be explored, 
and initiate a systematic study of their properties. In particular, we seek to understand  their place
in the classification of all algebraic varieties,  
and to obtain  for particular Weyl group elements explicit  descriptions    in various other  contexts.
The following theorem, which is one of our main results, illustrates what we have in mind:
 
\begin{maintheorem}
(Compare Thm.\ \ref{thm:structure c2})
For the       symplectic group $G=\Sp_4$  over the prime field $k=\FF_2$  with  isogeny $\varphi=F$, the Frobenius morphism,
the Deligne--Lusztig surface $X=\bX(s_2s_1)$   is a supersingular K3 surface with Artin invariant $\sigma=1$.
\end{maintheorem}

Recall that the \emph{K3 surfaces} play a central role in the Enriques Classification of algebraic surfaces, and can be characterized in terms of the numerical invariants 
by  $c_1=0$ and $b_2=22$. Supersingularity means that the non-degenerate lattice $\Pic(X\otimes k^\alg)$ has rank $\rho=22$,
and the  \emph{Artin invariant} $\sigma\geq 1$ measures how far the 
intersection form   deviates from unimodularity. For $\sigma=1$ the surfaces are unique up to twisted forms.
In our theorem,  $X$ is actually given by the minimal resolution of singularities for the family of   cubic curves over $\PP^1=\Spec k[t]\cup\Spec k[t^{-1}]$
given by a Weierstra\ss{} equation $y^2=x^3+ a_6$ with coefficient $a_6=t^5+t^7$.

The main idea for the above result is to exploit properties of the famous \emph{Tutte--Coxeter graph}, a highly symmetric configuration of  30 vertices and 45 edges,
which here appears as the \emph{spherical building} of parabolic subgroups $P\subset G$, and also
as  \emph{dual graph} for the  projective lines  sitting as irreducible components inside the curve $\bX(s_1)\cup\bX(s_2)\subset X$. See Example~\ref{ex:sp4 tutte-coxeter} and Figure~\ref{fig:tutte-coxeter} in Section~\ref{subsec:arithmetic of rational double points}.
This configuration of curves is so special
that only the indicated K3 surface in characteristic two is possible.
Confer  \cite{Nikulin 1975}, \cite{Shioda; Inose 1977}, \cite{Dolgachev; Kondo 2003}, \cite{Kondo; Schroeer 2021} for  statements of similar nature. Along the way, we also gain some results on the arithmetic of quasi-elliptic Weierstra\ss{} equation and rational double points, which are of independent interest.
 
In the above, it is crucial to    work over the prime field $k=\FF_p$ rather than its algebraic closure $\FF_p^\alg$.
As in \cite{Schroeer 2023}, this allows to combine  algebraic and arithmetic arguments, often to  striking effects.
\emph{In this paper we put a particular emphasis on the observation that every   Deligne--Lusztig variety $\bX(w)$  
originates from root theoretic data, and therefore in particular lives over the prime field.}
This builds on the principle  that each root system $\Phi$ gives rise to a family of reductive algebraic groups over the ring $R=\ZZ$
of integers, the   \emph{Chevalley groups}. In some sense, the situation is also analogous to \emph{toric varieties},
which stem from the datum  $(N,\Sigma)$, where the second entry is a collection of certain cones $\sigma\subset N_\RR$. The  key notion in this direction is: 

\begin{maindefinition}
(compare Def.\ \ref{def:deligne-lusztig datum}) A \emph{Deligne--Lusztig datum} is a tuple
$$
\DLD = (X^*,\Phi,X_*,\Phi^\vee, \Delta, p, \varphi_*)
$$
where the first four entries form a   root datum,   $\Delta $ is a system of simple roots, $p>0$ is a prime number,
and $\varphi_*:X_*\ra X_*$ is a $p$-endomorphism of root data such that  $\varphi_*^r=p^s\cdot\id_{X_*}$ for some exponents $r,s\geq 1$.
\end{maindefinition}

Note that $X^*$ is a free abelian group of finite rank, and $\Phi$ is a root system in the classical sense inside the vector space generated by $\Phi$ inside $X^*\otimes\RR$.
Now let $k$ be the prime field $\FF_p$, or any extension thereof.
By the  various forms of the \emph{Isogeny Theorem}, the quadruple 
$(X^*,\Phi,X_*,\Phi^\vee)$ determines a \emph{split reductive group} $G$, with maximal torus $T$, Borel group $B$, and Weyl group $W=N_G(T)/T$.
The additional information $\Delta$ provides a \emph{pinning}, in other words, a compatible choice of root groups
$U_\alpha\subset G$, $\alpha\in\Phi$, together with an identification with $\GG_a$.
In turn, $G$ becomes a \emph{pinned reductive group},
and $\varphi_*:X_*\ra X_*$ corresponds to an isogeny $\varphi:G\ra G$ satisfying the equation $\varphi^r=F^s$. For all this we refer to \cite{SGA3-3}, see also Section~\ref{subsec:isogeny theorem}.
Obviously,     this kind of  datum is exactly what is needed to define the \emph{Deligne--Lusztig variety} $X(w)\subset G/B$,
but now living over the chosen field $k$. 
An additional advantage of our setting is that it effortlessly incorporates the so-called \emph{Suzuki--Ree cases},
where the isogeny  does not respect  short and long roots, as discussed in \cite{Deligne; Lusztig 1976}, Section 11.

As mentioned above, an important tool for our analysis in the $C_{2}$ case is the description of the boundary divisor $\overline{X(w)}\setminus X(w)$ in terms of the $\varphi$-stable Tits building, whenever $w$ is a twisted Coxeter element, see Proposition~\ref{prop:stratification and building}. In the twisted cases, it is crucial to relate the surfaces $\bX(s_{i_1}s_{i_{2}})$ to Deligne--Lusztig curves via Deligne--Lusztig reduction, see Section~\ref{subsec:cyclic shift} and in particular Corollary~\ref{cor:structure of ruled surface}.

Recall that in light of the minimal model program,  the place of $X=\bX(w)$ in the classification of algebraic varieties is largely governed by the position of the
\emph{dualizing sheaf} $\omega_X=\O_X(K_X)$ with respect to the \emph{nef cone} $\operatorname{Nef}(X)$  
inside the real vector space $\Num(X)_\RR$.
Understanding   this starts with  the  \emph{Canonical Bundle Formula} established by Hansen~\cite{Hansen 1999}  
$$
mK_{\bX(w)}=   \sum_{j=1}^n (\langle \mu, \tilde{\alpha}_j^\vee\rangle -1) D_j.
$$
Here $D_j\subset X$ are the boundary divisors obtained by dropping the $j$-th factor from the simple
expression $w=s_{i_1}\ldots s_{i_n}$. The precise  definition of the integer $m\geq 1$, 
the character $\mu\in X^*$ and the coroots $\tilde{\alpha}_j^\vee\in\Phi^\vee$ is   cumbersome, and we refer to Section \ref{sec:dualizing sheaves} for details.
\emph{A key ingredient of this paper   is that all this is expressible via
vectors and matrices, and therefore can be determined via computer algebra.} 
We have developed a computer program that  allows to compute in each dimension $n=\ell(w)$  and for all Deligne--Lusztig data $\DLD$
the coefficients  $\lambda_j=\langle \mu, \tilde{\alpha}_j^\vee\rangle -1$. Not surprisingly, these have a strong tendency to be positive,
and one expects that most $\bX(w)$ are of general type. Here we concentrate
on the opposite, and particularly interesting cases where all the coefficients satisfy $\lambda_j\leq 0$; we then say that \emph{the canonical divisor $K_X$ has negativity} (including here the case that all the $\lambda_{j}$ vanish).

It turns out that in dimension $n=2$, the relevant cases with negativity are  the root systems of type
$A_2, {}^2A_2, C_2, {}^2C_2$ and $ {}^2G_2$, with only few possible words $w=s_{i_1}s_{i_2}$ and primes $p>0$.
As customary, the upper indices indicate the presence of a non-trivial action of $\varphi_*:X_*\ra X_*$  on the  coroots.
Note that $C_2$ corresponds to the symplectic group,  and we already discussed that the untwisted case is related to K3 surfaces.
In the twisted cases, the Deligne--Lusztig surfaces have the structure of a ruled surface over a Deligne--Lusztig curve. If the canonical divisor has negativity, a typical case is that this curve is an elliptic curve $E$, and the surface is isomorphic, or closely related to its \emph{symmetric square}:

\begin{maintheorem}
(Compare Thm.\ \ref{thm:c2twist})
For the   symplectic group $G=\Sp_4$  over the prime field $k=\FF_2$  and the isogeny $\varphi$ with  $\varphi^2=F$,
we have 
$$
\bX(s_1s_2)=\Sym^2_{E/k}\quadand \bX(s_2s_1)=\Sym^2_{E/k}\times_EE^{(p)},
$$
where $E:y^2+y=x^3+x$ is the supersingular elliptic curve over $k$ with five rational points.
\end{maintheorem}
  
It was already observed by Atiyah \cite{Atiyah 1957} that these  quotients  are \emph{ruled surfaces}, cf.~Section~\ref{subsec:ruled and albanese}. In the above
situation, the rulings stem from the so-called \emph{Deligne--Lusztig reduction}, and the same goes for
the purely inseparable projection  for the fiber product.  Also note that the ruling has an intrinsic meaning
as \emph{Albanese map} \cite{Laurent; Schroeer 2024}. Explicitly,  one ruling 
stems from the non-split extension $0\ra\O_E\ra\shE\ra\O_E(e)\ra 0$. As discussed in
\cite{Togashi; Uehara 2022}, the Frobenius pullback  $F^*\shE$ remains indecomposable, and  sits in another non-split extension $0\ra\O_E(e)\ra F^*\shE\ra\O_E(e)\ra 0$,
The resulting sections correspond to  $\bX(s_1)$ and $\bX(s_2)$, contained  as curves in the surfaces.
Here  the \emph{Suzuki group}  $G^\varphi(\FF_2)=\Sz(2)$ is not simple. Rather, it can be viewed as the semidirect product $C_5\rtimes\Aut(C_5)$,
which coincides with $E(k)\rtimes\Aut(E)$, and actually forms the full automorphism group of the surface $X=\bX(s_1s_2)$.

We also take the occasion to give a succinct derivation of the explicit description of 
the       isogeny $\varphi:\Sp_4\ra\Sp_4$ that leads to the Suzuki groups $\Sz(2^s)$ 
in terms of algebraic geometry. The key point is a detailed study of the Plücker embedding, see Section~\ref{subsec:matrix interpretation}.

The root system ${}^2G_2$ is the so-called \emph{Ree case}. Here negativity occurs only for $w=s_2s_1$ and  $p=3$.
Our main result in this direction is:

\begin{maintheorem}
(Compare Thm.\ \ref{thm:structure g2twist})
For the   reductive group $G$ with root system $G_2$  over the prime field $k=\FF_3$  and the isogeny $\varphi$ with  $\varphi^2=F$,
we have $\bX(s_2s_1)=\PP(\shE)$ for the non-split extension  $0\ra\O_C\ra\shE\ra\omega^{\otimes-1}_C\ra 0$ over the curve
$$
C:\quad y^3-y = x(x^3-x) \quadand z^3-z = x(y^3-y).
$$
of genus $g=15$.
\end{maintheorem}

Note that $\Ext^1(\omega_C^{\otimes-1},\O_C)$ is indeed one-dimensional.
Again it is crucial to work over the prime field, and exploit the symmetries originating from the Ree group 
$\operatorname{Ree}(3)=G^\varphi(\FF_3)$, which is isomorphic to $\PGL_2(\FF_8)\rtimes\Aut(\FF_8)$. 
In light on the work on curves whose symmetries beat the Hurwitz bound
(\cite{Stichtenoth 1973}, \cite{Stichtenoth 1973b}, \cite{Henn 1978}, \cite{Hansen; Pedersen 1993}), 
the so-called \emph{Ree curve} $C$ is indeed the only possibility.

The paper is structured as follows:
In Section \ref{sec:deligne--lusztig varieties} we recall the notation of Deligne--Lusztig varieties,   develop our approach via the Isogeny Theorem, and collect some foundational facts on compactification, connected components,
projective bundles, and purely inseparable maps.
In Section \ref{sec:dualizing sheaves} we review Hansen's Canonical Bundle formula, discuss how to compute the canonical divisor via computer algebra, and introduce the notation of Deligne--Lusztig varieties $X=\bX(w)$ where the canonical divisor $K_X$ has negativity.
In Section \ref{sec:weil restriction} we briefly discuss Weil restrictions in the context of Deligne--Lusztig varieties, while Section \ref{sec:case a2} contains an examination of root datum of type $A_2$.
In Section \ref{sec:prerequisites} we collect some prerequisites concerning K3 surfaces, ruled surfaces and Albanese maps, and elliptic curves over the prime field $k=\FF_2$, which are freely used throughout.
Section \ref{sec:case c2} is devoted to the Deligne--Lusztig surface of type $C_2$, and the identification with the supersingular K3 surface with Artin invariant $\sigma=1$. In Section \ref{sec:case a2twist}
and \ref{sec:case c2twist} we examine the twisted cases ${}^2A_2$ and ${}^2C_2$, where the Deligne--Lusztig surfaces become symmetric powers of supersingular elliptic curves in characteristic $p=2$. In Section \ref{sec:case g2twist} we analyze the remaining case ${}^2G_2$, which yields the ruled surface canonically attached to the Ree curve in characteristic $p=3$, a curve of genus fifteen of exceptional symmetry.
The Appendix \ref{sec:appendix} contains further tables for the canonical divisor of Deligne--Lusztig surface.

\begin{acknowledgement}
This research was partially conducted in the framework of the   research training groups
\emph{GRK 2240: Algebro-geometric Methods in Algebra, Arithmetic and Topology} and 
\emph{GRK 2553: Symmetries and classifying spaces: analytic, arithmetic and derived}, which are funded
by the Deutsche Forschungsgemeinschaft.
\end{acknowledgement}

\section{Deligne--Lusztig varieties}
\mylabel{sec:deligne--lusztig varieties}

\subsection{Absolute and relative \texorpdfstring{$q$}{q}-Frobenius}
\mylabel{subsec:frobenius maps}
Let $k$ be a field of characteristic $p>0$. Given a $k$-scheme $X$, we 
write $F^s_X\colon X\ra X$ for the powers of the absolute Frobenius map. This is the identity
on the underlying topological space, and   the $\FF_p$-linear endomorphism $f\mapsto f^{p^s}$ on the structure sheaf.
We denote the corresponding $k$-morphisms
by the symbol $F^s_{X/k}\colon X\ra X^{(p^s)}$. Here $X^{(p^s)}=X\otimes_kk$ is the base-change with respect to   $F^s_k\colon k\ra k$. 
It is convenient to call $F^s_X$ the \emph{absolute $q$-Frobenius},
and $F^s_{X/k}$ the \emph{relative $q$-Frobenius}, for the prime power $q=p^s$.

Suppose now that $k$ contains the finite field $k_0=\FF_q$, for $q=p^s$,
and $X=X_0\otimes_{k_0}k$ for some $k_0$-scheme $X_0$.
Then the absolute $q$-Frobenius on $X$ is $k_0$-linear, and thus  decomposes as
$
F^s_X= (\id_{X_0}\otimes F^s_k)\circ (F^s_{X_0}\otimes \id_k).
$
The commutative diagram
$$
\begin{CD}
\Spec(k)	@>F^s>>	\Spec(k)\\
@VVV		@VVV\\
\Spec(k_0)	@>>\id>	\Spec(k_0)
\end{CD}
$$
gives an  identification $X=X_0\otimes_{k_0}k=X^{(p^s)}$, and we see that  the relative $q$-Frobenius 
$F^s_{X/k}$ on $X$ is identified with the base-change of the absolute $q$-Frobenius $F^s_{X_0}$ on $X_0$.

In the special case $k_0=\FF_q$ and $k=\FF_q^\alg$, the absolute $q$-Frobenius   $\sigma=F^s_k$
is a topological generator for the Galois group $\Gal(k/k_0)$, and we observe  that the relative $q$-Frobenius
$$
X_0(k)=X(k)\stackrel{F^s_{X/k}}{\lra} X(k)=X_0(k)
$$
coincides with the action of  $\sigma\in\Gal(k/k_0)$.
In turn, the set of rational points $X_0(k_0)$ can be seen as the closed points $x\in X$
that are fixed by the $q$-Frobenius endomorphism $F^s_{X/k}:X\ra X$.
If for some reason one can write $F^s_{X/k}=\varphi^r$  with another endomorphism $\varphi:X\ra X$ and some exponent $r\geq 0$,
one can also look at the closed points $x\in X$ that are  fixed by $\varphi$, and obtains a certain   subset
$$
X(k)^\varphi \subset X(k)^{F^s_{X/k}} = X_0(k_0)=X_0(\FF_{p^s}).
$$
If $X_0=G_0$ is a group scheme, and the endomorphism $\varphi$ respects the group law, we get certain groups
$G(k)^\varphi$. Indeed, the finite   groups of Lie type arise in this way.

\subsection{The Isogeny Theorem}
\mylabel{subsec:isogeny theorem}
In this section we summarize the theory of pinned reductive groups, on which our framework for Deligne--Lusztig varieties relies.
The main source is \cite{SGA3-3}, Exposés~XXIII and XXV. Further references are \cite{Conrad; Gabber; Prasad 2015} Appendix~A.4, \cite{Milne 2017} Ch.~23, \cite{Steinberg 1999}.
For our purposes it suffices to work over a ground field $k$ rather than a general base scheme, and we write $p\geq 1$ for its characteristic exponent.

Recall that a \emph{split reductive group} is a pair $(G,T)$ comprising a linear algebraic group $G$ that is smooth, connected and reductive,
together with a split maximal torus $T\subset G$.
The conjugation action by $T$ on the Lie algebra $\lieg=\Lie(G)$ corresponds to a weight decomposition
$\lieg=\bigoplus \lieg_\alpha$ indexed by the character lattice $X^*=\Hom(T,\GG_m)$
and defines a root system
$$
\Phi=\{\alpha\in X^*\setminus \left\{ 0 \right\} \mid \lieg_\alpha\neq 0\}
$$
inside the real vector space $\RR\Phi\subset X^*\otimes\RR$.
The Borel groups $B\subset G$ containing $T$ correspond to the systems of positive roots, 
via $\Phi_+=\{\alpha\in\Phi\mid \lieg_\alpha\subset\Lie(B)\}$, or equivalently to the systems of simple roots $\Delta$.

The group scheme $N_G(T)/T$ is a constant finite group scheme, and its action by conjugation on the character lattice identifies it
with the Weyl group $W\subset \GL(\RR\Phi)$ of the root system. 
For each root $\alpha\in \Phi$, the eigenspace $\lieg_\alpha$ is one-dimensional, and there is a unique
$U_\alpha\subset G$ that is  isomorphic to the additive group, is normalized by the torus, and has  $\Lie(U_\alpha)=\lieg_\alpha$.
These $U_\alpha$ are called \emph{root groups}. The choice of a Borel group $B$ containing the torus,
together with isomorphisms $\psi_\alpha\colon\GG_a\ra U_\alpha$, $\alpha\in\Delta$, is called a \emph{pinning}
for the split reductive group $(G,T)$. 

A \emph{pinned reductive group} is  a quadruple  $(G,T,B,\psi)$ where $(G,T)$ is a split reductive group, together with a pinning
provided by $B$ and $\psi=(\psi_\alpha)$ as above. The pinned reductive groups form a category:
The morphisms
$$
f\colon(G,T,B,\psi)\lra (G',T',B',\psi')
$$
are homomorphisms $f\colon G\ra G'$ of group schemes  with $f(T)\subset T'$, such that there exist a bijection 
$d\colon \Phi\ra \Phi'$   with $d(\Delta)=\Delta'$   and  a map $q\colon \Phi\ra  p^\NN$ such that
$$
f\circ\psi_\alpha=\psi'_{d(\alpha)}\quadand f^*(d(\alpha))=q(\alpha)\alpha\quadand f_*(\alpha^\vee)=q(\alpha)d(\alpha)^\vee,
$$
for all roots $\alpha\in\Phi$. Here $f^*\colon X'^*\ra X^*$ and $f_*\colon X_*\ra X'_*$ are the induced maps
on   characters and co-characters of the fixed tori. Note that both maps $d$ and $q$   are uniquely determined by $f$,
and that furthermore $f(B)\subset B'$.

It turns out that the category of pinned reductive group can be described entirely in  terms of lattices, root systems,
and characteristic exponents, and otherwise does not depend on the ground field $k$.
For this we recall that a \emph{root datum} is a tuple 
$(X^*,\Phi,X_*,\Phi^\vee)$
where $X^*$ is a finitely generated free $\ZZ$-module, $X_*=\Hom(X^*,\ZZ)$ is the dual lattice, 
$\Phi\subset X^*$ and $\Phi^\vee\subset X_*$ are finite sets  endowed with a bijection  
 $\alpha\mapsto \alpha^\vee$, such that with respect to the evaluation pairing we have $\langle \alpha,\alpha^\vee\rangle=2$,
and that the ensuing reflections
$s_\alpha(x)=x-\langle x, \alpha^\vee\rangle \alpha$ and $s_{\alpha^\vee}(y)=y-\langle y,\alpha\rangle \alpha^\vee$
stabilize $\Phi$ and $\Phi^\vee$, respectively. 
A root datum is called reduced, if in addition $\alpha\in\Phi$ implies $2\alpha\not\in\Phi$.
Note that the bijection $\alpha\mapsto \alpha^\vee$ is unique, and
$\Phi$  is a reduced root system in the vector subspace $\RR\Phi\subset X^*\otimes\RR$.
The latter  inclusion might be strict, and $\Phi$ actually could be empty. 

While the definition of morphism of pinned reductive groups may seem overly complicated and even a bit arbitrary, it is justified by the following observations.
    Morphisms of pinned groups with $q(\alpha)\ne 1$ and even with $q(\alpha)$ depending on $\alpha$ do exist in positive characteristic; see Theorem~\ref{thm:isogeny theorem} below. On the other hand, \cite{SGA3-3} Proposition~4.2.12 shows that every isogeny $(G, T)\to (G', T')$ of split reductive groups is of the above form, for suitable choices of pinnings.

A \emph{pinned root datum} is a tuple $(X^*,\Phi,X_*,\Phi^\vee,\Delta)$
consisting of a root datum and a  system $\Delta\subset\Phi$ of simple roots.
To obtain a category, we  additionally fix 
an integer $p\geq 1$ that is either   prime or is equal to $1$; it is convenient to call
such numbers \emph{characteristic exponents}. Then a  \emph{$p$-morphism}
\[
h:(X^*,\Phi,X_*,\Phi^\vee,\Delta)\lra (X'^*,\Phi',X'_*,\Phi'^\vee,\Delta')
\]
of pinned root data is a homomorphism $h_*:X_*\ra X'_*$ of abelian groups such that there exist a bijection $d:\Phi\ra\Phi'$
with $d(\Delta)=\Delta'$   and a  map $q:\Phi\ra p^\NN$ such that
\begin{equation}\label{condition isogeny}
h^*(d(\alpha))=q(\alpha)\alpha\quadand h_*(\alpha^\vee)=q(\alpha)d(\alpha)^\vee
\end{equation}
for all roots $\alpha\in\Phi$. Here $h^*\colon X'^*\ra X^*$ is the dual of the map  $h_*$. 
With this definition, with respect to a chosen characteristic exponent $p\geq 1$,  the pinned root data form a category.

The pinned reductive groups and the pinned root data are related by a functor: 
Given   $(G,T,B,\psi)$ we obtain a pinned root datum $(X^*,\Phi,X_*,\Phi^\vee,\Delta)$ as explained above. From the definition of Hom sets in  the two categories, it is clear that this extends to a functor.
Now \cite{SGA3-3}, Exposé XXV, Théorème 1.1 that is usually referred to as the \emph{isogeny theorem} takes the form:

\begin{theorem}\mylabel{thm:isogeny theorem}
Let $k$ be a field with characteristic exponent $p\geq 1$. Then the above functor is an equivalence between
the category of  pinned reductive groups over $k$, and the category of pinned reduced root data and $p$-morphisms.
\end{theorem}

An important consequence of this is that  for each field extension $k\subset k'$, the ensuing base change $G\mapsto G\otimes k'$ induces
an equivalence between the categories of pinned reductive groups over $k$ and $k'$. In particular, each pinned reductive group over $k$
descends to  the prime field of $k$.

Recall that up to isomorphism, the irreducible reduced root systems $\Phi$ correspond to the Dynkin diagrams $A_\ell$, $B_\ell$, \ldots, $F_4$, $G_2$, and this symbol is called the \emph{type} of   $G$.

If the inclusion $\RR\Phi\subset X^*\otimes\RR$ is an equality, then the reductive group $G$ is called \emph{semisimple}. If $\Phi\subset X^*$ or $\Phi^\vee\subset X_*$ form a set of generators, $G$ is called \emph{of adjoint type} or   \emph{simply-connected}, respectively. 

From now on, we assume that we are in characteristic $p>0$. Let $(X^*,X_*,\Phi,\Phi^\vee,\Delta)$ be a pinned reduced root datum,
and  $(G,T,B,\psi)$ be the corresponding pinned reductive group over the prime field $k=\FF_p$.
First note that the multiplication maps $p^s\colon X_*\ra X_*$ correspond to the  iterated Frobenius maps $F^s\colon G\ra G^{(p^s)}=G$.
Suppose now that $\varphi_*:X_*\ra X_*$ is an endomorphism of the pinned root datum with the property
$\varphi_*^r= p^s\cdot\id_{X^*}$ for some exponents $r,s\geq 1$.  The corresponding isogeny $\varphi: G\ra G$
of the pinned reductive group has the property $\varphi^r=F^s$.
Write $G^\varphi\subset G$ for the scheme of fixed points, which is a subgroup scheme.

\begin{proposition}
\mylabel{prop:structure Gphi}
The group scheme $G^\varphi$ is finite and \'etale, and its base-change to $\FF_{p^s}$ is constant.
\end{proposition}

\proof
By descent theory, it is sufficient to check the final statement. But the fix point scheme of $F^s$ in $G_{\FF_{p^s}}$ is simply the finite closed subscheme of points with residue class field $\FF_{p^s}$, and $(G^\varphi)_{\FF_{p^s}}$ is its subscheme consisting of those points that are fixed by $\varphi$.
\qed

\medskip
The   groups $G^\varphi(\FF_{p^s})$ are called \emph{finite groups of Lie type}.
From a more traditional point of view, they can  also be seen as the group of closed points 
inside $G\otimes k^\alg$ that are fixed by the endomorphism $\varphi\otimes k^\alg$. All such points are fixed by 
$F^s\otimes k^\alg$, and are therefore defined over the subfield $\FF_{p^s}\subset k^\alg$.
For us, the finite \'etale group scheme $G^\varphi$ is the primordial object;
it would be interesting to study the significance of its scheme structure in more detail.

By definition, the morphism $\varphi_*$ induces a  permutation $d\colon \Delta\ra\Delta$.
Its order equals the smallest $r\geq 1$ with $\varphi_*^r=p^s$.
Since the permutation respects the edges in the Dynkin diagram (including their multiplicity) the only possibilities are   $1\leq r\leq 3$. 
As is common, we extend the Dynkin type notation to denote not only the group $G$, but the pair $(G, \varphi)$. Then the undecorated symbols $A_\ell$, $B_\ell$, etc.~denote the case $r=1$, i.e., that $\varphi$ is the relative $q$-Frobenius.
The \emph{twisted cases} $r=2$ and $r=3$ are  indicated by upper indices:
$$
{}^2A_\ell,\quad {}^2B_2={}^2C_2, \quad {}^2D_\ell,\quad {}^2E_6,\quad    {}^2F_4,\quad  {}^2G_2\quadand {}^3D_4.
$$
Note that in  the \emph{Suzuki--Ree cases}  ${}^2C_2$, ${}^2F_4$ and ${}^2G_2$, 
the permutation  does not preserve root lengths.
In any case, for a fixed permutation $d\colon \Delta\ra\Delta$, the possible endomorphisms $\varphi_*$ 
must satisfy \eqref{condition isogeny}. So with respect to the basis $\alpha_i^\vee\in X_*$, the map
is described by a monomial matrix of the form $M=(p^{s_i}\delta_{i,d(j)})_{i,j}$,
and its transpose has the property $\alpha_{d(j)}\mapsto p^{s_j}\alpha_j$.
Using that the coordinates of the simple roots with respect to the dual basis of the simple co-roots
form the rows of the Cartan matrix $C=(\langle \alpha_i,\alpha_j^\vee\rangle)$, one obtains
a classification of the possible $\varphi_*$.

\begin{example}
Consider, for example, the case ${}^2F_4$: The root system $\Phi$ lies in the
vector space $V=\RR^4$, with standard orthogonal basis $\epsilon_1,\ldots,\epsilon_4$ and   simple roots
$$
\alpha_1=\epsilon_2-\epsilon_3,\quad \alpha_2=\epsilon_3-\epsilon_4,\quad \alpha_3=\epsilon_4,\quad 
\alpha_4=\frac{1}{2}(\epsilon_1-\epsilon_2-\epsilon_3-\epsilon_4).
$$
The simple coroots are $\alpha_1^\vee=\alpha_1$, $\alpha_2^\vee=\alpha_2$, $\alpha_3^\vee=2\alpha_3$,  $\alpha_4^\vee=2\alpha_4$,
the permutation is $d= (14)(23)$, and the matrices $M$ and $C$ take the form
$$
M=\begin{pmatrix}
	&	&	& p^{s_4}\\
	&	& p^{s_3}\\
	& p^{s_2}\\
p^{s_1}\\
\end{pmatrix}
\quadand
C= \begin{pmatrix}
2	& -1\\
-1	& 2	& -2\\
	& -1	& 2	& -1\\
	&	& -1	& 2
\end{pmatrix}.
$$
The condition $\varphi^*(\alpha_2)=p^{s_3}\alpha_3$ gives $p=2$ and $s_4=s_3=s_2+1$,
while $\varphi^*(\alpha_4)=p^{s_1}\alpha_1$ means $s_1=s_2$.
Summing up, we have $p=2$ and $r=2$ and $s=2n+1$ for some $n\geq 0$, and the whole situation, with 
group scheme  $G^\varphi$ and finite group of Lie type $G^\varphi(\FF_{p^s})$,
is indicated by the symbol ${}^2F_4(q^2)$, with real number $q=2^{(2n+1)/2}$.
The situation for the other Suzuki--Ree cases is similar. We will look at the case ${^2}B_2$ in detail in Section~\ref{sec:case c2twist}, and at the case ${^2}G_{2}$ in Section~\ref{sec:case g2twist}.

We add the following table about the notation which unfortunately is not entirely consistent in the existing literature.

    \begin{center}
        {\renewcommand{\arraystretch}{1.3}
            \begin{tabular}{l|c|c}
                This paper & Hansen \cite{Hansen 1992} & Carter \cite{Carter 1972} \\\hline
                $n$ & $n$ & $n$ \\\hline
                $q = \sqrt{p^{2n+1}}$ & $Q$ & $q$ \\\hline
                $q_0 = p^n$ & $-$ & $-$  \\\hline
                $q^2$       & $q$ & $-$
            \end{tabular}
        }
    \end{center}
\end{example}

\subsection{The relative position map}
\mylabel{subsec:relative position}
Let us recall from~\cite{Deligne; Lusztig 1976} Section~1.2 the relative position map.
Let $k$ be a field, and let $(G, T)$ be a split reductive group over $k$ with a fixed Borel group $B \supset T$.

Denote by $W$ the Weyl group of $G$ with respect to $T$, i.e., $W=(N_G(T)/T)(k)$.
The choice of $B$ determines a system $\bbS = \{s_1,\dots, s_n\}$ of simple reflections which generate the Coxeter group $W$.
The choice of generators of $W$ determines the length function $\ell\colon W\to \ZZ_{\ge 0}$ and the Bruhat order.

For any fixed Weyl group element $w\in W$, the subscheme
\[
    \left\{ (g_1, g_2)\in G\times G;\ g_1^{-1} g_2 \in BwB \right\} 
\]
is invariant under the action of $B\times B$ by multiplication on the right. It is the inverse image of a subscheme $\catO(w) \subset G/B\times G/B$ under the natural projection. We call $\catO(w)$ the subscheme of pairs \emph{in relative position $w$}.

We may equivalently describe the scheme $\catO(w)$ as follows.
We have a natural morphism
\[
    \inv\colon G/B \times G/B \to [G\backslash (G/B\times G/B)] \xrightarrow{\cong} [B\backslash G/B],
\]
where $G$ acts diagonally on $G/B\times G/B$, and the isomorphism $[G\backslash (G/B\times G/B)] = [B\backslash G/B]$ is induced by $(g_1, g_2)\mapsto g_1^{-1}g_2$. The stack quotient $[B\backslash G/B]$ has corresponding topological space $W$ (with the topology induced from the Bruhat order). With this terminology, $\catO(w) \subset G/B \times G/B$ is the fiber of the morphism $\inv$ over the point $w\in W = [B\backslash G/B]$.

The closure $\overline{\catO(w)}$ is the union of all $\catO(v)$ for $v\le w$.

In particular, for every field $k$ we obtain the \emph{relative position map}
\[
    \inv\colon (G/B)(k) \times (G/B)(k) \to W.
\]
Explicitly, $\inv(g_1, g_2)=w$ is equivalent to $g_1^{-1}g_2\in BwB$. See~\cite{Deligne; Lusztig 1976} Section 1.2. There it is also explained that the above can be carried out ``without choosing a base point $B$'' of the variety of all Borel subgroups of $G$.

Furthermore (see loc.~cit.), for $w, w_1, w_2\in W$ with $w = w_1 w_2$ and $\ell(w_1) + \ell(w_2) = \ell(w)$, we have an isomorphism
\begin{equation}\label{eq: decomposition relative position}
    \catO(w_1)\times_{G/B}\catO(w_2)\isomarrow \catO(w),\qquad
    ((gB, g'B), (g'B, g''B))\mapsto (gB, g''B).
\end{equation}

For specific classical groups, the relative position can be expressed in terms of flags (or symplectic flags, etc.), see below.

Later we will need the following lemma on the closure $\overline{\catO(s)}$ of $\catO(s)$ in $G/B \times G/B$.

\begin{lemma}\mylabel{lem:structure barOs}
    Let $s\in \bbS$ be a simple reflection. The composition
    \[
        \overline{\catO(s)} \subset G/B\times G/B\xrightarrow{{\rm pr}_2} G/B
    \]
    is a Zariski-locally trivial $\PP^1$-bundle.
\end{lemma}

\begin{proof}
    Denote by $U \subset B$ the unipotent radical and by $w_0\in W$ the longest element. Then
    \[
        U \xrightarrow{\cong} U w_0 B/B \subset G/B
    \]
    is an open immersion, and $G/B$ is covered by translates of the \emph{open cell} $U w_0 B/B$. Thus it is enough to give an isomorphism $\overline{\catO(s)} \cap (G/B\times Uw_0B/B)\cong  \overline{BsB/B} \times Uw_0B/B$ of schemes over $Uw_0B/B$.

    Given $(g_1, g_2)\in \overline{\catO(s)} \cap (G/B\times Uw_0B/B)$ we may write $g_2 = uw_0 B/B$ for a unique $u\in U$, and then $((uw_0)^{-1}g_{1}, g_{2}) \in \overline{BsB/B} \times Uw_0B/B$. Conversely, we map $(h_{1}, h_{2})\in \overline{BsB/B} \times Uw_0B/B$ to $(h_{2} h_{1}, h_{2}) \in \overline{\catO(s)} \cap (G/B\times Uw_0B/B)$.
\end{proof}

\subsection{Relative position for type \texorpdfstring{$A$}{A}}
\mylabel{subsec:relpos A}
Let $k$ be a field. Denote by $(e_\nu)_{\nu=1, \dots, n}$ the standard basis of $k^n$.
For $G=\GL_n$ with Borel subgroup $B$ the subgroup of upper triangular matrices, we may identify the Weyl group $W$ for the diagonal maximal torus with the subgroup of permutation matrices of $G$, and view $G/B$ as the variety of full flags in $k^n$. It follows from the Bruhat decomposition that for flags $(F_i)$, $(G_i)$, the relative position $\inv((F_i), (G_i))$ is the unique element $w\in W = S_n$ such that
\[
    \dim(F_i\cap G_j) = \dim\left( \langle e_1,\dots, e_i\rangle\cap \langle e_{w(1)},\dots, e_{w(j)}\rangle \right).
\]
In fact, the numbers $\dim(F_i\cap G_j)$ are clearly constant on each $\mathcal O(w)$, and $w$ is uniquely determined by all of them taken together.
Let us write
\[
    w[i, j] = \dim\left( \langle e_1,\dots, e_i\rangle\cap \langle e_{w(1)},\dots, e_{w(j)}\rangle \right) = \# \{ \nu;\ 1\le \nu\le j,\ w(\nu)\le i \}.
\]
One checks in combinatorial terms that in terms of these numbers the Bruhat order is given by componentwise comparison (\cite{Bjoerner; Brenti 2005} Theorem 2.1.5; note that their normalization is slightly different), i.e.,
\[
    \inv((F_i), (G_i)) \le w \quad\Longleftrightarrow\quad
    \dim(F_i\cap G_j) \ge w[i,j]\ \text{for all}\ i, j.
\]

For $n=3$, we have $W = S_3 = \langle s_1, s_2\rangle$ and the possible relative positions have the following explicit description.

\begin{lemma}\label{lem:relative position explicit}
    Let $(L\subset U)$, $(L'\subset U')$ be flags in $k^3$. The following table gives, for each $w$ of length $\le 2$, the condition that $\inv((L\subset U), (L'\subset U')) \le w$.

\[
\begin{array}{l|*{5}{l}}
\toprule
w\phantom{aa}	& e	& s_1	& s_2	& s_1s_2		& s_2s_1		\\[.1cm]
\toprule
& L=L',U=U'\phantom{aaa}	& U=U'\phantom{aaa}	& L=L'\phantom{aaa}	& L'\subset U\phantom{aaa}	& L\subset U'	\\
\bottomrule
\end{array}
\]
\medskip
\end{lemma}

\begin{proof}
    This follows from elementary linear algebra considerations. For example, for $w = s_1s_2$, we have $w[2, 1] = 1$, so $\dim L'\cap U \ge 1$, or equivalently $L' \subset U$; all the other dimension estimates are then automatic.
\end{proof}

\subsection{Relative position for type \texorpdfstring{$C$}{C}.}
\mylabel{subsec:relpos C}
Now we consider type $C_n$ and let $V = k^{2n}$ be the symplectic vector space with the alternating form given by the matrix 
\[
    \begin{pmatrix}
    & & & &   & -1 \\
    & &  & & \iddots \\
    & &  & -1 \\
    & & 1 & &  \\
     & \iddots &\\
    1 & &
\end{pmatrix},
\]
(entries $=0$ omitted).
Let $G = {\Sp}_{2n}$ be the symplectic group of $V$, $B$ its Borel subgroup of upper triangular matrices, $T$ the diagonal torus.

The group $G$ is a subgroup of $\GL_{2n}$, and the ``symplectic full flag variety'' $\mathscr F:=G/B$ can be described as a subvariety of the flag variety of $\GL_{2n}$, namely
\[
    \mathscr F = \{ (F_i)_{i=1, \dots, 2n-1};\ \dim(F_i) = i,\ F_{2n-i} = F_i^\perp \}
\]
consists of all symplectic flags. In particular $F_n = F_n^\perp$, i.e., $F_n$ is maximally isotropic, and $F_1, \dots, F_n$ determine all the subspaces of the flag.

The embedding $G \subset \GL_{2n}$ induces an embedding of the Weyl group $W$ of $G$ into the symmetric group $S_{2n}$. Its image consists of those permutations $w$ which satisfy $w(i) + w(2n-i+1) = 2n+1$ for all $i$. The system of simple reflections  determined by $B$ is given, in terms of elements of $S_{2n}$, by
\[
    s_1 = (12)(2n-1,2n),\quad
    \dots,\quad
    s_{n-1} = (n-1, n)(n+1, n+2),\quad
    s_n = (n, n+1).
\]
Similarly as for $\GL_{2n}$, the relative position of flags can be described in terms of the dimensions of intersections of the members of the flags. Furthermore, this description extends to descriptions of the closures $\overline{\mathcal O}(w)$ in the same fashion, i.e., we have the following result.

\begin{proposition}[\cite{Bjoerner; Brenti 2005} Theorem 8.1.8]\mylabel{prop:c2 rel pos}
    Let $(F_i)$, $(G_i)$ be symplectic flags, and $w\in W$. Define $w[i,j]$ for $i,j=1,\dots, 2n$ as before (viewing $w\in S_{2n}$). Then
    \[
        \inv((F_i), (G_i)) \le w \quad\Longleftrightarrow\quad
        \dim(F_i\cap G_j) \ge w[i,j]\ \text{for all}\ i=1, \dots, 2n,\ j=1, \dots, n.
    \]
\end{proposition}

Note that because of the duality condition on symplectic flags here it is permissible to consider only $j=1, \dots, n$ in the condition on the right hand side.
This result implies in particular that $v \le w$ if and only if the images of $v$ and $w$ in $S_{2n}$ are related in the same way by the Bruhat order on $S_{2n}$.

The Weyl group $W$ is isomorphic to the semi-direct product $S_n \ltimes (\ZZ/2)^n$, where $S_n$ acts on $(\ZZ/2)^n$ by permutations in the natural way. To obtain this identification, note that the simple reflections $s_1, \dots, s_{n-1}$ generate a subgroup of $W$ isomorphic to $S_n$. The simple reflection $s_g$ corresponds to the vector $(0, \dots, 0, 1)\in (\ZZ/2)^n$, and more generally the vector $(0, \dots, 0, 1, 0, \dots, 0)$ with the $1$ in the $i$-th place identifies with the transposition $(i, 2n-i+1) \in W \subset S_{2n}$.

For computing with elements of $W$, it is often useful to view them as \emph{signed permutations}, cf.~\cite{Bjoerner; Brenti 2005} Section 8.1. This means that we write an element $(v, t)\in S_n \ltimes (\ZZ/2)^n$ as the tuple $(v(1), \dots, v(n))$ where in addition we put a bar over $v(i)$ if the vector $t$ has $i$-th entry equal to $1$. Said differently, every element of $W$ permutes the lines generated by the standard basis vectors of $V$. We label these standard basis vectors by the symbols $1, 2, \dots, n, \bar{n}, \dots, \bar{1}$. Then each $w$ is represented by its effect on the first $n$ basis vectors.

The case $n=2$ is particularly relevant for us later. In this case a symplectic flag
\[
    0 \subset L \subset U \subset L^\perp \subset V
\]
is determined by the line $L$ and the totally isotropic plane $U=U^\perp$, and we usually just write $L \subset U$ to denote the flag.

\begin{lemma} 
\mylabel{lem:position symplectic flags}
    Let $(L\subset U)$, $(L'\subset U')$ be symplectic flags in $k^4$. The following table gives, for the Weyl group elements $w\in W$ of length $\le 2$, the condition that $\inv((L\subset U), (L'\subset U')) \le w$.
$$
\begin{array}{l|*{5}{l}}
\toprule
w\phantom{aaa}	& e	& s_1	& s_2	& s_1s_2		& s_2s_1		\\[.1cm]
\toprule
& L=L',U=U'\phantom{aaa}	& U=U'\phantom{aaa}	& L=L'\phantom{aaa}	& L'\subset U\phantom{aaa}	& L\subset U'	\\
\bottomrule
\end{array}
$$ 
\end{lemma}

\proof
This is now a simple computation. To illustrate it, we consider the case $w = s_1 s_2$ which as a signed permutation is $(2, \bar{1})$ and as an element of $S_{2n}$ is the cycle $(1243)$. We see that $w[2, 1] = 1$, whence $L' \subset U$.
Conversely, if $L' \subset U$, then all the conditions from Proposition~\ref{prop:c2 rel pos} are satisfied, for example, $w[2,2] = 1$ says $\dim(U\cap U') \ge 1$ which follows from $L' \subset U$.
\qed

\subsection{Deligne--Lusztig data}
\mylabel{subsec:deligne-lusztig data}
In this section we introduce our set-up   for   Deligne--Lusztig varieties and their   compactifications.
We somewhat deviate   from the original definition in  \cite{Deligne; Lusztig 1976}, in order to obtain schemes defined over   prime fields,
and also to treat the Suzuki--Ree cases on exactly the same footing. 

\begin{definition}
\mylabel{def:deligne-lusztig datum}
A \emph{Deligne--Lusztig datum} is a tuple
$$
\DLD = (X^*,\Phi,X_*,\Phi^\vee, \Delta, p, \varphi_*)
$$
where $R=(X^*,\Phi,X_*,\Phi^\vee, \Delta)$ is a pinned root datum, $p>0$ is a prime number, and
$\varphi_*:R\ra R$ is an   endomorphism with respect to $p$, such that for some integers $r,s\geq 1$ we have   $\varphi_*^r=p^s\cdot\id_{X_*}$.
\end{definition}

The induced endomorphism of $X_*\otimes\QQ$ is bijective, and one 
easily checks that the pairs $(r,s)$ satisfying $\varphi_*^r=p^s\cdot\id_{X_*}$ form a rank-one subgroup  inside $\ZZ^2$.
The $r,s\geq 1$ whose pair generates this subgroup are called \emph{minimal exponents} of the Deligne--Lusztig datum.
Note that these are not necessarily coprime, and that the fraction $s/r\geq 1$ might be non-integral.

Let $\DLD = (X^*,\Phi,X_*,\Phi^\vee, \Delta, p, \varphi_*)$ be a Deligne--Lusztig datum. Fix a ground field $k$ of the given characteristic $p>0$,
and form the resulting
pinned reductive group $G$, with maximal torus $T$, Borel group $B$, Weyl group  $W=N_G(T)/T$, and isogeny $\varphi:G\ra G$, cf.~Section~\ref{subsec:isogeny theorem}.

\begin{definition}
\mylabel{def:deligne-lusztig variety}
Let $w\in W$.
The  subscheme $X(w)\subset G/B$ whose $R$-valued points are given by
$$
\{gB\mid g^{-1}\varphi(g)\in BwB\}\subset (G/B)(R)
$$
is called the \emph{Deligne--Lusztig variety}  over the ground field $k$ attached to the Deligne--Lusztig datum
$\DLD$ and the Weyl group element $w$.
\end{definition}

In other words, the cosets $gB$ and $\varphi(g)B$ are in relative position $w\in W$.
This  condition    is meant to hold for 
each faithfully flat ring extension $R\subset R'$ where  $gB\in (G/B)(R)$  stems from an element $g\in G(R')$. Equivalently, we can phrase this as $\inv(gB, \varphi(g)B)=w$, where $\inv$ denotes the relative position map, see Section~\ref{subsec:relative position}.
One easily sees that the subfunctor is   representable by a subscheme, which in particular defines a locally closed
set in $G/B$.
Note that we may use the prime field $k=\FF_p$, and that
the formation of $X(w)$ commutes with ground field extensions $k\subset k'$.
The following   basic facts are  well-known; we include a sketch of proof for the convenience of the reader.

\begin{lemma}
\mylabel{lem:properties deligne-lusztig variety}
The subscheme $X(w)\subset G/B$ is smooth and equidimensional, of dimension $\ell(w)\geq 0$.
Moreover, its  
schematic closure is normal and Cohen--Macaulay, with at most rational singularities,
and given by   $\overline{X(w)} =  \bigcup_{v\le w} X(v)$.
\end{lemma}

\proof
Recall that the Lang map $L\colon G\ra G$ $g\mapsto g^{-1}\varphi(g)$ is surjective  (\cite{Lang 1956}, Section 2 and   \cite{Steinberg 1968}, Theorem 10.1),
and actually \'etale (because it is injective on tangent vectors).
The preimage of the subscheme $BwB\subset G$  
is stable under multiplication by $B$ on the right, and we have $X(w)=L^{-1}(BwB)/B$. 
It comes with a correspondence
$$
X(w) \longleftarrow L^{-1}(BwB) \longrightarrow BwB / B,
$$
given by $g\mapsto gB$ and $g\mapsto L(g)B$, respectively. 
The arrow on the left  is a $B$-torsor, hence smooth of relative dimension $d=\dim(B)$.
The arrow on the right  is smooth of the same relative dimension, because $L$ is \'etale. Since the Schubert cell
$BwB/B$ is smooth of dimension $\ell(w)$, the same holds for the Deligne--Lusztig variety $X(w)$.
Taking schematic closures in $G/B$ and $G$,   we likewise obtain a correspondence 
$\overline{X(w)} \leftarrow\overline{L^{-1}(BwB)} \ra \overline{BwB / B}$
with smooth morphisms. Since the closure of the Schubert variety is normal, Cohen-Macaulay and has rational singularities, the same holds
for the closure of the Deligne--Lusztig variety.
The correspondence also shows that for $v, w\in W$ we have $X(v)\subseteq \overline{X(w)}$ if and only if 
$BvB/B\subseteq \overline{BwB/B}$, which holds if and only if $v\le w$.
\qed

\medskip
In turn, we have a stratification    $G/B=\bigcup_{w\in W} X(w)$ of the flag variety into locally closed 
sets. We call it the \emph{Deligne--Lusztig stratification}. In view of the lemma, the  
\emph{closure relation} $X(v)\subset \overline{X(w)}$ corresponds to the Bruhat order $v\leq w$ on the Weyl group.
The closed
stratum equals $X(e)=(G/B)^\varphi$, while the 
open stratum is given by $X(w_0)$, where $w_0\in W$ is the longest element.
The group scheme $G^\varphi$ acts by left translations on $G/B$, and respects the stratification.
Note that the center $Z(G)$ acts trivially on the flag variety.

\begin{proposition}
\mylabel{prop:kernel of action}
Let $w\in W$.
The kernel for $G^\varphi\ra \Aut_{X(w)/k}$ equals $G^\varphi\cap Z(G)$.
\end{proposition}

\proof
We may assume that $G$ is semisimple of adjoint type, and decomposing it into a product and handling the individual factors one by one may furthermore assume that it is simple, i.e., that the Dynkin diagram is connected. Without loss of generality we may also assume that $k$ is algebraically closed.
Now first observe that every element $g\in G(k)$ that acts trivially on $X(w)$ also acts trivially on its closure and in particular on $X(\id)$. It is therefore sufficient to consider the case $w=\id$. Saying that $g$ acts trivially on $X(\id)$ is equivalent to saying that $g$ is contained in the intersection $\bigcap_{h\in G^\varphi} hBh^{-1}$ of conjugates of $B$.

In particular, the subgroup of $G^\varphi$ consisting of elements that act trivially on $X(\id)$ is a normal subgroup which is contained in $B^\varphi$. It follows from \cite{Malle; Testerman 2011}, Lemma 24.12 that this subgroup is contained in the center of $G^\varphi$ which according to loc.~cit., Corollary 24.13 is equal to $G^\varphi\cap Z(G)$. The key point here is that $G^\varphi$ carries a BN-pair structure.
\qed

\begin{remark}
\mylabel{rem:adjoint type and simply-connected}
Passing from the reductive group $G$ to its adjoint group $G^{\rm ad}=G/Z(G)$ leaves 
the flag variety, Weyl group, and Deligne--Lusztig varieties unchanged. In forming the $X(w)$, we thus 
may assume that $G$ is semisimple, and actually  of adjoint type 
or simply connected. In terms of root data, the respective conditions mean
that $\ZZ\Phi\subset X^*$ and $\ZZ\Phi^\vee\subset X_*$ have  finite index, and $\ZZ\Phi= X^*$, or $\ZZ\Phi^\vee=X_*$.
\end{remark}

\subsection{Comparison with the original definition}
\mylabel{subsec:comparison with original}

Suppose we have a Deligne--Lusztig datum
$\mathscr{D} = (X^*,\Phi,X_*,\Phi^\vee, \Delta, p, \varphi_*)$,
with    relation $\varphi_*^r=p^s\cdot\id_{X_*}$ and  minimal exponents $r,s\geq 1$.  Write   
$$
G^\Chev_\ZZ=G(X^*,\Phi,X_*,\Phi^\vee)
$$
for the corresponding \emph{Chevalley group}, a family of   reductive groups over the ring $\ZZ$,
and consider the isogeny $\varphi:G^\text{Chev}_{\FF_p}\ra G^\Chev_{\FF_p}$ on the fiber over $\FF_p$.
Under suitable conditions, one may regard this as  an iterated Frobenius map, 
albeit stemming from a twisted form: Suppose that $s/r\geq 1$ is an integer,
and that 
\begin{equation}
\label{eq:factorization varphi}
\varphi_*=p^{s/r}\cdot\sigma_*
\end{equation} 
for some endomorphism $\sigma_*$. Then $p^s\cdot\id_{X_*}=\varphi_*^r=p^s\cdot\sigma_*^r$,
thus $\sigma_*$ is   bijective of period $r\geq 1$, 
with resulting automorphism $\sigma:G^\Chev_{\FF_p}\ra G^\Chev_{\FF_p}$. 
Write $\Gamma=\ZZ/r\ZZ$, and consider the homomorphism $\Gamma\ra\Aut(G_{\rm Chev})$ 
corresponding to $\sigma$.
Set $k=\FF_{p^{s/r}}$ and $G=G^\Chev_k$.
Viewing $Q=\Spec(\FF_{p^s})$ as a $\Gamma$-torsor over $k$ via  the inverse of $F^{s/r}$, we obtain a twisted form
$\tilde{G}  = \Gamma\backslash (G \times Q) = \Gamma\backslash (G^\Chev_{\FF_p} \otimes k')$, 
where the quotient is formed with respect to the diagonal action, and $k'=\FF_{p^s}$. It comes with an identification
$$
\psi: G\times Q\lra \tilde{G}\times Q,\quad (g,q)\longmapsto (\Gamma\cdot(g,q),q).
$$
Note that one may  likewise form the Borel group $\tilde{B}$ and the maximal torus $\tilde{T}$, which is  not necessarily split.

\begin{proposition}
\mylabel{prop:isogeny becomes frobenius}
After base-change to $k'=\FF_{p^s}$, the isogeny $\varphi:G\ra G$ coincides with the iterated Frobenius map $F^{s/r}:\tilde{G}\ra\tilde{G}$.
\end{proposition}

\proof
Consider the   diagram
$$
\begin{CD}
G\times Q		@>\psi>>	\tilde{G}\times Q\\
@VF^{s/r}\times\id VV			@VVF^{s/r}\times\id V\\
G\times Q		@>\psi\circ(\sigma\times\id)>>	\tilde{G}\times Q\\
@V\id\times F^{s/r}VV			@VV\id\times F^{s/r}V\\
G\times Q		@>>\psi>	\tilde{G}\times Q.
\end{CD}
$$
The outer square is commutative, by the functoriality of   absolute Frobenius maps. In the lower square, 
the vertical arrows are isomorphisms. The two ways of composition are
$$
(g,q)\longmapsto (\Gamma\cdot (\sigma g,q),q)\longmapsto (\Gamma\cdot(\sigma g,q),F^{s/r}q)= (\Gamma\cdot(g,F^{s/r}q),F^{s/r}q),
$$
and  $(g,q)\mapsto (g,F^{s/r}q)\mapsto (\Gamma\cdot(g,F^{s/r}q),F^{s/r}q)$, so this square is commutative as well. It follows that the upper square is commutative.
This shows that with respect to the identification $\psi$ over $k'=\FF_{p^s}$, the isogeny $F^{s/r}:\tilde{G}\ra\tilde{G}$
equals   $\sigma\circ F^{s/r}=F^{s/r}\circ\sigma$. According to   \eqref{eq:factorization varphi}, this coincides
with $\varphi$.
\qed

\begin{example}\mylabel{ex:unitary DL datum}
Suppose that $\Phi$ is a root system of type $A_2$, with $\ZZ\Phi^\vee=X_*$. Then the reductive group is $G=\SL_3$, where we now work
over the prime field $k=k_0=\FF_p$.  The endomorphism $\varphi_*=(\begin{smallmatrix} &p\\p&\end{smallmatrix})$ that flips the simple coroots satisfies
$\varphi_*^r=p^s\cdot\id_{X_*}$, with minimal exponents $r=s=2$, and the   matrix $\sigma_*=(\begin{smallmatrix} &1\\1&\end{smallmatrix})$
yields a decomposition $\varphi_*=p\cdot\sigma_*$. Using the quadratic extension $k'=\FF_{p^2}$, we obtain a twisted form 
$\tilde{G} =\Gamma\backslash (G\otimes k')$.
This can be viewed as the special unitary group $\operatorname{SU}_3$ over the prime field $k$, cf.~Section~\ref{subsec:formulation a2twist}.
So  after base-changing to $k'$, the   isogeny $\varphi:G\ra G$ coincides
with the Frobenius map $F:\tilde{G}\ra \tilde{G}$.
\end{example}
 
\medskip 
We now can compare our definition   of the scheme $X(w)$  with the original construction in \cite{Deligne; Lusztig 1976}, Section 1:
There the starting point is a   reductive group $G=G_0\otimes_{k_0}k$ over $k=\FF_{p^t}^\alg$, obtained
via base-change from an algebraic group $G_0$ over a finite field $k_0=\FF_{p^t}$, with resulting isogeny $F^t\otimes\id$. 
Fixing a maximal torus $T$ and Borel group $B$ over $k$ together with a  Weyl group element $w\in W$,
Deligne and Lusztig   defined  an algebraic set $X(k)$ inside the flag variety $G/B$, as in  Definition \ref{def:deligne-lusztig datum}. 

Let $k_0\subset k_1\subset k$ be the smallest subextension such that $B,BwB\subset G$ descend to 
to some $B_1,B_1wB_1$ inside $G_1=G_{\rm Chev}\otimes k_1$. Obviously, the   algebraic set $X(k)$ stems from a subscheme $X_1\subset G_1/B_1$.
Now let $k_1\subset k_2\subset k$ be the smallest subextension such that $T$ descends to a split torus $T_2$ inside $G_2=G_{\rm Chev}\otimes k_2$.
So the latter is a  split reductive group,   and thus corresponds to a root datum $(X_*,\Phi, X^*,\Phi^\vee)$, with corresponding
split reductive group $G_\text{\rm Chev}$ over the prime field $\FF_p$. Choosing a pinning, the 
isogeny $F^t\otimes\id$ corresponds to some endomorphism $\varphi_*$. This yields a Deligne--Lusztig datum in our sense,
with resulting $X=X(w)$ already defined over the prime field $\FF_p$. By construction, this gives back
$$
X(k)=X(\FF_p^\alg)\quadand X_1\otimes_{k_1} k_{2}=X\otimes_{\FF_p} k_2.
$$
Furthermore, $X_1$ is a twisted form of $X\otimes_{\FF_p} k_1$ (cf.~Example~\ref{ex:twisted DL}). 
Note that this is exactly the situation encountered in  Proposition \ref{prop:isogeny becomes frobenius}.

The case where only a power of the isogeny
$\varphi:G\ra G$ can be related to an iterated Frobenius map is discussed in \cite{Deligne; Lusztig 1976}, Section 11, and the definition of Deligne--Lusztig varieties is extended to this case.
On then obtains as the groups $G^\varphi(k)$ the so-called \emph{Suzuki--Ree groups}.
Our definition unifies the two cases. In view of Proposition~\ref{prop:isogeny becomes frobenius} we refer to them using the following terminology:

\begin{definition}
\mylabel{def:frobenius and suzuki-ree}
Let $\mathscr{D} = (X^*,\Phi,X_*,\Phi^\vee, \Delta, p, \varphi^*)$ be a Deligne--Lusztig datum,
with    $\varphi_*^r=p^s\cdot\id_{X_*}$ and minimal exponents $s,r\geq 1$. If the  
 map $q:\Phi\ra p^\NN$ appearing in the morphism $\varphi_*$ is constant, we say that   $\mathscr{D}$ is of \emph{Frobenius type}.
Otherwise it is called of \emph{Suzuki--Ree type}.
\end{definition}

In forming the $X(w)\subset G/B$, we may assume that $G$ is simply connected, in other words
$\ZZ\Phi^\vee= X_*$. Then the simple coroots form a basis of $X_*$, and the 
 endomorphism takes the form $\varphi_*=D\cdot P$, where $P$ is a permutation matrix, 
and $D$ is a diagonal matrix with entries $p^{\nu_\alpha}=q(\alpha)$.
If $s/r\geq 1$ is an integer and $D\cdot P=\varphi_*=p^{s/r}\cdot\sigma_*$, we see that the diagonal matrix
$p^{-s/r} D= \sigma_* P^{-1}$  has entries from $\ZZ$ and  determinant from $\ZZ^\times$,  and thus $q:\Phi\ra p^\NN$ is constant.
Conversely, if this map  takes constant value $p^\nu$, we get $p^s=\varphi_*^r= p^{\nu r} P^r$,
hence $r/s=\nu$ is an integer, and thus $\varphi_*^r=p^{s/r}\cdot \sigma_*$ with $\sigma_*=P$.

Summing up, under the assumption $\ZZ\Phi^\vee= X_*$,
the Deligne--Lusztig datum  $\DLD$ is of Frobenius type if and only if Proposition \ref{prop:isogeny becomes frobenius} applies.

(The same argument often applies even without the assumption $\ZZ\Phi^\vee= X_*$. In fact, if the root datum is irreducible of type different from $D$, then the quotient of the weight lattice by the root lattice is cyclic and every subgroup is fixed by the map $\sigma^*$. See~\cite{Malle; Testerman 2011} Example~22.8.)

\begin{example}\mylabel{ex:twisted DL}
    Consider again the situation in Example~\ref{ex:unitary DL datum}, a Deligne--Lusztig datum of Frobenius type corresponding, in the sense of Deligne and Lusztig, to the special unitary group $G_0 = \SU_3$ over $\FF_p$. If $w\in W$ is fixed by $\varphi$, then the field $k_1$ in the above discussion if $\FF_p$. Independently of $w$, the field $k_2$ is $\FF_{p^2}$. As explained above, both Deligne--Lusztig varieties (i.e., in our sense and in the sense of Deligne and Lusztig) are naturally defined over $k_1$, and can be identified after base change to $k_2$. Note however that they will typically not be equal over the field $k_1$. This already happens for $w=\id$ where viewing the Deligne--Lusztig definition over $k_1$ leads to a disjoint union of copies of $\Spec \FF_p$, whereas our definition results in a scheme which also has points with residue class field $\FF_{p^2}$.
\end{example}

\subsection{Compactifications}
\mylabel{subsec:compactifications}
Fix a Deligne--Lusztig datum $\DLD$ and a ground field $k$, and form the resulting pinned reductive group
$(G,T,B,\psi)$ as above. By construction, the schemes $X(w)$ come with a naive compactification $\overline{X(w)}\subset G/B$ given as their closure inside
the flag variety. In this section we recall the smooth compactification $\bar{X}(w)$ defined by Deligne and Lusztig
(\cite{Deligne; Lusztig 1976}, Section~9) in the spirit of the construction of Bott--Samelson \cite{Bott; Samelson 1958} and Demazure \cite{Demazure 1974}.

Fix a reduced expression $w = s_{i_1}\cdots s_{i_n}$ where  $n = \ell(w)$.
Embed $X(w) \to \prod_{\nu=1}^n G/B$ as follows: 
For $g\in X(w)$, since $g_1 := g$ and $g_{n+1}:=\Frob(g)$ have relative position $w$,
 there is a unique sequence $g_2, \dots, g_n$ of elements in $G/B$ such that 
$\inv(g_\nu, g_{\nu+1}) = s_{i_\nu}$ for all $\nu=1, \dots, n$. 
We then map $g\mapsto (g_1,\dots, g_n)$. This procedure indeed works for $R$-valued points,
in view of the isomorphism~\eqref{eq: decomposition relative position}.
Now
$$
\bar{X}(w)\subset \prod_{\nu=1}^n G/B
$$
is defined  as the closure of $X(w)$.
 Note that our notation is imprecise, 
because $\bar{X}(w)$ depends on the choice of reduced expression for $w$. 
When it is necessary to indicate the reduced expression, 
we write $\bar{X}(s_{i_1},\dots, s_{i_n})$ instead. We thus have 
\[
\bar{X}(s_{i_1},\dots, s_{i_n}) =
\left\{
    (g_\nu)_\nu\in\prod_{\nu=1}^n G/B\mid \text{$ \inv(g_{\nu}, g_{\nu+1})\in \{ \id, s_{i_\nu}\}$ for $ \nu=1,\dots, n$}
\right\},
\]
where we set $g_{n+1}:=\Frob(g_1)$. Then $\bar{X}(s_{i_1},\dots, s_{i_n})$ is smooth, as can be shown by reducing it to the smoothness of Demazure resolution of the corresponding Schubert variety using the argument of Lemma~\ref{lem:properties deligne-lusztig variety}; alternatively see~\cite{Deligne; Lusztig 1976} Lemma 9.11.

Usually the compactifications $\bar{X}(w)$ and $\overline{X(w)}$ differ. In fact, the closure $\overline{X(w)}$ is often singular. But for Coxeter elements $w$, and more generally if every simple reflection occurs at most once in a (equivalently, every) reduced expression of $w$, the situation simplifies.

\begin{proposition}\mylabel{prop:comparison compactifications}
    Let $w\in W$ be an element such that every simple reflection occurs at most once in a (fixed) reduced expression of $w$.
Then the natural morphism from $\bar{X}(w)$ to the closure $\overline{X(w)}$ of $X(w)$ inside the flag variety $G/B$ is an isomorphism.
\end{proposition}

\proof
We may assume that $k$ is algebraically closed.
The assumption ensures that all subexpressions of a reduced expression of $w$ are themselves reduced. Now given any closed point $gB \in \overline{X(w)}$, say $gB\in X(v)$, $v\le w$, there is a unique subexpression of the given reduced expression of $w$ which is a reduced expression for $v$, and thus (using~\eqref{eq: decomposition relative position}) a unique preimage of $gB$ in $\bar{X}(w)$. We conclude that the birational morphism $\bar{X}(w)\to \overline{X(w)}$ is bijective, and hence an isomorphism, because $\overline{X(w)}$ is normal.
\qed

\subsection{Connected components}
\mylabel{subsec:connected components}
Despite being termed ``varieties'', the   $X(w)$ are often disconnected.
The goal of this section is to determine the scheme of connected components,   
describe the ensuing closure relation inside $G/B$, and analyze the relation to   the Tits building $\Delta(G)$. A key input is the characterization, originally due to Lusztig, which Deligne--Lusztig varieties are connected, see~\cite{Bonnafe; Rouquier 2006}.
 
Fix a Deligne--Lusztig datum  $\DLD=(X^*,\Phi,X_*,\Phi^\vee, \Delta, p,\varphi_*)$ and  a ground field   $k$   of the given characteristic $p>0$.
Form the resulting pinned reductive group $G$, with maximal torus $T$, Borel group $B$, Weyl group $W=N_G(T)/T$ and  isogeny $\varphi:G\ra G$.
The latter induces an automorphism  of the Weyl group, which we also denote by $\varphi: W\ra W$. It     stabilizes the 
set of simple reflections $\bbS\subset W$  corresponding to the set of simple roots $\Delta\subset\Phi$.

We now fix some $w\in W$ and write $\ell\geq 0$ for its length. For each reduced expression $w=s_{i_1}\ldots s_{i_\ell}$, the resulting set
$\{s_{i_1},\ldots, s_{i_\ell}\}\subset \bbS$ is independent of the factorization (\cite{LIE 4-6}, Chapter IV, \S1, No.\ 8, Proposition 7), 
and we denote it  by $\Supp(w)$.
This leads to the following notion:

\begin{definition}
\mylabel{def:phi-support}
The \emph{$\varphi$-support} $\Supp_\varphi(w)\subset\bbS$ is the smallest $\varphi$-stable set that contains  $\Supp(w)$.
\end{definition}

Let $B\subset P_w\subset G$ be the \emph{parabolic subgroup} corresponding to $\varnothing\subset \Supp_\varphi(w)\subset\bbS$.
By construction, this is $\varphi$-stable, and we get an induced   $\varphi:G/P_w\ra G/P_w$.

\begin{proposition}
\mylabel{prop:fixed in quotient by parabolic}
The scheme of fixed points $(G/P_w)^\varphi$ is finite and \'etale, and the action of $G^\varphi$ is transitive.
\end{proposition}

\proof 
Set $P=P_w$, and consider the Lang map $L:G\ra G$, which is the \'etale morphism  defined by $L(g)=g^{-1}\varphi(g)$.
Similarly as in the proof of Lemma~\ref{lem:properties deligne-lusztig variety} we have  $(G/P)^\varphi=L^{-1}(P)/P$, and it follows that 
the scheme of fixed points is smooth and zero-dimensional, hence finite and \'etale.
For the transitivity of the action we may assume that $k$ is algebraically closed.
Suppose some closed point $g\in G$ has $\varphi$-fixed coset $gP$, in other words
$g^{-1}\varphi(g)\in P$. Using the surjectivity of $L:P\ra P$, we find some closed point $h\in P$
with $h^{-1}\varphi(h)=g^{-1}\varphi(g)$. Then  $g'=gh^{-1}$ is a $\varphi$-fixed point
with $g'P=gP$. This shows that $gP$ belongs to the orbit of $eP$ with respect to $G^\varphi$.
\qed

\medskip
In turn, the orbit map for the coset $eP_w$ yields an identification
$$
(G/P_w)^\varphi= G^\varphi/(G^\varphi\cap P_w).
$$
For each coset $gB\in  G/B$ inside $X(w)$ we have $g\varphi(g)^{-1}\in BwB\subset P_w$, and it follows
that $gP_w\in G/P_w$ is $\varphi$-fixed. This yields a commutative diagram
$$
\begin{CD}
X(w)		@>>>	G/B\\
@V\pi VV		@VVV\\
(G/P_w)^\varphi	@>>>	G/P_w.
\end{CD}
$$

\begin{proposition}
\mylabel{prop:geometrically connected fibers}
The  morphism $\pi:X(w)\ra (G/P_w)^\varphi$ is smooth and surjective, with geometrically connected fibers. Moreover,
it induces a bijection between the sets of connected components.
\end{proposition}

\proof
Surjectivity follows from the transitivity of the $G^\varphi$-action on $(G/P_w)^\varphi$,
and the remaining task is to verify that  $\pi^{-1}(eP_w)$ is smooth and   connected.
We achieve this by identifying this fiber with a certain   Deligne--Lusztig variety
formed with respect to the Levi quotient of $P_w$: Let $U_w=\Rad_u(P_w)$ be the unipotent radical.
Then $G'=P_w/U_w$ is a reductive group,   $B'=B/U_w$ is a Borel group, the image of $T$ is a maximal torus $T'$, and we have an identification
$G'/B'=P_w/B$.
From $\pi^{-1}(eP_w)= X(w)\cap (P_w/B)$ we see that the fiber indeed can be seen as the Deligne--Lusztig
variety $X'(w)$ for the Deligne--Lusztig datum corresponding to the reductive group $G'$ and the induced isogeny $\varphi'$. 
Recall that   the projection $P_w\ra P_w/U_w$ admits a unique splitting whose image   contains $T$.
This  image is a Levi subgroup, and we obtain a semidirect product decomposition $P_w=U_w\rtimes G'$ 
and an inclusion of the Weyl group of $G'$ into $W$.
By construction,
the $\varphi'$-support of $w$, regarded as an element in $W'=N_{G'}(T')$, is the set of all simple reflections.
According to \cite{Bonnafe; Rouquier 2006}, Theorem 1 the Deligne--Lusztig variety $X'(w)$ is connected. 
\qed

\medskip
We see that the \emph{scheme of connected components} for $X(w)$ coincides with the  finite \'etale   scheme   $(G/P_w)^\varphi=G^\varphi/(G^\varphi\cap P_w)$,
which also can be interpreted  as the finite set of $k^\sep$-valued points  endowed with the Galois action.
Also note that the inclusions of $X(w)$ into the compactifications $\bar{X}(w)$ and $\overline{X(w)}$ yield bijections
between    sets of  irreducible components. Since these  schemes are normal,
these irreducible components are precisely the    connected components.

Given $v\leq w$ we obtain an embedding $X(v)\subset\overline{X(w)}$, and now seek to understand the ensuing closure relation
on connected components. The obvious inclusions $\Supp_\varphi(v)\subset\Supp_\varphi(w)$ and $P_v\subset P_w$ yield  a commutative diagram
$$
\begin{CD}
X(v)            @>>>            \overline{X(w)} @>>>    G/B\\
@V\pi VV                       @VV\pi V               @VVV\\
(G/P_v)^\varphi @>>\rho_{v,w}>    (G/P_w)^\varphi  @>>>    G/P_w,
\end{CD}
$$
and the morphism $\rho_{v,w}:(G/P_v)^\varphi\ra (G/P_w)^\varphi$ of finite \'etale schemes can be interpreted as the
morphism on the scheme of connected components induced by $X(v)\subset \overline{X(w)}$.
Given rational points $g,h\in G^\varphi$, it is convenient to write
$$
g X'(v)\subset X(v)  \quadand h\overline{X'(w)}\subset \overline{X(w)}
$$
for the connected component over the   cosets $gP_v$ and $hP_w$, respectively. We now can unravel the relation between 
strata, cosets and parabolics:

\begin{lemma}
\mylabel{lem:strata cosets parabolics}
For $v\leq w$ we have the equivalences
$$
g X'(v)\subset h\overline{X'(w)}\quad\Longleftrightarrow\quad
gP_v\subset hP_w\quad\Longleftrightarrow\quad
gP_vg^{-1}\subset hP_wh^{-1},
$$
which hold for all rational points $g,h\in G^\varphi$.
\end{lemma}
 
\proof
By the above discussion, the left condition is equivalent to $\rho_{v, w}(gP_v) = hP_w$. Since $\rho_{v, w}(gP_v) = g P_w$ by construction, this shows the first equivalence. Moreover, $gP_v \subseteq hP_w$ implies $h^{-1} g\in P_w$ and thus $gP_vg^{-1}\subset hP_wh^{-1}$. Finally, to show that the right hand condition implies the middle one, we easily reduce to the case $h=e$ and $P_v=B$. The unique parabolic subgroup of type $P_w$ containing $B$ is $P_w$ itself, so $B  \subseteq gP_wg^{-1}$ implies $gP_wg^{-1} = P_w$ and hence $g\in P_w$.
\qed

\medskip
Recall that the \emph{Tits building} $\Delta(G)$ (over $\FF_p^\alg$) is the simplicial complex whose \emph{vertices} are the  maximal parabolic  subgroups $P\subsetneqq G$ (over $\FF_p^\alg$),
and the \emph{$d$-dimensional faces} are the   $\{P_0,\ldots, P_d\}$ whose intersection remains parabolic.
By abuse of notation, one often identifies     faces and parabolic subgroups  $P=P_0\cap\ldots\cap P_d$.
Each \emph{apartment} comprises all $P$ containing  a  fixed split maximal torus. We obviously have 
$$
P\subset Q \quad\Longleftrightarrow\quad  \{P_0,\ldots, P_d\}\supset \{Q_0,\ldots, Q_e\}
$$
Recall that  in the Deligne--Lusztig stratification $G/B= \bigcup_{w\in W} X(w)$ the closure relation corresponds
to the Bruhat order in $W$. Using the  decomposition   of each $X(w)$ into connected components, 
we obtain the  \emph{refined Deligne--Lusztig stratification} of the flag variety $G/B$.
Note  that the Tits building $\Delta(G)$ alone usually does not determine its closure relations, 
because the parabolic subgroup $P_w\subset G$, or equivalently the set $\Supp_\varphi(w)\subset\bbS$,
usually do  not determine the element $w\in W$.
In the following special case, the situation simplifies:

\begin{definition}
\mylabel{def:varphi coxeter element}
An element 
$w\in W$ is called a \emph{$\varphi$-Coxeter element},
if each $\varphi$-orbit $J\subset \bbS$ contains exactly one element from $\Supp(w)$.
\end{definition}

Note that this are precisely the \emph{Coxeter elements}, provided that $\varphi:\bbS\ra\bbS$ is the identity.
If $\Phi\subset X^*$ is a root system   of type $A_2$ and the isogeny flips the two simple reflections in $\bbS=\{s_1,s_2\}$, 
then $w=s_1$ and $w'=s_2$ are 
$\varphi$-Coxeter elements but not Coxeter elements, while $v=s_1s_2$ is a Coxeter element  but not a $\varphi$-Coxeter element. 

For each parabolic subgroup $P\subset G$, the image $\varphi(P)$ is parabolic of the same dimension, and $\varphi^r(P)=P$.
Consequently, the  isogeny $\varphi:G\ra G$ induces an automorphism of  the Tits building $\Delta(G)$ of  period $r\geq 1$.
A face $P=P_0\cap\ldots\cap P_d$ is called \emph{$\varphi$-stable}, if the subscheme  $P\subset G$, or equivalently
the subset  $\{P_0,\ldots,P_d\}\subset\Delta(G)$, is  $\varphi$-stable.  

 For the sake of exposition we now assume that the finite \'etale scheme $(G/P_w)^\varphi$
is constant. Then 
\begin{equation}
\label{eq:strata to faces}
gX'(v)\longmapsto gP_vg^{-1}
\end{equation} 
is an \emph{order-reversing} map from  the set of connected strata $gX'(v)\subset \overline{X(w)}$, where $g\in G^\varphi$  
and $v\leq w$, to  the set of faces in the Tits building.
By construction the images are $\varphi$-stable.
The order relation on the connected strata is the closure relation, while the order relation
on faces is given by inclusion.

\begin{proposition}
\mylabel{prop:stratification and building}
Suppose that the finite \'etale scheme $G^\varphi/(G^\varphi\cap P_w)$ is constant, and that 
$w\in W$ is a $\varphi$-Coxeter element. Then   \eqref{eq:strata to faces} defines an order-reversing bijection between
the set of connected strata $gX'(v)\subset G/B$ in the flag variety contained  in $\overline{X(w)}$ with $v < w$,
and the set of  $\varphi$-stable faces $\{P_0,\ldots,P_d\}\subset \Delta(G)$ in the Tits building.
\end{proposition}

\proof
We have defined above an order-reversing map from the set of connected strata $gX'(v)\subset \overline{X(w)}$, $v<w$, to the set of $\varphi$-stable faces in the Tits building of $G$. Since $w$ is a $\varphi$-Coxeter element, every proper $\varphi$-stable subset of $\mathbb S$ arises as the $\varphi$-support of some $v < w$. This implies that the map we constructed is surjective. For its injectivity, let $g_1 X'(v_1)$ and $g_2 X'(v_2)$ be strata giving rise to the same parabolic subgroup $g_1 P_{v_1}g_1^{-1} = g_2 P_{v_2} g_2^{-1}$. It follows that $P_{v_1}$ and $P_{v_2}$ are parabolic subgroups of the same type, so $\Supp_{\varphi}(v_1) = \Supp_{\varphi}(v_2)$. Using that $w$ is a $\varphi$-Coxeter element once more, we conclude that $v_1 = v_2$ and then that $g_2^{-1} g_1 \in P_{v_1}$, so that $g_1 X'(v_1) = g_2 X'(v_2)$.
\qed

\begin{example}
\mylabel{ex:sp4 tutte-coxeter}
Suppose now that $\Phi$ is a root system of type $C_2$, with $\ZZ\Phi^\vee=X_*$ and  $\varphi_*=p\cdot\id_{X_*}$.
Then the reductive group can be seen as the automorphism group scheme $G=\Sp_4$  of the standard symplectic vector space $V=k^4$, where we now work
over the prime field $k=\FF_p$.
The maximal parabolic subgroups $Q\subsetneqq G$ are   the stabilizers of lines $L\subset V$ or 
isotropic planes $U\subset V$. Note that the isotropic $U$ containing a fixed $L$ can be viewed as lines inside $L^\vee/L$.
The minimal parabolic subgroups $B\subset G$, which are precisely the Borel groups,
are the isotropic flags $L\subset U$.
In turn, the Tits building $\Delta(G)$ is a $(p+1)$-regular bipartite graph with 
 $2\cdot (p^4-1)/(p-1)$ vertices  and $(p+1)(p^4-1)/(p-1)$ edges.
2
Write $\bbS=\{s_1,s_2\}$ for the set of simple reflections. Let  $w=s_1s_2$ or $w=s_2s_1$. Then $w$ is a $\varphi$-Coxeter element,
and  $\overline{X(w)}\subset G/B$ is a surface.
Proposition \ref{prop:stratification and building} tells us that the Tits building $\Delta(G)$ is the dual graph for the divisor 
$D=\overline{X(s_1)}\cup \overline{X(s_2)}$. As discussed in the preceding paragraph,  
the  isotropic subspaces $L$ and $U$ correspond to the irreducible components, while the isotropic flags $L\subset U$ 
indicate intersections.

We are mainly interested in the case $p=2$: Then $\Delta(G)$ is a  3-regular bipartite graph 
with  30 vertices and 45 edges, and coincides with the famous   \emph{Tutte--Coxeter graph}
(\cite{Tutte 1958} and  \cite{Coxeter 1958}, compare also   \cite{Everitt 2014}, Section 5).
See Section~\ref{sec:case c2}, in particular Figure~\ref{fig:tutte-coxeter}. 
\end{example}

\subsection{Cyclic shifts, Deligne--Lusztig reduction, and compactifications}
\mylabel{subsec:cyclic shift}
Fix a Deligne--Lusztig datum $\DLD=(X^*,\Phi,X_*,\Phi^\vee,\Delta,p,\varphi_*)$ and a ground field $k$ of the given characteristic $p>0$,
and form the resulting pinned reductive group $G$, with maximal torus $T$, Borel group $B$,   Weyl group $W=N_G(T)/T$,
and isogeny $\varphi:G\ra G$.
Write $\bbS\subset W$ for the set of simple reflections.
In this section we discuss certain canonical morphisms that go by the name \emph{Deligne--Lusztig reduction}
(\cite{Deligne; Lusztig 1976}, Theorem 1.6) and study under what conditions they extend to compactifications. See also, for instance, \cite{Wang 2024} Section 4.1.

First suppose that we have two reduced expressions of the form
\begin{equation}
\label{eq:two reduced expressions}
w=s_{i_1}s_{i_2}\ldots s_{i_\ell}\quadand w'=s_{i_2} \ldots s_{i_{\ell-1}}\varphi(s_{i_1}),
\end{equation} 
for some elements $w,w'\in W$ of common length $\ell\geq 1$, and hence   $w'=sw\varphi(s)$ with the simple reflection
 $s=s_{i_1}$. Consider the   diagram
\begin{equation}
\label{eq:cyclic shift diagram}
\begin{CD}
X(w)				@>>>	X(w')\\
@VVV					@VVV\\
\bX(s_{i_1},s_{i_2},\ldots,s_{i_\ell}) 	@>>>	\bX(s_{i_2}, \ldots, s_{i_\ell}, \varphi(s_{i_1})),
\end{CD}
\end{equation} 
where  the lower horizontal arrow is given by
\begin{equation}
\label{eq:cyclic shift}
(g_1B,\ldots,g_\ell B)\longmapsto (g_2B,\ldots, g_\ell B,\varphi(g_1)B),
\end{equation} 
while the upper horizontal arrow sends $gB$ to the unique coset $g'B$ with relative positions
$\inv(g,g')=s$ and $\inv(g',\varphi(g))=sw$. We call these morphisms \emph{cyclic shifts}.
The vertical arrows are the compactifications discussed
in Section \ref{subsec:compactifications}: 
Recall that $X(w)\ra \bX(s_{i_1}, \ldots,s_{i_\ell})$ sends $gB$ to the tuple $(g_1B,\ldots, g_\ell B)$ with $g_1=g$ and
$\inv(g_\nu,g_{\nu+1})=s_{i_\nu}$, where we formally set $g_{\ell+1}=\varphi(g)$.

\begin{proposition}
\mylabel{prop:finite universal homeomorphisms}
In the above situation, the diagram \eqref{eq:cyclic shift diagram} is commutative, and the horizontal arrows are finite universal homeomorphisms.
\end{proposition}
 
\proof
The morphisms are of finite type, and commutativity is obvious from the functorial description. It remains to check
that the arrows are universal homeomorphisms. For this we may assume that $k$ is algebraically closed, and have to check
that the maps induce   bijections  between closed points. The maps \eqref{eq:cyclic shift} and $gB\mapsto g'B$ are indeed injective,
because this holds for the isogeny $g\mapsto \varphi(g)$.  To see that \eqref{eq:cyclic shift} is surjective,
let 
\begin{equation}
\label{eq:surjectivity cyclic shift}
(g_2B,\ldots,g_\ell B, hB)\in \bX(s_{i_2}, \ldots, s_{i_\ell}, \varphi(s_{i_1}))
\end{equation} 
be a closed point. Since $g\mapsto \varphi(g)$ is surjective, we can write $h=\varphi(g_1)$. Then 
$\inv(\varphi(g_1),\varphi(g_2))\leq  \varphi(s_{i_1})$. It follows that $(g_1B,\ldots,g_\ell B)$ belongs to $\bX(s_{i_1},s_{i_2},\ldots,s_{i_\ell})$
and maps to \eqref{eq:surjectivity cyclic shift}. This establishes surjectivity for \eqref{eq:cyclic shift}. We also see that a point in $\bX(s_{i_1},s_{i_2},\ldots,s_{i_\ell})$ is mapped to $X(w')$ if and only if it lies in the open subscheme $X(w)$, so the surjectivity of the upper horizontal arrow follows, as well.
\qed

\medskip
Note that if two elements $w,w'\in W$ of the same length are related by $w'=sw\varphi(s)$ for some simple reflection $s$ such that  $\ell(sw)<\ell(w)$, then
any reduced expression for $sw$ yields reduced expressions as in \eqref{eq:two reduced expressions}.
Also note that the cyclic shift $X(w)\ra X(w')$ in general does not extend to the closures in the flag variety $G/B$.

Suppose now  that we have reduced expressions of the form
\begin{equation}
\label{eq:second reduced expressions}
w=s_{i_1} \ldots s_{i_{\ell-1}}\varphi(s_{i_1}) \quadand w'=s_{i_2} \ldots s_{i_{\ell-1}}\varphi(s_{i_1}),
\end{equation} 
for some elements $w,w'\in W$, of respective lengths $\ell\geq 2$ and $\ell-1\geq 1$, now with $w'=sw$
for the simple reflection $s=s_{i_1}$. 
For each $(g_1B,\ldots,g_\ell B)$ from the regular compactification $ \bX(s_{i_1},\ldots,s_{i_{\ell-1}},\varphi(s_{i_1}))$ we have
$$
\inv(g_1,g_2)\leq s_{i_1}\quadand \inv(g_\ell,\varphi(g_1))\leq \varphi(s_{i_1})\quadand \inv(\varphi(g_1),\varphi(g_2))\leq \varphi(s_{i_1}),
$$
and thus $\inv(g_\ell,\varphi(g_2))\leq \varphi(s_{i_1})$. In turn,   $(g_1B,\ldots,g_\ell B)\mapsto (g_2B,\ldots, g_\ell G)$ defines   a morphism
$$
\bX(s_{i_1},\ldots,s_{i_{\ell-1}},\varphi(s_{i_1}))\lra \bX(s_{i_2},\ldots,s_{i_{\ell-1}},\varphi(s_{i_1})).
$$

\begin{proposition}
\mylabel{prop:bundles with sections}
The above morphism is a $\PP^1$-bundle, and the assignment 
\begin{equation}
\label{eq:section for bundle}
(g_2B,\ldots,g_\ell B)\mapsto (g_1B,g_2B,\ldots, g_\ell B)
\end{equation} 
with $g_1=g_2$ defines a section. Furthermore, the projection on the first coordinate defines an identification
$\bX(s_{i_2},\ldots,s_{i_{\ell-1}},\varphi(s_{i_1}))=\overline{X(w')}\subset G/B$, provided that the   simple reflections  
in $w'=s_{i_2} \ldots s_{i_{\ell-1}}\varphi(s_{i_1})$ appear without repetitions.
\end{proposition}

\proof
As above, we write $s = s_{i_1}$. Consider the commutative diagram
$$
\begin{CD}
\bX(s_{i_1},\ldots,s_{i_{\ell-1}},\varphi(s_{i_1})) @>>>    \overline{\O(s)}\\
@VVV                                                        @VVV\\
\bX(s_{i_2},\ldots,s_{i_{\ell-1}},\varphi(s_{i_1})) @>>>    G/B,
\end{CD}
$$
where $\overline{\O(s)}$ is as in Section~\ref{subsec:relative position}, the upper horizontal  morphism is   $(g_1B,\ldots, g_\ell B)\mapsto (g_1B,g_2B)$, the vertical morphism
to the right is $(g_1B,g_2B)\mapsto g_2B$, and 
the lower horizontal morphism
is $(g_2B,\ldots,g_\ell B)\mapsto g_2B$. The diagram is obviously cartesian, and by  Lemma~\ref{lem:structure barOs} the projection $\overline{\O(s)}\ra G/B$
is a $\PP^1$-bundle. For the induced $\PP^1$-bundle over $\bX(s_{i_2},\ldots,s_{i_{\ell-1}},\varphi(s_{i_1}))$ the assignment \eqref{eq:section for bundle} clearly
defines a section. The last statement follows from Proposition~\ref{prop:comparison compactifications}.
\qed

\medskip
Note that these $\PP^1$-bundles are locally trivial in the Zariski topology, because they admit a section.
The following special case will play an important role throughout:

\begin{corollary}
\mylabel{cor:structure of ruled surface}
Let $s$ be a  simple reflection  that is not $\varphi$-fixed, and set $w=s\varphi(s)$.  Then the morphism
$$
\overline{X(w)}=\bX(s,\varphi(s))\lra \bX(\varphi(s))=\overline{X(\varphi(s))},\quad (g_1B,g_2B)\longmapsto g_2B
$$
is a $\PP^1$-bundle with a section, giving $\overline{X(w)}$ the structure of a ruled surface.
\end{corollary}

\subsection{Deligne--Lusztig varieties of dimensions 0 and 1}
\mylabel{subsec:small dimensions}
Let us consider a Deligne--Lusztig datum $\DLD$ with associated pinned group, Borel subgroup and maximal torus $G\supset B\supset T$, $\varphi_*^r = p^s$ and $q=p^{s/r}$.

The $0$-dimensional Deligne--Lusztig variety $X(\id)$ is the fix point scheme $(G/B)^\Frob = G^\Frob/B^\Frob$, where the equality follows from Lang's theorem. As a consequence of Proposition~\ref{prop:structure Gphi}, it is finite étale over $\FF_p$.

We then have
\[
    \sum_{x\in X(\id)} [\kappa(x):\FF_p] = \lvert X(\id)(\FF^\alg)\rvert = \sum_{w\in W^\Frob} q^{\ell(w)}.
\]
This follows from the following lemma, known as the rational Bruhat decomposition.

\begin{lemma}\mylabel{lem:rational Bruhat decomposition}
Let $G$ be a connected reductive group over $\FF^\alg$ with fixed maximal torus $T$ and Borel subgroup $B$. Let $\Frob\colon G\to G$ be an isogeny such that $\Frob^r = F_{G/\FF^\alg}^s$ for some $r$, $s$.

We denote by $G^\Frob$ the $\Frob$-fix points in $G(\FF^\alg)$, and similarly for $B$ and $G/B$.
The Bruhat decomposition of $G/B$ induces a Bruhat decomposition
\[
    (G/B)^\Frob = G^\Frob/B^\Frob = \bigsqcup_{w\in W^\Frob} B^\Frob w B^\Frob/B^\Frob.
\]
Moreover, for each $w\in W^\Frob$ we have $\lvert B^\Frob wB^\Frob/B^\Frob \rvert = q^{\ell(w)}$, where $q = p^{s/r}$ and $\ell$, as before, is the length function on $W$.
\end{lemma}

\begin{proof}
    The first equality follows from Lang's theorem. For a proof of the other statements we refer to~\cite{Steinberg 1968} Theorem~11.1 and~\cite{Carter 1972} Chapters 13, 14, in particular Lemma~14.1.2. See also~\cite{Malle; Testerman 2011} Chapter~24.
\end{proof}

The irreducible Deligne--Lusztig curves arise from Dynkin types $A_1$, ${^2}A_2$, ${^2}B_2$ and ${^2}G_2$, see Section~\ref{subsec:connected components}.

\begin{proposition}\mylabel{prop:DL curves}
    Consider a Deligne--Lusztig datum $\mathscr D$ with underlying root datum of type ${^2}A_2$, ${^2}B_2$ or ${^2}G_2$, let $q=p^{s/r}$ where $\varphi_*^r=p^s$ and denote by $s_1, s_2$ the simple reflections of the Weyl group. Then the Deligne--Lusztig curves $X(s_1)$ and $X(s_2)$ are isomorphic, and likewise their closures $\overline{X(s_1)}$ and $\overline{X(s_2)}$ in $G/B$ are isomorphic.
\end{proposition}

\begin{proof}
    Using cyclic shifts we obtain purely inseparable maps $X(s_1)\to X(s_2) \rightarrow X(s_1)$ whose composition is the endomorphism $\varphi^2$ (Section~\ref{subsec:cyclic shift}). In the cases at hand, $\varphi^2$ is the relative $q^2$-Frobenius of $X(s_{1})$. This implies that $X(s_2)$ is isomorphic to a Frobenius pullback of $X(s_{1})$ (for a $q'$-Frobenius for some divisor $q'$ of $q^2$). Since $X(s_1)$ is defined over the prime field, it follows that $X(s_1)$ and $X(s_2)$ are isomorphic. The same reasoning applies to the closed Deligne--Lusztig curves.

    In the case of ${}^2 A_2$, one can also argue more directly and use that there exists a group automorphism exchanging $s_1$ and $s_2$.
\end{proof}

For Dynkin type $A_1$, the associated Deligne-Lusztig curve is the projective line $\PP^1_k$, for ${}^2 A_2$ it is the Hermitian curve. In Section~\ref{subsec:tables} we give formulas for the genera of Deligne--Lusztig curves.

\section{Dualizing sheaves for Deligne--Lusztig varieties}
\mylabel{sec:dualizing sheaves}

The goal of this section is  to express the canonical divisor of a Deligne--Lusztig
surface $X=\bX(w)$, say $w=s_{i_1}s_{i_2}$, as a linear combination of the curves 
$D_1=\bX(s_{i_2})$ and $D_2=\bX(s_{i_1})$. We will first, following Hansen~\cite{Hansen 1999}, develop the general theory, and then specialize to the case of surfaces.

\subsection{The canonical bundle formula of Hansen}
\mylabel{subsec:canonical bundle formula}

Throughout this section, we fix a  Deligne--Lusztig datum
$$
\mathscr D = (X^*,\Phi,X_*,\Phi^\vee, \Delta, p, \varphi_*)
$$
and form  the corresponding pinned reductive group $G$ over the prime field $k=\FF_p$, with maximal torus $T$,
Borel group $B$, and isogeny $\varphi:G\ra G$ satisfying $\varphi^r=F^s$.

Fix $w\in W$. We then have the Deligne--Lusztig variety $X(w)$.
Write $\alpha_1,\ldots,\alpha_\ell\in \Delta$ for the simple roots,  
and $s_1,\ldots, s_\ell$ for the corresponding simple reflections in the Weyl group $W$,
and choose a reduced expression $w=s_{i_1}\ldots s_{i_m}$, $m = \ell(w)$.
This defines  a smooth compactification  $\bar{X}(w)$.
Write  
$$
\partial\bar{X}(w)=\bar{X}(w)\smallsetminus X(w)\quadand \pi\colon\bar{X}(w)\lra G/B
$$
for  the boundary divisor regarded as reduced scheme, and the projection to the flag variety.
We now seek to express the dualizing sheaf 
$\omega_{\bar{X}(w)}=\det(\Omega^1_{\bar{X}(w)/k})$, or equivalently the canonical divisor $K_{\bX(w)}$, 
in terms of invertible sheaves on the flag variety  and the above boundary divisor.

First observe that the unipotent radical $\Rad_u(B)$ is a complement for the maximal torus $T\subset B$,
and gives a decomposition $B=\Rad_u(B)\rtimes T$, with ensuing identification $T=B/\Rad_u(B)$.
Consequently, the restriction map 
$\Hom(B,\GG_m)\ra \Hom(T,\GG_m)$ is bijective.  In turn, each  character $\lambda\colon T\ra\GG_m$
gives rise to a
line bundle
$$
G\wedge^B\AA^1=(G\times\AA^1)/B.
$$
Here the quotient is formed with respect to the right action by $B$ given by
\[
    (g, x) \cdot b = (gb, \lambda(b^{-1})x).
\]
The corresponding invertible sheaf $\O_{G/B}(\lambda)$ on the flag variety $G/B$ is defined as the sheaf of sections of this line bundle.
Note that $\O_{G/B}(\lambda)$ is ample if and only if $\langle \alpha, \lambda \rangle < 0$ for all simple roots $\alpha$, i.e., if and only if $\lambda$ is strictly $B$-antidominant.

As customary, we write
$$
\rho=\frac{1}{2}\sum_{\alpha\in\Phi_+}\alpha\quad  \in X^*\otimes \QQ
$$
for the half-sum of positive roots. Since $2\rho\in X^*$, we get a character $2\rho\colon T\ra\GG_m$ and an invertible sheaf $\O_{G/B}(2\rho)$. Moreover, since $\rho$ is also equal to the sum of the fundamental weights, $\rho$ itself is a character if $G$ is semisimple and simply connected. Since the flag variety $G/B$ and the Deligne--Lusztig varieties are not changed when we replace $G$ by the simply connected cover of its adjoint group $G^{\rm ad}$ (the quotient of $G$ by its center), $\rho$ defines a line bundle on $G/B$ which we denote by $\O_{G/B}(\rho)$.

\begin{proposition} 
The dualizing sheaves for the flag variety $G/B$ and the  smooth compactification $\bar{X}(w)$ are given by  
$$
\omega_{G/B}=\O_{G/B}(2\rho)\quadand
\omega_{\bar{X}(w)} = \O_{\bar{X}(w)}(-\partial\bar{X}(w)) \otimes \pi^*\O_{G/B}(\rho-\Frob^*(\rho)).
$$
\end{proposition}

\begin{proof}
The former appears in \cite{Jantzen 2003}, Section II.4.2 (note that in loc.~cit., the Borel subgroup $B$ corresponds to the \emph{negative} roots, a normalization opposite to ours; this explains the sign discrepancy between the above statement and the formula given in loc.~cit.).The  latter is Hansen's Canonical Bundle Formula~\cite{Hansen 1999}, Proposition~1.
\end{proof}

\medskip
Note that we have    $\partial\bar{X}(w)=D_1+\ldots+D_m$ for   the divisors
$$
D_j = \overline{X}(s_{i_1}, \dots, \widehat{s_{i_j}}, \dots, s_{i_m}),
$$
and that the summands share  no common irreducible component. To make the above formula more useful,
we will express the preimages $\O_{\bar{X}(w)}(\lambda)=\pi^*( \O_{G/B}(\lambda))$   in terms of the above boundary divisors.

Recall that  $\alpha_1,\ldots,\alpha_l\in \Phi$ are the simple roots, and that we have a
reduced expression     $w=s_{i_1}\ldots s_{i_m}$. Write
 $\alpha_1^\vee,\ldots,\alpha_l^\vee\in\Phi^\vee$ for the   simple coroots.
For $1\leq j\leq m$, we now consider the coroots
\begin{equation}
\label{coroots in formula}
\tilde{\alpha}_j^\vee= s_{i_1}\cdots s_{i_{j-1}}(\alpha_{i_j}^\vee)\in\Phi^\vee.
\end{equation} 
Our Weyl group element $w$ induces an automorphism of $X^*$, and
combining this with $\varphi^*$ we get  an endomorphism $\varphi^*\circ w^{-1}-\id_{X^*}$ of $X^*$, so that each character $\mu\in X^*=\Hom(T,\GG_m)$ yields a character
$\lambda= \varphi^*(w^{-1}(\mu))-\mu$.
According to \cite{Deligne; Lusztig 1976}, Section 9.4  we have an identification
$$
\O_{\bar{X}(w)}(\lambda) = \O_{\bar{X}(w)}  \left(\sum_j \langle \mu, \tilde{\alpha}_j^\vee\rangle D_j\right).
$$
In particular, the invertible sheaf corresponding to $\lambda:T\ra \GG_m$ corresponds to a Cartier divisor which is a linear combination of the boundary divisors of $\bar{X}(w)$.

As observed  in \cite{Haastert 1986}, proof of Proposition 2.3, the endomorphism $\varphi^*\circ w^{-1}-\id_{X^*}\in \End(X^*)$
has non-zero determinant.
For the half-sum of positive roots $\rho$ we choose   some $m>0$ and some $\mu\in X^*$ such that  
\begin{equation}
\label{m and mu}
m(\rho-\varphi^*(\rho)) = \varphi^*(w^{-1}(\mu))-\mu.
\end{equation} 
This gives the desired canonical bundle formula in terms of the boundary divisor, at least
after passing to multiples, cf.~\cite{Hansen 1999} Proposition~2.

\begin{proposition}\mylabel{prop:canonical divisor DL}
The canonical divisor of $\bar{X}(w)$ satisfies
$$
mK_{ \bar{X}(w)}  =   \sum_{j=1}^n (\langle \mu, \tilde{\alpha}_j^\vee\rangle -1) D_j,
$$
with the integer $m>0$,  character $\mu:T\ra\GG_m$, and  coroots  $\tilde{\alpha}_j^\vee$  from \eqref{coroots in formula}
and  \eqref{m and mu}.
\end{proposition}

Summing up, our Deligne--Lusztig variety $X=\bX(w)$ comes with distinguished  boundary divisors $D_1,\ldots, D_n$,
and we have and some multiple $mK_X$ becomes a linear combination $\sum_{j=1}^n\lambda_jD_j$,
and one may compute the coefficients $\lambda_j$ by hand  or via computer algebra programs.
Although this information in not sufficient to determine the position of the ray $\RR\cdot K_X$ with respect to the nef cone $\operatorname{Nef}(X)$ inside the real vector space $\Num(X)_\RR$, the following terminology is very natural: We say that the canonical divisor $K_X$ \emph{has negativity} if all the coefficients satisfy $\lambda_j\leq 0$.

\subsection{Tables}
\mylabel{subsec:tables}

In Proposition \ref{prop:canonical divisor DL}, one may compute
$\tilde{\alpha}_j^\vee$,  $m$,  $\mu$ and thus the   coefficients in the
canonical bundle formula by hand, or more efficiently with  computer algebra. We
record the results for those connected Deligne--Lusztig surfaces arising from
connected Dynkin diagrams and where both coefficients of $K_{X}$ as in
Proposition~\ref{prop:canonical divisor DL} are non-positive (for at least one
$q$) in the following table. Furthermore, we add a table with data on $0$- and
$1$-dimensional Deligne--Lusztig varieties, because it will be useful to have
this information at hand in later sections.
See the appendix (Section~\ref{sec:appendix}) for the analogous data for connected Deligne--Lusztig surfaces where the canonical divisor does not have negativity for any $q$.

\begin{landscape}
    \noindent
    \textbf{Surfaces with negativity in the canonical divisor.}
    In the Suzuki--Ree cases ${}^2C_2$ and ${}^2 G_2$ we have $p=2$ and $p=3$, respectively; we then have $q = p^{\frac{2n+1}{2}}$ for some $n \ge 0$ and we write $q_{0} = p^n$.
    \bigskip
    \begin{longtable}{p{1.5cm}|C{2cm}C{1.5cm}|C{1.8cm}p{2.8cm}p{2.5cm}C{2cm}p{5cm}}
\toprule
\makecell{Dynkin\\type} 	& $r$, $s$	& $q = p^{\frac{s}{r}}$	& $w$ 	& \multicolumn{2}{l}{Coefficients in $K_X$}							& \makecell{$K_X$\\negative?} & Geometry\\
\toprule
$A_2$			& $1$, $s$			& $p^s$	& $s_1s_2$		& $\dfrac{q^{2} - 2 q - 2}{q^{2} + q + 1}$, &  $\dfrac{-3}{q^{2} + q + 1}$				& \makecell{$p=2$\\$q=2$} & blow-up of $\PP^2_{\FF_p}$ in $\PP^2(\FF_q)$\\
\midrule
${}^2A_2$			& $2$, $2 s_{0}$			& $p^{s_{0}}$	& $s_1s_2$		& $\dfrac{q - 2}{q + 1}$, & $\dfrac{-3}{q + 1}$						& \makecell{$p=2$\\$q=2$}& {sym.\ product of $E_3/\FF_2$}\\
\midrule
$C_2$	
        & $1$, $s$			& $p^s$	& $s_2s_1$		& $\dfrac{q^{2} - q - 2}{q^{2} + 1}$, & $\dfrac{2 q - 4}{q^{2} + 1}$				& \makecell{$p=2$\\$q=2$} & {supersing.\ K3 surface/$\FF_2$}\\
\midrule
${}^2C_2$		
          &	$2$, $2n+1$		& $2^{\frac{2n+1}{2}}$ & $s_1s_2$		& $\dfrac{4 q_{0}^{2} - 6 q_{0} + 2}{2 q_{0}^{2} - 1}$, & $\dfrac{-4 q_{0} + 3}{2 q_{0}^{2} - 1}$				& \makecell{$p=2$\\$q_{0}=1$} & {sym.\ product of $E_5/\FF_2$}\\[.5cm]
          & 			& 	& $s_2s_1$		& $\dfrac{2 q_{0}^{2} - 4 q_{0} + 2}{2 q_{0}^{2} - 1}$, & $\dfrac{-6 q_{0} + 4}{2 q_{0}^{2} - 1}$				& \makecell{$p=2$\\$q_{0}=1$} & {can.\ ruled over $E_5/\FF_2$}\\
\midrule
${}^2G_2$
        &		$2$, $2n+1$	&	$3^{\frac{2n+1}{2}}$ & $s_2s_1$		& $\dfrac{3q_{0}^{2} - 5 q_{0} + 2}{3q_{0}^{2} - 1}$, & $\dfrac{-9 q_{0} + 5}{3q_{0}^{2} - 1}$		& \makecell{$p=3$\\$q_{0}=1$} &  {ruled over Ree curve/$\FF_3$}\\
\bottomrule
\end{longtable}

\bigskip
\noindent
\textbf{Deligne--Lusztig varieties of dimension $0$.}
In this table we list the number of geometric points $\lvert X(\id)(\FF_p^\alg)\rvert$ for the root systems of rank $2$.

\bigskip
\begin{tabular}{p{3cm}|C{4cm}C{4.5cm}C{2cm}C{2cm}C{2cm}}
    \toprule
    Dynkin type & $A_2$ & $C_{2}$ & ${}^2 A_2$ & ${}^2 C_2$ & ${}^2 G_2$ \\
    \midrule
    $\lvert X(\id)(\FF_p^\alg)\rvert$ & $q^3 + 2q^2 + 2q + 1$ &  $q^4+2q^3+2q^2+2q+1$ & $q^3+1$ & $q^4+1$ & $q^6+1$\\[.3cm]
    \toprule
\end{tabular}
\end{landscape}

\bigskip
\noindent
\textbf{Deligne--Lusztig curves.}
In this table we record the genera of the irreducible Deligne--Lusztig curves $\bar{X}(s)$. 
Recall that by Proposition~\ref{prop:DL curves} the genus is independent of the choice of element $w$ of length $1$, once the Deligne--Lusztig datum is fixed. See also~\cite{Hansen 1992}. See above for the definition of $q_{0}$ in the Suzuki--Ree cases.

\bigskip
\begin{tabular}{p{3cm}|C{1cm}C{1.7cm}C{2cm}C{4cm}}
    \toprule
    Dynkin type & $A_1$ & ${}^2 A_2$ & ${}^2 C_2$ & ${}^2 G_2$ \\
    \midrule
    Genus $g(\bar{X}(s))$ & $0$ & $\dfrac{q^2-q}{2}$ & $2q_{0}^3 - q_{0}$ & $\dfrac{27q_0^5 + 9q_0^4 - 3q_0^2 - 3q_0}{2}$\\[.3cm]
    \toprule
\end{tabular}

\section{Weil restriction}
\mylabel{sec:weil restriction}

\subsection{Cases of Frobenius type}
\mylabel{sec:case of frobenius}

Let $q= p^{s_0}$ be a power of the prime $p$ and let $d > 1$. We consider the Deligne--Lusztig datum with underlying pinned root datum the product of $d$ copies of the root datum of $\SL_2$, i.e., $X^* = (\ZZ^2/(1, 1)\ZZ)^d$ and $\Phi = \{ 0, (1, -1), (-1, 1)\}^d\setminus \{(0, \dots, 0)\}$, and further $X_* = \left( \{(a, -a);\ a\in \ZZ\} \right)^d$ and $\Phi^\vee = \{ 0, (1, -1), (-1, 1) \}^d \setminus \{(0, \dots, 0)\}$. We let $\Delta = \left\{ (0, \dots, 0, (1, -1), 0, \dots, 0) \right\}$.
The corresponding pinned reductive group is $G=(\SL_2)^d$ with the standard pinning.
The prime $p$ is the prime chosen above, and we define $\varphi_*$ by
\[
    \varphi_* ((a_1, -a_1), \dots, (a_d, -a_d)) = (q(a_2, -a_2), \dots, q(a_d, -a_d), q(a_1, -a_1)).
\]
We have $(\varphi_*)^d = q^d = p^{s_{0}d}$, and it is easy to check that this is a Deligne--Lusztig datum of Frobenius type.

In fact, we may describe a corresponding group $G_0 / \FF_q$ explicitly in terms of Weil restriction of scalars. Let $k$ be the finite field with $q^d$ elements. Let $G_0 = \Res_{k/\FF_q}(\SL_{2, k})$ be the Weil restriction of scalars of $\SL_2$ along the extension $k/\FF_q$. We then have $G_k=G_0\otimes_{\FF_q}k = \prod_{\gamma\in \Gal(k/\FF_q)} \SL_{2, k}$.

Let $B_0\subset G_0$ be the standard Borel subgroup (i.e., the Weil restriction of the Borel of upper triangular matrices in $\SL_2$). Then $B= B_0\otimes_{\FF_q}k$ is the product of the subgroups of upper triangular matrices in each factor. The action of the $q$-Frobenius $\Frob = F_{G/k}^{s_0}$ on $G$ is given, for a suitable numbering $G = \prod_{i=1}^d \SL_2$, by
\[
    \Frob((g_1, \dots, g_d)) = (F^{s_{0}}(g_2), \dots, F^{s_0}(g_d), F^{s_0}(g_1)),
\]
where $F^{s_0} = F_{\SL_{2, k}/k}^{s_0}$ is the $q$-Frobenius on $\SL_{2, k}$.

We have $G/B = (\PP^1)^d$.
The absolute Weyl group $W$ is the product $S_2\times\cdots\times S_2$ of $d$ copies of the Weyl group of $\SL_2$. We denote by $s_i\in W$ the element which is non-trivial in the $i$-th entry, and trivial in all other entries. The elements $s_i$ have order $2$ and commute with each other. The action on $W$ induced by $\Frob$ is given by cyclically permuting the factors.

Let $1\le i < j \le d$, and let $w = s_is_j$. Then each element of $X(w)$ is determined by its entries in places $i$ and $j$, and on these, an open condition is imposed. Thus $\overline{X}(w) = \overline{X(w)} \cong \PP^1_k\times\PP^1_k$. It is not hard to write down the ``boundary divisors'' $\overline{X(s_j)}$, $\overline{X(s_i)}$ in terms of graphs of powers of $F^{s_{0}}$, and to compute the coefficients of an expression of the canonical divisor as a linear combination of the boundary divisors.

If we start with a Deligne--Lusztig datum corresponding to a connected Dynkin
diagram, there may be non-connected Deligne--Lusztig surfaces whose connected
components are of this type. See the discussion in Section~\ref{subsec:connected components}. Examples are the cases ${}^2A_3$, $w = s_1s_3$, and
${}^2A_4$, $w = s_1s_4$, and ${}^3D_4$, $w = s_1s_3$.

\subsection{A Suzuki--Ree case}
\mylabel{subsec:non-abs simple Suzuki-Ree}

A similar case which however is not of Frobenius type can be constructed as follows; clearly, many variants of this are possible. Let again $q = p^{s_{0}}$ for some prime number $p$. Consider the pinned root datum which is the product of two copies of the root datum for $\SL_2$, as in the previous section but now, for simplicity, with just two factors.

We define
\[
    \varphi_*((a_1, -a_1), (a_2, -a_2)) = ((a_2, -a_2), q(a_1, -a_1)).
\]
We then have $(\varphi_*)^2 = q = p^{s_{0}}$.
Again it is easy to check that we have defined a Deligne--Lusztig datum of Suzuki--Ree type.

The corresponding split reductive group is $G = \SL_2\times \SL_2$, and the isogeny $\varphi$ is given by
\[
    \Frob(g_1, g_2) = (g_2, F_{\SL_2}^{s_{0}}(g_1)).
\]
Similarly as in Section~\ref{sec:weil restriction}, $\overline{X(s_1s_2)} \cong \PP^1\times \PP^1$.
Deligne--Lusztig surfaces of this type occur as the connected components of the Deligne--Lusztig surface for Dynkin type ${}^2F_4$ and $w = s_1s_4$.

\section{The case \texorpdfstring{$A_2$}{A2}}
\mylabel{sec:case a2}

\subsection{Deligne--Lusztig varieties in Dynkin type \texorpdfstring{$A_n$}{An}}
\mylabel{subsec:an}
Let $(X^*, \Phi, X_*, \Phi^\vee, \Delta)$ be a pinned root datum of type $A_{n-1}$ which corresponds to the reductive group $G=\SL_n$, the Borel subgroup $B$ of upper triangular matrices and the diagonal maximal torus $T$. Fix a prime number $p$, an integer $s \ge 1$, and set $\varphi_* = p^s\cdot\id$. We obtain a Deligne--Lusztig datum $(X^*, \Phi, X_*, \Phi^\vee, \Delta, p, \varphi_*)$. Write $q = p^s$. The associated isogeny is $\varphi = F^s\colon \SL_n\to \SL_n$, the $q$-Frobenius morphism.

\begin{example}\mylabel{ex:drinfeld-space}
    Let $k = \FF_p$ and fix $s \ge 1$, $q = p^s$, and $\varphi$ the $q$-Frobenius as above.
    Let $w$ be the ``standard'' Coxeter element $w = s_1s_2\cdots s_{n-1}$. The Deligne--Lusztig variety $X(w)$ is called the \emph{Drinfeld upper half space} over the finite field $\FF_q$. See~\cite{Deligne; Lusztig 1976} Section~2.2.

    Let $\pi\colon G/B \to \PP^{n-1}$ be the projection which maps a full flag $(F_i)_i$ to the line $F_1$. Then $\pi$ induces an isomorphism from $X$ onto its image
    \[
        \pi(X) = \PP_k^{n-1} \setminus V,
    \]
    where $V$ is the (unique) closed subscheme of $\PP_k^{n-1}$, such that $V\otimes_k k^\alg \subset \PP^{n-1}_{k^\alg}$ is the union of all $\FF_q$-rational hyperplanes. In fact, this union clearly descends to a hypersurface $V$ in $\PP_k^{n-1}$, and it is thus enough to prove the description over $k^\alg$, where it follows from Lemma~\ref{lem:relative position explicit} without much difficulty; see loc.~cit. Explicitly, $V$ is the vanishing locus of the Moore determinant attached to the homogeneous coordinates.

    One can construct the compactification $\bar{X}(w)=\overline{X(w)}$ of the Drinfeld half space as a sequence of blow-ups of projective space. See \cite{Genestier 1996} Proposition II.1.6, \cite{Langer 2019} Proposition~8.3.
See the following section for a more detailed study of the case $n=3$.
\end{example}

\subsection{Description of Deligne--Lusztig surfaces in type \texorpdfstring{$A_2$}{A2}}
\mylabel{subsec:case a2}

We now restrict to the case $n=3$. We work over the prime field $k=\FF_p$.
We write $\Flag := G/B$ and view this as the variety of full flags in $k^3$.
For $w = s_{1}s_{2}$ we obtain the Drinfeld space $X := \bX(s_{1}s_{2})$, Example~\ref{ex:drinfeld-space}. We also define $X' = \bX(s_{2}s_{1})$.

We have the following descriptions in terms of flags. These are easy consequences of Lemma~\ref{lem:relative position explicit} (and have analogs for the Drinfeld upper half space of higher dimension; cf.~\cite{Deligne; Lusztig 1976} Section~2.2),
\[
    X(s_{1}s_{2}) = \left\{ (L, L\oplus \Frob(L)) \in \Flag;\ 
        L\subset k^3\ \text{a line s.t.}\
        k^3 = \bigoplus_{i=0}^2 \Frob^i(L)
    \right\}
\]
and similarly
\[
    X(s_{2}s_{1}) = \left\{ (U\cap \Frob U, U)\in \Flag;\ 
        U\subset k^3\ \text{a plane s.t.}\
        k^{3\vee} = \bigoplus_{i=0}^2 \Frob^i(U^\perp)
    \right\}
\]
where $-^\vee$ denotes the dual vector space, and $U^\perp = (k^3/U)^\vee\subset k^{3\vee}$, and where the descriptions have to be read as descriptions of the functors represented by $X(s_{1}s_{2})$ and $X(s_{2}s_{1})$ in the obvious way.

The compactifications $X$ and $X'$ are obtained by taking the union of $X(s_{1}s_{2})$ or $X(s_{2}s_{1})$, respectively, with $\bX(s_1)\cup \bX(s_2)$, and
\[
\bX(s_1) = \left\{ (L, U)\in\Flag;\ U=\varphi(U) \right\},\quad
\bX(s_2) = \left\{ (L, U)\in\Flag;\ L=\varphi(L) \right\}.
\]

We have an isomorphism $X\cong X'$ induced by identifying $(k^3)^\vee$ with $k^3$ in the natural way, i.e., matching the standard basis of $k^3$ and its dual basis, and correspondingly, identifying $\PP^2$ with the dual projective space. Explicitly, this morphism maps $L$ to $L^\perp\subset (k^3)^\vee\cong k^3$.

The projection $\pi\colon X\to \PP^2$ mapping $(L \subset U)\mapsto L$ restricts to an isomorphism on $X(s_{1}s_{2})$ and contracts the curve $\bX(s_2)$ to the finite closed subscheme $(\PP^2_k)^\varphi$. More precisely, the morphism $X\to \PP^2$ may be identified with the blow-up of $\PP^2$ in $(\PP^2_k)^\varphi$, which is the closed subscheme consisting of all points whose residue class field is contained in $\FF_q$. In fact, since $\bX(s_2) = \pi^{-1}((\PP^2)^\varphi)$ is a Cartier divisor in $X$, we have a natural morphism from $X$ into the blow-up, and this birational morphism is easily checked to be bijective, and hence an isomorphism. After base change to $\FF_q$, this is the blow-up of $\PP^2$ in the $q^2+q+1$ points $\PP^2(\FF_q)$. Note that these points are not in general position, so even if $q=2$, we do not obtain a del Pezzo surface.

The exceptional divisor of the blow-up $\pi$ is $\bX(s_2)$. The divisor $\bX(s_1)$ is mapped onto the closed subscheme $\PP^2\setminus\pi(X(s_1 s_2))$.
Cf.~Example~\ref{ex:drinfeld-space}. See also~\cite{Hansen 1999} \S~4 Example~1.

We see from the table in Section~\ref{subsec:tables} or by computing the canonical divisor using the description as a blow-up, that there is negativity in the canonical divisor if and only if $p=q=2$.

\section{Prerequisites from algebraic geometry}
\mylabel{sec:prerequisites}

In this section we collect some basic facts concerning K3 surfaces, ruled surfaces, elliptic curves, and Albanese maps, which play a role throughout.

\subsection{K3 surfaces}
\mylabel{subsec:k3 surfaces}

Let $k$ be a ground field of characteristic $p\geq 0$,
and $X$ be  a smooth proper surface that is geometrically connected.
We say that $X$ \emph{satisfies $c_1=0$}  or is \emph{numerically $K$-trivial}, if  the dualizing sheaf  has intersection number $(\omega_X\cdot C)=0$ for each curve $C\subset X$.
Then   the only possible values for the second Betti number   are $2,6,10, 22$,
by  the classification of algebraic surfaces (\cite{Bombieri; Mumford 1977}, Introduction).
For  $b_2=22$ one says that $X$ is a  \emph{K3 surface}. The simplest examples are smooth quartics in $\PP^3$.

We now collect standard  facts on K3 surfaces that are used freely throughout.
For further details we refer to the monographs of Huybrechts \cite{Huybrechts 2016} and Kond\={o} \cite{Kondo 2020}. First note that $\omega_X=\O_X$,
and the  cohomological invariants are 
$$
h^0(\O_X)= 1\quadand h^1(\O_X)=0\quadand  h^2(\O_X)=1.
$$
The group $\Pic^0(X)$  is trivial, hence $\Pic(X)=\Num(X)$ is a free
abelian group of some rank $\rho$,  endowed with a non-degenerate intersection pairing.
In turn, the Picard scheme $\Pic_{X/k}$ is a local system over $\Spec(k)$ 
whose rank is some $\rho^\sep$. These numbers are related by $1\leq\rho\leq \rho^\sep\leq b_2=22$. 
For $k=\CC$, the singular cohomology group $H^2(X,\ZZ)$ with   cup product is isometric to the unimodular lattice 
$E_8(-1)^{\oplus 2}\oplus U^{\oplus 3}$.
Moreover, the Hodge decomposition shows   $\rho\leq h^{1,1}=20$.

This is no longer true in characteristic $p>0$: If the lattice $L=\Pic(X\otimes k^\sep)$ has rank $b_2=22$,
the K3 surface $X$ is called \emph{supersingular}. It then follows that the discriminant group $L^*/L$
is an elementary abelian $p$-group endowed with a symplectic form, of even dimension  
$2\sigma$ over $\FF_p$ for some $1\leq \sigma\leq 10$. The latter is called the \emph{Artin invariant} for the supersingular
K3 surface $X$; it determines the lattice $L$ up to isometry. For details we refer to 
the discussions in \cite{Kondo; Schroeer 2021} before Lemma 3.4 and Proposition 3.5. 
The Artin invariant is lower semi-continuous in families.
Up to twisted forms, there is but one supersingular K3 surface $X$ of Artin invariant $\sigma =1$,
which over $k^\sep $  can be characterized by   configurations of curves (\cite{Dolgachev; Kondo 2003}, Theorem 1.1).

Suppose we have a morphism $f:X\ra B$ to a curve $B$ with $f_*(\O_X)=\O_B$. Then $h^0(\O_B)=1$
and $h^1(\O_B)=0$  by the Leray--Serre spectral sequence, hence the base is  isomorphic to $\PP^1$ provided that $X$ carries
a zero-cycle   of odd degree, or $k$ has trivial Brauer group. For the generic fiber $X_\eta$, 
the dualizing sheaf coincides with the structure sheaf,
  thus $h^0(\O_{X_\eta})=h^1(\O_{X_\eta})=1$, and one says that $f:X\ra B$ is a \emph{genus-one fibration}.
The types of the geometric fibers are distinguished by their
Kodaira symbols $\I_n,\II, \ldots, \II^*, \I_n^*$. For each closed  point $s\in B$, the fiber $D=X_s$ has the property   $(D\cdot D_i)=0$
for all its irreducible components $D_i$.
Conversely, suppose we have a curve $D=\sum_{i=1}^r m_iD_i$   with intersection numbers $(D\cdot D_i)=0$, and obviously also $(\omega_X\cdot D_i)=0$.
We then say that $D$ is a \emph{curve of canonical type}. It   follows that the invertible sheaf $\shL=\O_X(D)$
is semiample, and defines a morphism $f:X\ra B$ to a curve as above,  with $D$ as a multiple of some closed fiber. 
In turn, to specify a genus-one fibration $f:X\ra B$ amounts to choosing a curve of canonical type $D\subset X$.
For   details, see \cite{Cossec; Dolgachev 1989}, Chapter III, Section 1 and \cite{Schroeer 2023}, Section 3.

\subsection{Ruled surfaces and Albanese maps}
\mylabel{subsec:ruled and albanese}
Fix a ground field $k$ of characteristic $p\geq 0$. Following \cite{Laurent; Schroeer 2024},  a proper scheme $P$ is called  \emph{para-abelian variety} if for some field extension $k\subset k'$, the base-change
$P'=P\otimes k'$ admits a group law that turns the scheme into an abelian variety.  According to loc.\ cit.\, Theorem 10.2
each proper scheme $X$ with $h^0(\O_X)=1$ comes with  
a morphism $X\ra\Alb_{X/k}$ that is universal for morphisms to para-abelian varieties. One calls $\Alb_{X/k}$ the \emph{Albanese variety}, and $X\ra\Alb_{X/k}$ the \emph{Albanese map}.
Note that  this  is functorial in $X$,   equivariant with respect to $\Aut_{X/k}$, and commutes with base field extension $k\subset k'$. 
Once a rational point $e$ is chosen, which  is always possible if the ground field is  separably closed or finite, one may view the para-abelian variety $P$ as abelian variety $A=(P,e)$, and then has $\Aut_{P/k}=A\rtimes\Aut_{A/k}$. 

Now let $X$ be a smooth proper surface with $h^0(\O_X)=1$. A \emph{ruling} is a morphism $f:X\ra C$ to a curve with $\O_C=f_*(\O_X)$ such that all geometric fibers are projective lines. It then follows that the induced map 
on Albanese varieties is an isomorphism. If $h^1(\O_C)=h^1(\O_X)$ is non-zero, we conclude that
$f:X\ra C$ can be seen as the Stein factorization of the Albanese map $X\ra\Alb_{X/k}$, hence the ruling is an intrinsic property of the surface.
If the ruling admits  sections, we have $X=\PP(\shE)$ for some locally free sheaf $\shE$ of rank two on $C$,
and the sections $B\subset X$ correspond to invertible quotient $\shL=\shE/\shN$, via $B=\PP(\shL)$.
Note that $B^2=\deg(\shL)-\deg(\shN)$, and that the  intersection $B\cap B'$ of two sections $B\neq B'$ can be identified with the zero-scheme on $C$ for the composite map $\shN\subset \shE\ra \shL'$.

Given a smooth proper curve $C$ of  genus $g=h^1(\O_C)$,  the vector space $\Ext^1(\omega_C^{\otimes-1},\O_C)$
is one-dimensional, and we call the ensuing non-split extension
$$
0\lra \O_C \lra \shF\lra \omega_C^{\otimes-1}\lra 0
$$
the \emph{canonical extension}. The ensuing ruled surface $X=\PP(\shE)$ comes with a distinguished section
$B=\PP(\omega_C^{\otimes-1})$, with selfintersection number $B^2=2-2g$.

Recall that for elliptic curves $B=E$, Atiyah \cite{Atiyah 1957} classified all the  locally free sheaves.
It turns  out that the   above canonical extension $\shF$ is the indecomposable sheaf of rank $r=2$ and
degree $d=0$. For each  rational point $a\in E$, the vector space $\Ext^1(\omega_E(a),\O_E)$ is one-dimensional as well, and we can form the non-split extension $0\ra\O_E\ra\shE\ra\O_E(a)\ra\O_E\ra0$, which is indecomposable 
As already remarked by Atiyah  (\cite{Atiyah 1957}, Section 3.2) the ruled surface $\PP(\shE)$
is isomorphic to the symmetric product $\Sym^2_{E/k}$, 
see also \cite{Arbarello et al 1985}, Chapter VII, Proposition 2.1. Moreover, if $k=\FF_2$ and $E$ is supersingular, we have $F^*(\shE)=\shF(2a)$, according to \cite{Togashi; Uehara 2022}, Lemma 2.12.

\subsection{Elliptic curves over \texorpdfstring{$\FF_2$}{F2}}
\mylabel{subsec:elliptic curves F2}
Over any finite field, there are only finitely many Weierstra\ss{} equations, hence only finitely many
isomorphism classes of elliptic curves. 
For the prime field in characteristic two, the situation is analyzed in 
\cite{Husemoeller 1987}, Chapter 3, Section 6:
Up to isomorphism, there are exactly five
elliptic curves $E_1,\ldots, E_5$ over the prime field $k=\FF_2$, which we may index so that $|E_i(k)|=i$.
The curves $E_2,E_4$ have invariant $j=1$, whereas
the curves $E_1,E_3,E_5$ are supersingular with invariant $j=0$.
The supersingular curves can be described by
\begin{equation}
E_5\colon y^2+y=x^3+x\quadand E_3\colon y^2+y=x^3\quadand E_1\colon y^2+y=x^3+x+1.
\end{equation} 
Using change of coordinates $x=x'+r$ and $y=y'+sx'+t$ with coefficients  $r,s,t\in\FF_2$, 
one easily checks that the automorphism group is   defined by the
equations $r=s$ and $r(a_4+a_6+1)=0$. From this one gets
$$
E_i(k)\rtimes\Aut(E_i)=\begin{cases}
C_i\rtimes\Aut(C_i)	& \text{if $i=3,5$;}\\
\{\pm 1\}		&\text{if $i=1$,}
\end{cases}
$$
where $C_i$ denotes the cyclic group of order $i\geq 1$.

\section{The untwisted case \texorpdfstring{$C_2$}{C2}}
\mylabel{sec:case c2}
 
The goal of this section is to unravel in the untwisted case $C_2=B_2$  both  geometry and arithmetic of the  Deligne--Lusztig surface 
over the prime field with negativity of the canonical divisor.

\subsection{Formulation of structure result}
\mylabel{subsec:formulation c2}

Let us first review the root system: Working inside $V=\RR^2$ with the standard orthogonal basis $\epsilon_1,\epsilon_2$ we set
\begin{equation}
\label{eq:root systems c2}
\alpha_1=\epsilon_1-\epsilon_2 ,\quad \alpha_2=2\epsilon_2\quadand
\alpha_1^\vee= \epsilon_1-\epsilon_2,\quad \alpha_2^\vee= \epsilon_2.
\end{equation} 
We obtain a root system of type $C_2$ in $V$ with $\Delta=\{\alpha_1,\alpha_2\}$ the system of simple roots. 
Denote by $\Phi$ and $\Phi^\vee$ the sets of roots and dual roots.
The matrices $s_1=(\begin{smallmatrix}&1\\1\end{smallmatrix})$ and 
$s_2=(\begin{smallmatrix}1\\&-1\end{smallmatrix})$ generate the  Weyl group
$W\subset\GL(V)$ of this root system, which  can be seen as the wreath product  $\{\pm 1\} \wr S_2=\{\pm 1\}^2\rtimes S_2$ of signed permutations.
Let $X_*\subset V$ be the lattice generated by the coroots. Then the dual lattice $X^*$ is generated by
$\alpha_1$ and $\frac{1}{2}\alpha_2$. Note that the corresponding pinned reductive group stems from  $G=\Sp_4$.
 
We now fix $s\ge 0$, set $\varphi_*=p^s\cdot\id_{X_*}$ and consider the Deligne--Lusztig datum
\begin{equation}
\label{eq:dld c2}
\DLD=(X^*,\Phi,X_*,\Phi^\vee,\Delta,p,\varphi_*).
\end{equation} 
For   $w = s_1 s_2$, the Deligne--Lusztig surface $X=\bX(w)$ over $k=\FF_p^\alg$ was already studied in~\cite{Deligne; Lusztig 1976} Section 2.4, and in more detail by Rodier \cite{Rodier 2000} Section~7. It turns out that this surface is not minimal, 
and contracts to the smooth surface $Y\subset \PP^3$ defined by the equation  $x_0x_3^q - x_0^qx_3 + x_1x_2^q - x_1^q x_2=0$, where $q=p^s$. The exceptional divisor equals $\bX(s_1)$, a disjoint union of  $q^3+q^2+q+1$ projective lines. For $q=2$, the divisor $Y\subset\PP^3$ is a del Pezzo surface, while for $q=3$ we get a K3 surface. 

According to the table in Section~\ref{subsec:tables} (and the table in the Appendix),  negativity of the canonical divisor for $\bX(w)$  arises only for   $q=2$ and $w=s_2s_1$.

\begin{theorem}
\mylabel{thm:structure c2}
For $p=2$ and $s=1$ and $w=s_2s_1$, the Deligne--Lusztig variety $X=\bar{X}(w)$ over the prime field $k=\FF_2$
is a supersingular K3 surface with Artin invariant $\sigma=1$.
It arises as  the minimal resolution of singularities for the family of   cubic curves over $\PP^1=\Spec k[t]\cup\Spec k[t^{-1}]$
given by the Weierstra\ss{} equation $y^2=x^3+ a_6$ with $a_6=t^5+t^7$.
\end{theorem}

The idea is to exploit the configuration of the irreducible components in the divisors $\bX(s_2)$ and $\bX(s_1)$,
which stems from  the building of the symplectic group $G=\Sp_4$, or equivalently the Tutte--Coxeter graph (see Section~\ref{subsec:connected components}, Example~\ref{ex:sp4 tutte-coxeter}).
The proof requires some preparation, and the theorem will be proved, in consequence to Theorem~\ref{thm:characterization k3} below, at the end of Section \ref{subsec:arithmetic of rational double points}. 

\subsection{Explicit description in terms of symplectic flags}
\mylabel{subsec:explicit symplectic flags}
We consider $V_0 = \FF_q^4$ as a symplectic vector space with the alternating form given by the matrix 
\[
    \begin{pmatrix} & & & -1 \\ & & -1 \\ & 1 \\ 1\end{pmatrix},
\]
(entries $=0$ omitted).
Let $G = {\Sp}_{4, \FF_q}$ be the symplectic group of $V$, $B$ its Borel subgroup of upper triangular matrices, $T$ the diagonal torus.

The ``symplectic full flag variety'' $\mathscr F:=G/B$ can be described as the variety
\[
    \mathscr F = \{ (L\subset U \subset V);\ \dim(L) = 1,\ \dim(U) = 2,\ U = U^\perp \}.
\]
Mapping $(L\subset U)$ to the full flag $(L\subset U\subset L^\perp)$ in $V$ we get an embedding of $\mathscr F$ into the flag variety for $\SL_4$ (which is simply the map induced by the embedding $\Sp_4\to \SL_4$).
Mapping $(L\subset U)$ to $L$ or to $U$, we obtain maps from $\mathscr F$ to $\PP^3$ and to ${\rm Gr}$ (the Grassmannian of maximal totally isotropic subspaces of the fixed symplectic vector space), respectively.

\bigskip
Next we investigate the surface attached to the element $w = s_2 s_1$.

\begin{proposition}
\mylabel{prop:description symplectic flags}
    \begin{enumerate}
        \item
            We have
            \[
                X(s_2 s_1) = \{ (L\subset U)\in \mathscr F;\ L = U\cap \Frob U,\ L\ne \Frob L\}.
            \]
        \item
            We have
            \[
                X(s_2) = \{ (L\subset U)\in \mathscr F;\ L = \Frob L,\ U\ne \Frob U\}.
            \]
        \item
            Consider the map $p_2\colon \overline{X(s_2s_1)}\to {\rm Gr}$ and denote by $S$ its image.  We have
            \[
                S = \{ U\in {\rm Gr};\ \dim(U\cap \Frob(U)) \ge 1 \}.
            \]
            The morphism $p_2$ restricts to an isomorphism
            \[
                \overline{X(s_2s_1)} \setminus \overline{X(s_1)} = X(s_2s_1)\cup X(s_2)\to S \setminus {\rm Gr}(\FF_q),
            \]
            and contracts each irreducible component of $\overline{X(s_1)}$ to a point in ${\rm Gr}(\FF_q)$.
        \item
            The singular locus of $S$ equals ${\rm Gr}(\FF_q)$. Each singularity is a singularity of type $A_1$.
    \end{enumerate}
\end{proposition}

\begin{proof}
    Parts~(i) and~(ii) follow from Lemma~\ref{lem:position symplectic flags}, and for Part~(iii) note that in view of (i) and (ii) each flag lying in $X(s_2s_1)\cup X(s_2)$ can be recovered from its $2$-dimensional component $U$ by $U\mapsto (U\cap \Frob U, U)$. This defines a morphism which is inverse to $p_2$.

    Let us check Part~(iv) by writing down a (local) equation for $S$. The surface $S$ consists of all elements $U\in {\rm Gr}$ such that $\dim U \cap \Frob U \ge 1$, equivalently that $\dim U+\Frob U \le 3$. In the open chart $U_{12}\cong \mathbb A^3$ of ${\rm Gr}$ of subspaces with basis the columns of a matrix of shape
    \[
        \begin{pmatrix}
            1 \\
            & 1 \\
            a & b \\
            c & -a
        \end{pmatrix},
    \]
    the surface $S\cap U_{12}$ is given by the condition
    \[
        \det
        \begin{pmatrix}
            1        & & 1\\
                     & 1       & & 1 \\
            a & b & a^q & b^q\\
            c & -a & c^q & -a^q
        \end{pmatrix}
        =0
    \]
    which gives the equation
    \[
        (a^q-a)^2 + (b^q-b)(c^q-c) = 0.
    \]
    The singular points are precisely the points in $\mathbb A^3(\FF_q)$, and using the action of $G_0(\mathbb F_q)$ on $S\subset {\rm Gr}$ we also obtain the same description on all of $S$. Locally around any fixed rational point, after a change of coordinates the equation becomes $a^2 + bc = 0$ considered in the point $(a, b, c) = (0, 0, 0)$, a rational double point of type $A_1$.
\end{proof}

\bigskip
Computing the canonical bundle $K$ of $\overline{X(s_2s_1)}$ one sees that it is
numerically trivial if $q=2$, cf.~the table in Section~\ref{subsec:tables}. In this case, the Euler
characteristic is $24$, as can be checked using the formula in~\cite{Deligne; Lusztig 1976} Theorem 7.1. It follows that for $q=2$, $\bar{X}(s_2s_1)$ is a K3 surface.

The closed subschemes $\overline{X(s_i)}$, $i=1, 2$, are disjoint unions of $15$ projective lines each. Their intersection behavior is described by the building of $\Sp_{4, \FF_q}$ (see Proposition~\ref{prop:stratification and building}, Example~\ref{ex:sp4 tutte-coxeter}),  or in other words by the Tutte-Coxeter graph. Thus Theorem~\ref{thm:characterization k3} below applies and determines the Deligne--Lusztig surface in this case up to isomorphism; see Theorem~\ref{thm:structure c2}.

\subsection{A particular K3 surface}
\mylabel{subsec:particular k3}
In this section we formulate a structure result for  $K$-trivial surfaces $X$ having a certain configuration 
of twenty-two projective lines whose dual graph 
$\Gamma=\{\gamma_0,\ldots,\gamma_8,\delta_0,\ldots,\delta_8,\rho, \epsilon_0,\epsilon_1,\epsilon_2\}$  
is contained in the  Tutte--Coxeter graph. The graph $\Gamma$ in question is:
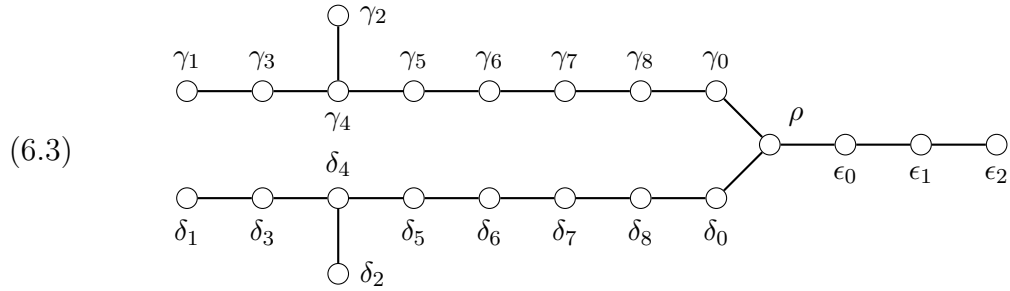
\begin{equation}
\label{eq:graph with 22 vertices}
\begin{gathered}
\begin{tikzpicture}
[node distance=1cm, font=\small]
\tikzstyle{vertex}=[circle, draw, inner sep=0mm, minimum size=1.5ex]
\node[vertex]	(C1)  	at (0,0) 		[label=above:{ $\gamma_1$}] 		{};
\node[vertex]	(C3)			[right of=C1,label=above:{$\gamma_3$}]	{};
\node[vertex]	(C4)			[right of=C3,label=below:{$\gamma_4$}]	{};
\node[vertex]	(C2)			[above of=C4,label=right:{$\gamma_2$}]	{};
\node[vertex]	(C5)			[right of=C4,label=above:{$\gamma_5$}]	{};
\node[vertex]	(C6)			[right of=C5,label=above:{$\gamma_6$}]	{};
\node[vertex]	(C7)			[right of=C6,label=above:{$\gamma_7$}]	{};
\node[vertex]	(C8)			[right of=C7,label=above:{$\gamma_8$}]	{};
\node[vertex]	(C0)			[right of=C8,label=above:{$\gamma_0$}]	{};

\node[vertex]	(R)			[below right of=C0,label=above right:{$\rho$}]	{};

\node[vertex]	(E0)			[right of=R,label=below:{$\epsilon_0$}]	{};
\node[vertex]	(E1)			[right of=E0,label=below:{$\epsilon_1$}]	{};
\node[vertex]	(E2)			[right of=E1,label=below:{$\epsilon_2$}]	{};

\node[vertex]	(D0)			[below left of=R,label=below:{$\delta_0$}]	{};
\node[vertex]	(D8)			[left of=D0,label=below:{$\delta_8$}]	{};
\node[vertex]	(D7)			[left of=D8,label=below:{$\delta_7$}]	{};
\node[vertex]	(D6)			[left of=D7,label=below:{$\delta_6$}]	{};
\node[vertex]	(D5)			[left of=D6,label=below:{$\delta_5$}]	{};
\node[vertex]	(D4)			[left of=D5,label=above:{$\delta_4$}]	{};
\node[vertex]	(D3)			[left of=D4,label=below:{$\delta_3$}]	{};
\node[vertex]	(D1)  	 		[left of=D3,label=below:{$\delta_1$}]	{};
\node[vertex]	(D2)			[below of=D4,label=right:{$\delta_2$}]	{};

\draw [thick] (C1)--(C3)--(C4)--(C5)--(C6)--(C7)--(C8)--(C0)--(R)--(E0)--(E1)--(E2);
\draw [thick] (C4)--(C2);

\draw [thick] (D1)--(D3)--(D4)--(D5)--(D6)--(D7)--(D8)--(D0)--(R);
\draw [thick] (D4)--(D2);
\end{tikzpicture}
\end{gathered}
\end{equation}
Let $A$ be its adjacency matrix, and consider the symmetric matrix  $N=-2E+A$,
where  $E$ is the identity matrix.  It endows the group $ \bigoplus_{v\in \Gamma}\ZZ$ with the structure of an even lattice.
The associated real quadratic space is isometric to a space of the form $(+1)^{\oplus m} \oplus (-1)^{\oplus n} \oplus (0)^{\oplus r}$
for uniquely determined integers $m,n,r\geq 0$ with $m+n+r=22$, and we regard the  tuple $(m,n,r)$ as the  \emph{signature} of the lattice.
Note that $r\geq 0$ is the rank of the radical $\Rad(N)\subset\bigoplus_{v\in \Gamma}\ZZ$. 
One easily checks that for the vectors 
\begin{gather*}
a=2\gamma_1+3\gamma_2+4\gamma_3+6\gamma_4+5\gamma_5+4\gamma_6+3\gamma_7+2\gamma_8+\gamma_0,\\
b=2\delta_1+3\delta_2+4\delta_3+6\delta_4+5\delta_5+4\delta_6+3\delta_7+2\delta_8+\delta_0,
\end{gather*}
the difference $a-b$ belongs to the radical.

\begin{proposition}
\mylabel{prop:signature}
The lattice  $\bigoplus_{\nu\in \Gamma}\ZZ$ has signature $(1,20,1)$. Moreover, its radical $\Rad(N)$ is generated by  $a-b$.
\end{proposition}
 
\proof
Write $(m,n,r)$ for the signature.
The vectors $\gamma_1,\ldots,\gamma_8$ and $\delta_1,\ldots,\delta_8$ and $\rho,\epsilon_1,\epsilon_2,\epsilon_3$ generate
a root lattice of type $-(E_8\oplus E_8\oplus A_4)$. This  is negative-definite, so $n\geq 20$.
The vectors $\rho$ and  $a$ have Gram matrix $(\begin{smallmatrix}-2&1\\1&0\end{smallmatrix})$. This yields
$(\rho+2a)^2=2$, and thus $m\geq 1$.
The vector $a-b$ is non-zero and belongs to the radical, hence  $r\geq 1$.  From $22=m+n+r\geq 20+1+1$ one infers $n=20$ and $m=r=1$.
In light of its coefficient at $\gamma_0$, the element $a-b\in \Rad(N)$ is primitive, hence a generator.
\qed

\medskip
Given a ground field $k$ and a finite \'etale scheme $F$, we write $\ZZ F$ for the ensuing 
sheaf of free abelian groups of rank $d=\deg(F)$ on the category $(\Aff/k)$   endowed with the \'etale topology.
Such sheaves are called  \emph{local systems} of rank $d$, and  $\det(\ZZ F)$ is a local system of rank one.
The corresponding Galois representation $\Gal(k^\sep/k)\ra\{\pm 1\} $ is obtained from the
permutation representation stemming from $F$ by composing with the signum $S_r\ra\{\pm1\}$.
We now can  formulate our result on $K$-trivial surfaces:
 
\begin{theorem}
\mylabel{thm:characterization k3}
Let $X$ be a $K$-trivial  surface over  a ground field $k$ of characteristic $p\geq 0$. Suppose that it  
contains an snc-configuration $\sum_{\nu\in\Gamma}R_\nu$ of twenty-two projective lines with dual graph $\Gamma$.
Then the following holds:
\begin{enumerate}
\item 
The characteristic must be $p=2$.
\item 
The surface $X$ is a supersingular K3 surface  with Artin invariant $\sigma=1$.
\item 
The local system  of Picard groups sits in an exact sequence 
$$
0\lra \operatorname{Rad}(N)_k\lra (\bigoplus_{\nu\in \Gamma}\ZZ)_k\lra\Pic_{X/k}\lra \det(\ZZ\mu_{3,k})\lra 0 
$$
where the sum  of rank twenty-two is generated by  the curves $R_\nu\subset X$.
\item
The K3 surface $X$ is the minimal resolution of singularities for the family of   cubic curves over $\PP^1=\Spec k[t]\cup\Spec k[t^{-1}]$
given by a Weierstra\ss{} equation $y^2=x^3+\alpha^3(t^5+t^7)$ for some $\alpha\in k^\times$.
\end{enumerate}
Conversely, for each  non-zero $\alpha$ from a field $k$ of characteristic $p=2$, the minimal resolution 
of singularities $X_\alpha$ for the family $y^2=x^3+\alpha^3(t^5+t^7)$ is a K3 surface as above.
Moreover, we have $X_\alpha\simeq X_{\alpha'}$ provided that  $\alpha/\alpha'\in k^{\times 2}$. 
\end{theorem}

So for perfect fields, and in particular for $k=\FF_2$, the K3 surface $X$ is unique up to isomorphism. 
The proof   of the theorem requires some preparation and will be completed towards the end of Section \ref{subsec:arithmetic of rational double points}.
Throughout, $k$ denotes  a ground field of characteristic $p\geq 0$,  
and $X$ is a $K$-trivial surface with an snc-configuration $\sum_{\nu\in\Gamma}R_\nu$ 
as in the theorem. From Proposition \ref{prop:signature} we get $b_2\geq 21$, and immediately see that    $X$ is a K3 surface.

To gain further insight we have to introduce certain genus-one fibrations (cf.~Section~\ref{subsec:k3 surfaces}).
To start with, write $C_i,D_i,E_i,R\subset X$ for the curves corresponding to the vertices $\gamma_i,\delta_i,\epsilon_i,\rho\in\Gamma$.
Then $$F_0=2C_1+3C_2+4C_3+6C_4+5C_5+4C_6+3C_7+2C_8+C_0$$ is a curve of canonical type $\II^*$. 
Let $f:X\ra\PP^1$ be the resulting  genus-one fibration. It is jacobian, because $(F_0\cdot R)=1$.
The connected curves $D_0+\ldots+D_8$ and $E_0+E_1+E_2$ are disjoint from $F_0$, hence belong to certain fibers.
In fact,  $$F_\infty=2D_1+3D_2+4D_3+6D_4+5D_5+4D_6+3D_7+2D_8+D_0$$ is another curve of canonical type $\II^*$, hence equals
some fiber.  Applying an automorphism of $\PP^1$ we achieve
$$
F_0 = f^{-1}(0)\quadand  F_\infty= f^{-1}(\infty)\quadand  E_0+E_1+E_2\subset f^{-1}(1) 
$$
for the rational points $0,1,\infty$ inside the projective line $\PP^1=\Spec k[t]\cup\Spec k[t^{-1}]$.
We will now show in two steps that the geometric fiber over $t=1$ has Kodaira symbol $\I_0^*$.

\begin{proposition}
\mylabel{prop:supersingular}
The K3 surface $X$ is supersingular, and the  geometric fiber over the point $t=1$ has Kodaira symbol 
$\I_4$ or $\I_5$ or $\I_0^*$. 
\end{proposition}

\proof
It suffices to treat the case that $k$ is algebraically closed. Let $\eta\in\PP^1$ be the generic point. 
Recall that $\MW(X/\PP^1)=\Pic^0(X_\eta)$
is called the \emph{Mordell--Weil group}.
It can be identified with the set of sections $P\subset X$ for   $f:X\ra\PP^1$, via the bijection
$P\mapsto \O_X(P-R)|X_\eta$, 
with $P=R$ as the zero element. 

Let  $G=\Reg(X/\PP^1)$ be the \emph{N\'eron model} for  $f:X\ra \PP^1$, which we can define in the situation at hand as the 
complement of the closed subscheme of $X$ defined by the sheaf of first Fitting ideals for $\Omega^1_{X/\PP^1}$.
It carries the structure of a relative group scheme, with the zero section $R$ as neutral element.
For each closed point $s\in\PP^1$ we write $\Phi_s=G_s/G^0_s$ for the
finite \'etale   group scheme of connected components of the fiber $G_s$. We then have   \emph{specialization maps}
$$
\MW(X/\PP^1) = G_\eta\lra \Phi_s
$$
sending $\shN=\O_X(P-R)$ to the  fiber  component of   $G_s$ hitting the section $P$.

Write $m,\rho$  for the rank of the Mordell--Weil group and the Picard group of $X$, respectively.
For each closed point $s\in \PP^1$, let $r_s=\ord(\Phi_s)$ be the number of irreducible components of $f^{-1}(s)$.
Since a fiber cannot equal a chain of $3$ projective lines, we have $r_1\geq 4$. Expressing $\rho$ by the Shioda--Tate formula, we find that our quantities are related by 
$$
22=b_2\geq \rho =  m  + 2 + \sum_s (r_s-1) = m+17+r_1 + \sum_{s\neq 0,1\infty}(r_s-1)\geq 17+r_1,
$$
and thus  $r_1\leq 5$. The only possible Kodaira symbols   with four or five components are   $\I_4$ and $\I_5$ and $\I_0^*$.

The above  also implies  $22=b_2\geq \rho\geq 17+4=21$,
and if $r_1=5$ we must have $ \rho=22$. 
In the remaining case $f^{-1}(1)$ has Kodaira symbol $\I_4$. Then the genus-one fibration $f$ is elliptic,
for example by \cite{Cossec; Dolgachev 1989}, Chapter V, Lemma 5.2.3. 
If the formal Brauer group $\widehat{\Br}(X)$ of our elliptic K3 surface $X$ is isomorphic to $\widehat{\GG}_a$, we have $\rho=22$ by \cite{Artin 1974},
Theorem 1.7.
Otherwise, the formal Brauer group  has some finite height $h\geq 1$. By loc.\ cit., Theorem 0.1 we then have $\rho\leq b_2-2h\leq 20$,
contradiction.  
\qed

\medskip
Let $Y\ra\PP^1$ be the \emph{Weierstra\ss{} model} of $X$,  obtained by contracting all vertical curves disjoint
from the zero section $R\subset X$, and write  $\PP^1=\Spec k[t]\cup\Spec k[t^{-1}]$. Arguing as in \cite{Schroeer 2023}, Section 6  we find polynomials $a_i\in k[t]$ of degree $\deg(a_i)\leq 2i$
so that  the Weierstra\ss{} equations
\begin{equation}
\label{eq:weierstrass equations}
 \begin{aligned}
     y^2+a_1xy+a_3y & = x^3 + a_2x^2 + a_4x +a_6,\\
     y^2+a'_1xy+a'_3y & = x^3 + a'_2x^2 + a'_4x +a'_6  
 \end{aligned}
\end{equation}
with $a'_i=a_i/t^{2i}$ describe the Weierstra\ss{} model $Y$ over the charts with coordinate rings $k[t]$ and $k[1/t]$. The ensuing discriminants $\Delta_t$ and $\Delta_{1/t}$
are related by $\Delta_t=t^{24}\Delta_{1/t}$, and thus define a global section for the invertible sheaf $\O_{\PP^1}(24)$.
Note that for each closed $s\in \PP^1$, the   Weierstra\ss{} equation is relatively minimal with respect to the 
discrete valuation ring $\O_{\PP^1,s}$.
 
\begin{proposition}
\mylabel{prop:fiber t=1}
The  geometric fiber over  $t=1$ has Kodaira symbol $\I_0^*$.
\end{proposition}

\proof
It suffices to treat the case that $k$ is algebraically closed.
Seeking a contradiction, we assume that the Kodaira symbol is not $\I_0^*$. By Proposition \ref{prop:supersingular} it must
be either $\I_4$ or $\I_5$, hence the generic fiber $X_\eta$ is smooth.
In particular, the   discriminant  $\Delta$ is non-zero.
Let  $v_s=\val_s(\Delta)$ be its valuation at the closed point $s\in\PP^1$. We   have 
\begin{equation}
\label{eq:discrimiant divisor and ogg}
24 = \sum_s v_s\quadand v_s =  r_s + \tau_s + \delta_s -1.
\end{equation}
The former holds because the discriminant divisor $\operatorname{div}(\Delta)\subset\PP^1$ has degree $24$.
The latter is Ogg's Formula, where  $\tau_s\geq 0$ is the tame part, and $\delta_s\geq 0$ is the wild part of the   conductor exponent.
Note that the tame part is given by
$$
\tau_s=\begin{cases}
0	& \text{if the fiber is smooth;}\\
1	& \text{if the fiber is semistable;}\\
2	& \text{if the fiber is unstable.}
\end{cases}
$$
The wild part is a representation-theoretic invariant that may be non-trivial only for unstable fibers in characteristic  two and three.
From \eqref{eq:discrimiant divisor and ogg} and using that the fibers over $t=0$ and $t=\infty$ are unstable with $9$ 
components each, and the fiber over $t=1$ by our assumption is semistable with at least $4$ components, we get
$$
24\geq \sum_{s=0,\infty,1} v_s= 24 + \delta_0+\delta_\infty+ (r_1-4).
$$
It follows that the fibers outside  $s\neq 0,1,\infty$ must be smooth,   that the configuration of singular fibers
is $\II^* + \II^* +\I_4$, and that $\delta_0=\delta_\infty=0$. From $22=  m  + 2 + \sum_s (r_s-1)$
we see that   $\MW(X/\PP^1)$ has rank $m=1$.

The lattice $\Pic(X)$ is generated by  the sections and the  vertical curves for the fibration $f:X\ra\PP^1$.
The curves $C_1,\ldots, C_8$ and $D_1,\ldots,D_8$ are disjoint from all sections since they occur with multiplicity $>1$ in a fiber. They generate a unimodular sublattice  
of type $(-E_8)^{\oplus 2}$. Its orthogonal complement $L\subset \Pic(X)$ has rank $r=6$, and is generated by
the sections $P$ and the vertical curves $E_0,E_1,E_2,E_3$ that form the fiber with Kodaira symbol $\I_4$.
We identify  the  group of components $\Phi_1$ with $\ZZ/4\ZZ$, such that $E_i$ corresponds to   the congruence class $i+4\ZZ$.

Consider the sublattice $L_0\subset L$ generated by $E_0,\ldots,E_3,R$. This is a lattice of type $ -A_3\oplus U$,
which has rank $r_0=5$ and discriminant $d_0=4$. Its primitive closure $L_0'$
is an even lattice satisfying  $4=[L'_0:L]\cdot\disc(L'_0)$.
Now if  the inclusion $L_0\subset L'_0$ is not an equality,  $L'_0$ would be an even unimodular lattice of odd rank.
But this does not exist, see for example \cite{Serre 1973} V.2 Corollary~2.
Thus  $L_0$ is primitive in $L$.  

Choose a section $P\subset X$ so that $ E_0,\ldots,E_3,R,P\in L$ form a basis. Set $c=(P\cdot R)$, and let
$E_n\subset f^{-1}(1)$ be the irreducible component that intersects the additional section $P$.
By symmetry, it suffices to treat the cases  $0\leq n\leq 2$.
The Gram matrix, written with Kronecker deltas, takes the form
\begin{equation}
\label{eq:gram matrix}
S=\begin{pmatrix}
-2 		& 1 		& 0 		& 1 	& 1 	& \delta_{n0}\\
1 		& -2 		& 1 		& 0 	& 0 	& \delta_{n1} \\
0 		& 1 		& -2		& 1 	& 0 	& \delta_{n2}\\  
1 		& 0 		& 1 		& -2 	& 0 	& 0\\  
1 		& 0 		& 0 		& 0 	& -2 	& c\\ 
\delta_{n0} 	& \delta_{n1} 	& \delta_{n2} 	& 0 	& c 	& -2 
\end{pmatrix}.
\end{equation}
Computing the determinant we get
\begin{equation}
\label{eq:three discriminants}
\disc(L)=\det(S) = \begin{cases}
-(8c + 16) 	& \text{if $n=0$;}\\
-(8c + 13) 	& \text{if $n=1$;}\\
-(8c + 12) 	& \text{if $n=2$.}\\
\end{cases} 
\end{equation} 
Suppose $n=1$.
The invertible sheaf $\shN=\O_X(P-R)|X_\eta$ generates the Mordell--Weil group, and its image in the component group
$\Phi_1=\ZZ/4\ZZ$ is a generator.
Consider the invertible sheaf  $\shN^{\otimes 2}=\O_X(2P-2R)|X_\eta$.
It defines via $\shN^{\otimes 2} = \O_X(P'-R)|X_\eta$  a new section $P'\subset X$, which by construction passes through $E_2$.
Moreover, the difference $(P'-R) - (2P-2R) $ is a sum of vertical curves, which implies that 
$E_0,\ldots,E_3,R,P'$ generate a sublattice $L'\subset L$ of index two.
We saw in \eqref{eq:three discriminants}  that the ensuing Gram matrix has $\disc(L')= -(8c + 12)$.
In turn, we get
$$
-(8c+12) = \disc(L') = [L:L']^2\cdot\disc(L) = -4\cdot(8c+13)
$$
and thus    $3c=-5$, contradiction.
Now suppose $n=2$. Arguing in the same way, we find a sublattice $L'\subset L$ of index two,
now with $\disc(L')=-(8c+16)$, giving  $-(8c+16) =  -4\cdot(8c+13)$ and hence $2c=-3$, again   a contradiction.
 
Consequently we have  $n=0$. Now recall that $\disc(L)=\disc\Pic(X) = -p^{2\sigma}$ for some $1\leq \sigma\leq 10$,
and \eqref{eq:three discriminants} gives  $-2^3(c + 2)=  -p^{2\sigma}$. Thus $p=2$ and $c=p^{2\sigma-3}-2$.
The finitely many possibilities for $\sigma$ and $c$
can be tabulated:
$$
\begin{array}{*{9}l}
\toprule
\sigma	 		& 3	& 4	& 5	& 6	& 7	& 8	& 9	& 10	\\
\midrule
c	 		& 6	& 30	& 126	& 510	& 2046	& 8190	& 32766	& 131070	 \\
\bottomrule
\end{array}
$$
With computer algebra, one sees that for $3\leq \sigma\leq 10$ the   invariant factors for the matrix \eqref{eq:gram matrix}
are given by $d_1=p^2$ and $d_2=p^{4+2(c-3)}$. 
In turn, the discriminant groups are never $p$-elementary. This yields the desired contradiction.
\qed

\begin{proposition}
\mylabel{prop:quasi-elliptic}
The characteristic is $p=2$, the fibration $f:X\ra\PP^1$ is quasi-elliptic, 
the configuration of geometrically reducible fibers is $\II^* + \II^* +\I_0^*$,
and the Artin invariant is $\sigma=1$.
\end{proposition}

\proof
We may assume that $k$ is algebraically closed.
The fibers over $t=0,\infty $ have Kodaira symbol $\II^*$, 
whereas at $t=1$ the Kodaira symbol is $\I_0^*$, according to Proposition \ref{prop:fiber t=1}.
Suppose that $f:X\ra\PP^1$ is elliptic. We then get
$$
24=\sum_s\val_s(\Delta) = \sum_s (r_s+\tau_s + \delta_s-1)\geq 10 + 10 + 6,
$$
contradiction. Thus the fibration is quasi-elliptic, which occurs only in characteristic two and three.
In the latter case, the only possible Kodaira symbols are  $\II,\II^*,\IV,\IV^*$. Thus  $p=2$.
According to \cite{Ito 1994}, Proposition 4.3 , the configuration of 
geometrically reducible fibers must be $\II^* + \II^* +\I_0^*$. 
Inside $L=\Pic(X)$, the  vertical curves and the section   generate  a sublattice $L_0$ of type $(-E_0)^{\oplus 2}\oplus D_4\oplus U$. From $\disc(L_0)=-4$ and $\disc(L)=-2^{2\sigma}$ we infer that the inclusion $L_0\subset L$ is an equality, and $\sigma=1$.
\qed

\subsection{Arithmetic of quasi-elliptic Weierstra\ss{} equations}
\mylabel{subsec:quasielliptic weierstrass}
We continue to work towards the  proof of Theorem \ref{thm:characterization k3}. We already established 
that our K3 surface $X$ admits  a jacobian quasi-elliptic fibration with rather special properties.
To understand these matters in full clarity, it is now necessary to carry out an analysis of  quasi-elliptic Weierstra\ss{} equations
in an \emph{arithmetic setting},
extending   results of Miyanishi \cite{Miyanishi 1977} and  Ito \cite{Ito 1994}. The material appears to be of foundational
nature and  independent interest.

Fix a ground field  $k$  of characteristic $p=2$, which may be imperfect, and consider a Weierstra\ss{} equation
\begin{equation}
\label{eq:long weierstrass equation}
y^2+a_1xy+a_3y = x^3 + a_2x^2 + a_4x +a_6
\end{equation}
over some  $k$-algebra $R$. We write $Y\ra\Spec(R)$ for the resulting family of cubic curves. Note that the family of points at infinity,
with homogeneous coordinates $(0:1:0)$, defines a section for this family.

Each change of variables $x=u^2x'+r$ and $y=u^3y'+su^2x'+t$,
where $u\in R^\times$ and $r,s,t\in R$, produces another Weierstra\ss{} equation. Throughout, we will freely use the formulas in \cite{Deligne 1975}
and \cite{Tate 1975} for the  resulting new coefficients $a_i'$.

\begin{proposition}
\mylabel{prop:vanishing left side}
We have $a_1=a_3=0$ if and only if there is a change of coordinates over some fpqc extension $R\subset R'$ that transforms 
\eqref{eq:long weierstrass equation} into the trivial Weierstra\ss{} equation $y^2=x^3$.
\end{proposition}

\proof
In characteristic two, any change of coordinate yields $ua_1'=a_1$ and $u^3a_3'= a_3+ra_1$.
In turn, the condition $a_1= a_3=0$ holds over $R$ if and only if there is a change of coordinates over some ring extension $R\subset R'$
with $a'_1=a'_3=0$.

Suppose now that we have $a_1=a_3=0$. Then the effect of any change of coordinates   simplifies to 
\begin{equation}
\label{eq:change of coordinates}
u^2a_2'=a_2+ r+s^2, \quad u^4a_4'= a_4 +  r^2,\quad u^6a_6'=a_6+ ra_4 +r^2a_2+r^3+t^2.
\end{equation}
Clearly, the desired change of coordinates   exists  over the   locally free extension 
$$
R'=R[X,Y,Z]/(X^2-\alpha,Y^2-\beta, Z^2-\gamma),
$$
of degree eight, with suitable   $\alpha,\beta,\gamma\in R$.
\qed

\medskip
We are mainly interested in the case $a_1=a_3=0$. By   \eqref{eq:change of coordinates} 
we see that we additionally may assume $a_2=0$.
So from now on, we will assume that our Weierstra\ss{} equation has the classical short form
$y^2=x^3+a_4x+a_6$, albeit in characteristic two. Then all $b$-values,   $c$-values, and the discriminant $\Delta$ vanish,
with the sole exception of $b_8=a_4^2$. A straightforward computation shows:

\begin{proposition}
\mylabel{prop:transformation coefficients}
The change of variables $x=u^2x'+r$ and $y=u^3y'+su^2x'+t$ preserves the short form if and only if $r=s^2$.
In this case 
$$
u^4a'_4= a_4+s^4, \quad u^6a'_6=a_6+t^2+ s^2(s^4+a_4),\quad u^8b'_8= b_8+s^8.
$$
\end{proposition}

In light of Proposition \ref{prop:vanishing left side}, there are no non-constant polynomials $P(a_4,a_6)$ with coefficients in $R$ that are invariant under
all changes of variables over all fpqc extensions $R\subset R'$.
This changes if  one adds structure to the ground ring:

\emph{From now on, we assume that the ring $R$ is endowed with a  derivation $D:R\ra R$.} So besides $a_4,a_6$
we also have $Da_4,Da_6\in R$.
Following Ito \cite{Ito 1994}, Section 3 but using different notation, we consider the expression
$$
\Psi = a_4 \cdot (Da_4)^2 + (Da_6)^2\in R 
$$
and call it  the \emph{quasi-discriminant} of our Weierstra\ss{} equation $y^2=x^3+a_4x+a_6$. Again a straightforward computation shows:

\begin{proposition}
\mylabel{prop:transformation quasi-discriminant}
A change of variables $x=u^2x'+s^4$ and $y=u^3y'+su^2x'+t$ transforms the quasi-discriminant according to  $u^{12}\Psi'=\Psi$.
\end{proposition}

Note that the formation of $\Psi$ commutes  with base-change to algebras $R'$
endowed with a compatible derivation $D'$, and in particular with localizations $R'=S^{-1}R$.
It also commutes with base changes $R'=R\otimes_{R_0}R'_0$,
where $R_0\subset R$ is the kernel of the derivation.

Recall that our Weierstra\ss{} equation $y^2=x^3+a_4x+a_6$ yields a family of cubic curves $Y\ra \Spec(R)$.
Note that a change of coordinates gives another family of cubics $Y'$ that, when the embedding into $\PP^2_R$ 
is disregarded, is isomorphic to $Y$.
The locus of non-smoothness $\Sing(Y/R)$, which we view, as in \cite{Fanelli; Schroeer 2020}, Section 2, as a closed subscheme defined by the sheaf of first Fitting ideals
for $\Omega^1_{Y/R}$, is also called the \emph{scheme of cusps}. It is given by the equation $x^2+a_4=0$,
hence constitutes an effective Cartier divisor, with coordinate ring
$R[x,y]/(x^2-a_4,y^2-a_6)$.
If the ring  $R$ is noetherian, we also have the singular  
locus $\Sing(Y)$, which is contained in the union of $\Sing(Y/R)$ and the preimage of $\Sing(R)$.

\begin{proposition}
\mylabel{prop:over fields}
If $R=k$ is a field, then the curve $Y$ is regular if and only if $a_4\not\in k^2$ or $a_6\not\in k^2$.
This is also equivalent to the existence of a derivation $D:k\ra k$ with $\Psi\neq 0$.
\end{proposition}

\proof
Clearly, the sets $\Sing(Y)\subset\Sing(Y/k)$ are singletons  or empty, and the latter is defined as a scheme by $x^2=a_4$ and $y^2=a_6$.
Suppose first $a_4,a_6\in k^2$. Then $\Sing(Y/k)$ has a $k$-rational point, hence  $Y$ is singular.
If $a_6\not\in k^2$, then the effective Cartier divisor defined by $x=0$ is supported on the union of $\Sing(Y/k)$ and the point at infinity.
Since the latter is a Cartier divisor, so is $\Sing(Y/k)$. It has coordinate ring $k[x]/(x^2-a_4)$, which is a field. It follows that $Y$ is regular.
Finally, suppose that $a_4\not\in k^2$, $a_6\in k^2$, and let $\sqrt{a_6}\in k$ be a square root of $a_6$. The effective Cartier divisor  defined by $y=\sqrt{a_6}$ is supported on the union of $\Sing(Y/k)$ and the point $(0, \sqrt{a_6})$, and an analogous argument applies.

If $a_4,a_6$ are squares  then $Da_4$, $Da_6$ and hence $\Psi=0$ vanish for every $D\in\Der(k)$.
For the converse, suppose first that $a_4\not\in k^2$. Then $da_4\in\Omega^1_{k/k^2}$ is non-zero,
hence there is a $p$-basis $\omega_i\in k$, $i\in I$ for the height-one extension $k^2\subset k$ containing $a_4$ 
(\cite{A 4-7}, Chapter V, \S13, No.\ 2, Corollary 2 and Theorem 1).
Choose a derivation $D:k\ra k$ with $Da_4=1$. Then $\Psi=a_4 + (Da_6)^2$ is non-zero, because it does not belong to $k^2$.
It remains to treat the case $a_4\in k^2$ and $a_6\not\in k^2$. Arguing as above, we find a derivation $D$ with $Da_6=1$.
Now $\Psi=(Da_6)^2=1$.
\qed

\begin{proposition}
\mylabel{prop:over power series}
Suppose   $R=k[[\pi]]$ is a ring of  formal power series  over the field $k$, 
equipped with the derivation $D=\partial/\partial\pi$, and that  $a_4=\sum_{n=0}^\infty\lambda_n\pi^n$ and $a_6=\sum_{n=0}^\infty\mu_n\pi^n$.
Then the following are equivalent:
\begin{enumerate} 
\item The base-change $Y\otimes_kk'$ is regular for every finite field extension $k\subset k'$.
\item The coefficients of $a_4, a_6\in R$ satisfy $\lambda_0\lambda_1^2+\mu_1^2\neq 0$.
\item The quasi-discriminant $\Psi\in R$ is a unit.
\end{enumerate}
Moreover, the closed fiber $Y\otimes_Rk$ is regular if and only if $\lambda_0\not\in k^2$ or $\mu_0\not\in k^2$.
\end{proposition}

\proof
First observe that the quasi-discriminant is given by 
$$
\Psi = \left(\sum_{i}\lambda_it^i\right) \left( \sum_{\text{$i$ odd}}\lambda^2_it^{2i-2}\right) + \sum_{\text{$i$ odd}} \mu^2_it^{2i-2},
$$
with constant term $\lambda_0\lambda_1^2+\mu^2_1$. This already gives (ii)$\Leftrightarrow$(iii).
The statement on the closed fiber follows from Proposition \ref{prop:over fields}.

Consider the expression
$f=y^2+x^3+a_4x +a_6$,
and recall that the singular locus of $Y$ is contained in the scheme of cusps $C=\Sing(Y/R)$, which  has coordinate ring $k[[\pi,x,y]]/(x^2-a_4,y^2-a_6)$.
Let $c\in C$ be the closed point. Then the scheme $Y$ is regular if and only if the local ring $\O_{Y,c}$ is regular.
Let $\maxid\subset k[[\pi]][x,y]$ be the maximal ideal  corresponding to the coordinates $\pi=0$ and $x^2=a_4$ and $y^2=a_6$.
Then the cotangent space for $\O_{Y,c}$ coincides with $\maxid/ (\maxid^2, f)$.
Since $\pi\in\maxid$, neither the maximal ideal nor the cotangent space changes if we modify $f$ by elements from $(\pi^2)$.
We thus may assume $a_4=\lambda_0+\lambda_1\pi$ and $a_6=\mu_0+\mu_1\pi$.
This also allows us to work with the subring $k[\pi,x,y]$.
Computing partial derivatives, we see 
\begin{equation}
\label{eq:partial derivatives}
\partial f/\partial y=0\quadand  \partial f/\partial x\in\maxid\quadand (\partial f/\partial\pi)^2 \equiv \lambda_0\lambda_1^2+\mu_1^2,
\end{equation} 
the latter congruence modulo $(\maxid,\pi)$.

We now can establish (i)$\Leftrightarrow$(ii). For this we may assume that $k$ is perfect.
If $\lambda_0\lambda_1^2+\mu_1^2\neq 0$, we see that $\partial f/\partial\pi\not\in\maxid$, hence $f\not\in\maxid^2$ and $\O_{Y,c}$ is regular.
Conversely, if this local ring is regular, so is the residue class ring  $k[\pi,x,y]/(f)$   at $\maxid$.
Since $k$ is perfect, the ring is smooth over $k$. In turn, one of the partial derivatives does  not belong to $\maxid$.
From \eqref{eq:partial derivatives} we see that  $\partial f/\partial\pi\not\in\maxid$, and thus $\lambda_0\lambda_1^2+\mu_1^2\neq 0$.
\qed

\subsection{Arithmetic of rational double points}
\mylabel{subsec:arithmetic of rational double points}

We keep the set-up of the previous section, but now  specialize  to the polynomial ring $R=k[t]$ in some indeterminate $t$,
and assume that the coefficients of the Weierstra\ss{} equation $y^2=x^3+a_4x+a_6$ have
degrees $\deg(a_i)\leq d i$, for some fixed $d\geq 1$.
Note there is an  unfortunate clash of notation,
since $t$ is also a standard symbol in   changes of variables, but we believe this
should cause no confusion.  

With the  change of variables with $u=t^d$ in the ring $k[t^{\pm1}]$ we now actually have a pair of Weierstra\ss{} equations
$$
y^2=x^3 + a_4x+a_6\quadand y'^2=x'^3 + a_4'x'+a'_6,
$$
where  $a'_i= a_i/t^{di}$ in $R'=k[1/t]$. 
This allows to  globalize our setting from ground rings to  the projective line  $\PP^1=\Spec(R)\cup \Spec(R')$. 
First recall
that any pair $(f_0,f_1)\in R\times R'$  satisfying  the relation $f_0=t\cdot f_1$ 
defines a global section of $\O_{\PP^1}(1)$.
We now consider $\shL=\O_{\PP^1}(-d)$.
In light of the relation $a_i=t^{di}a'_i$  the   pairs $(a_i,a'_i)$ define global sections of $\shL^{\otimes- i}$.
Let $Y\ra\PP^1$ be the family of cubics defined by the above Weierstra\ss{} equations,   inside
the relative homogeneous spectrum of $\Sym^\bullet(\O_{\PP^1}\oplus \shL^2\oplus \shL^3)$. 
The scheme $Y$ is locally of complete intersection, and in particular Gorenstein.
As before, we have the \emph{curve of cusps} $C=\Sing(Y/\PP^1)$, defined by the sheaf of first Fitting ideals for $\Omega^1_{Y/\PP^1}$, which
is locally given by $x^2+a_4=0$ and $x'^2+a'_4=0$ and contains $\Sing(Y)$.

Note that the standard derivations $D_t=\partial/\partial t$ and $D_{t^{-1}}=\partial/\partial t^{-1}$ are not compatible.
Rather, they are related by $D_t = t^{-2}D_{t^{-1}}$. For the ensuing quasi-determinants 
$\Psi_t$ and $\Psi_{t^{-1}}$ it follows that $ \Psi_t=t^{12d-4}\Psi_{1/t}$,
 whence the pair $(\Psi_t,\Psi_{1/t})$ defines a global section 
$\Psi$ of the invertible sheaf $\O_{\PP^1}(12d-4)$. 

Suppose for the moment that $\Psi\neq 0$. It then defines 
an effective Cartier divisor $\operatorname{div}(\Psi)=\sum\val_s(\Psi)\cdot s$ of degree $12d-4$ called the \emph{quasi-discriminant divisor}.
According to Proposition \ref{prop:over fields}, the generic fiber $Y_\eta$ is regular, and thus a  \emph{quasi-elliptic curve}  over the function field $F=k(t)$.
This also ensures that the surface $Y$ is normal. 
According to Proposition \ref{prop:over power series}, the singular locus $\Sing(Y)$ lies
over the quasi-discriminant divisor.
Let  $h:X\ra Y$ be the minimal resolution of singularities. 
Then  $R^1h_*(\O_X)$ is a coherent sheaf supported by $\Sing(Y)$.
If it vanishes,  
the singularities on $Y$ are \emph{rational Gorenstein singularities}, or equivalently \emph{rational double points}.
Note that we allow $k$ to be imperfect, so $Y$  need not be geometrically normal, and $X$ need not be geometrically regular.

\begin{proposition}
\mylabel{prop:classification by value}
Suppose     $d=2$, and  that  $\Psi\in H^0(\PP^1,\O_{\PP^1}(20))$ is non-zero with
\begin{equation}
\label{eq:estimates valuations}
\val_0(\Psi),\val_\infty(\Psi)\geq 8\quadand \val_1(\Psi)\geq 4.
\end{equation}
If $k$ is perfect, then  after a  change of variables with $u=1$ we obtain  $a_4=\lambda_2t^2+\lambda_6t^6$ and $a_6=\mu(t^5+t^7)$
with certain $\mu\in k^\times$ and $\lambda_2,\lambda_6\in k$.
\end{proposition}

\proof
First note  that the  divisor $\operatorname{div}(\Psi)$ has degree $12d-4=20$.
In turn, the  inequalities \eqref{eq:estimates valuations} are actually equalities, and for    $s\neq 0,1,\infty$
we have $\val_s(\Psi)=0$.  
Thus the quasi-discriminant takes the form $\Psi_t=\epsilon(t^{12}+t^8)$ for some $\epsilon\in k^\times $.
Now write $a_4=\sum_{i=0}^8\lambda_it^i$ and $a_6=\sum_{j=0}^{12}\mu_jt^j$. Using Proposition \ref{prop:transformation coefficients},
we may assume that $\lambda_i=0$ and $\mu_j=0$ for $4\mid i$ and $2\mid j$. So $\Psi_t=a_4(Da_4)^2+(Da_6)^2$
takes the form 
$$
(\lambda_1t  + \ldots+\lambda_7t^7)(\lambda_1^2+\lambda_3^2t^4 +\ldots+\lambda_7^2t^{12}) + (\mu_1^2+\mu_3^2t^4 +\ldots + \mu_9t^{16} + \mu_{11}^2t^{20}).
$$
The support of the product on the left does not contain the degrees $d=0,4,16,20$, and the support for the second summand does not contain $d=1,19$.
Comparing coefficients with $\Psi_t=\epsilon(t^{12}+t^8)$,
we see that $\lambda_i^3,\mu_j^2$ vanish for  for $i=1,7$ and   $j=1,3,9,11$. So the same holds for the corresponding $\lambda_i,\mu_j\in k$.
The quasi-discriminant simplifies to  
$$
\Psi_t = (\lambda_2t^2+\lambda_3t^3+\lambda_5t^5+\lambda_6t^6)(\lambda^2_3t^4+\lambda^2_5t^8) + (\mu_5^2t^8+\mu_7^2t^{12})
$$
Arguing as above with $d=7$ we see that $\lambda_3$ vanishes. Then using  $d=13$ we get that $\lambda_5$   vanishes as well.
Thus  $ \Psi_t=\mu_5^2t^8+\mu_7^2t^{12}$. Comparing coefficients with $\Psi_t=\epsilon(t^{12}+t^8)$
yields   $\mu_5=\mu_7$.
\qed

\medskip
%
Suppose now that our Weierstra\ss{} equation $y^2=x^3+a_4x+a_6$ is given by  some polynomials of the form
$$
a_4=\lambda_2t^2+\lambda_6t^6\quadand a_6=\mu(t^5+t^7)
$$
with  $\mu\in k^\times$ and $\lambda_2,\lambda_4\in k$, where $k$ may be imperfect.
The quasi-discriminant is $\Psi_t=\mu(t^{12}+t^8)$.
Computing the partial derivatives, we see that $\Sing(Y/k)$ indeed comprises three points $y_0,y_1,y_\infty\in Y$,
with respective images $0,1,\infty\in\PP^1$.
The point $y_0$ is rational, with coordinates $t=x=y=0$. The point $y_\infty$ is rational as well, with coordinates
$1/t=x'=y'=0$. 
The  remaining   $y_1$ has $t=1$ and $y=0$ and $x^2=\lambda_2+\lambda_6$,  which is not necessarily a rational point.

\begin{proposition}
\mylabel{prop:rdp regular}
The local ring $\O_{Y,y_1}$ is regular if and only if $\lambda_2+\lambda_6\in k$ is not a square.
\end{proposition}

\proof
Developing $a_4,a_6\in k[t]$ in terms of   $u=t-1$ we obtain
$$
a_4= \lambda_6u^6+ \lambda_2u^4 + (\lambda_6+\lambda_2)u^2+ (\lambda_6+\lambda_2)\quadand 
a_6= \mu(u^7+u^6+u^3+u^2).
$$
By Proposition \ref{prop:over fields}   the fiber $Y\otimes\kappa(1)$ and hence the local ring $\O_{Y,y_1}$  are regular provided   $\lambda_2+\lambda_6\not\in k^2$.
Conversely, if $\lambda_2+\lambda_6\in k^2$, then $y_1\in Y$ is a rational point belonging to $\Sing(Y/k)$, so the local ring
$\O_{Y,y_1}$ is singular, by \cite{Fanelli; Schroeer 2020}, Corollary 2.6.
\qed

\medskip
Suppose  that $y_1\in Y$ is singular, in other words $\omega=(\lambda_2+\lambda_6)^{1/2}$ belongs to $k$. Then $y_1\in Y$ is the rational point with coordinates
$t=1$ and $x=\omega$ and $y=0$, and the local ring $\O_{Y,y_1}$ is singular.
Developing the Weierstra\ss{} equation $y^2=x^3+a_4x+a_6$ in terms of $u=t-1$ and $v=x-\omega$, we see that the 
defining polynomial is
$$
y^2 + v^3 +  v(\lambda_6u^6+ \lambda_2u^4+\omega^2u^2)  + \mu(u^7+u^6+u^3+u^2) + \omega(\lambda_6u^6+ \lambda_2u^4+\omega^2u^2).
$$
It turns out that the  quadratic and cubic parts  
\begin{equation}
\label{eq:quadratic and cubic}
Q(u, y)= y^2+(\mu+\omega^3) u^2\quadand P(u,v)= v^3+ \omega^2u^2v + \mu u^3.
\end{equation}
govern the type of the singularity. By abuse of notation, we write their dehomogenizations   as
 $Q(T)=T^2+\mu+(\lambda_2+\lambda_6)^{3/2}$
and $P(T)=T^3+(\lambda_2+\lambda_6)T+\mu$. The latter is separable; it  is convenient
to call it \emph{semi-split} if it is the product of a linear  and an irreducible quadratic factor.
 
\begin{proposition}
\mylabel{prop:rdp d4}
In the above situation, the   local ring $\O_{Y,y_1}$ is a rational double point of type
$$
A_1,G_2,C_3\quad\text{or}\quad D_4.
$$
If $\mu+(\lambda_2+\lambda_6)^{3/2}\in k$ is not a square the type is $A_1$. Otherwise, the type depends
on the separable cubic $P(T)=T^3+(\lambda_2+\lambda_6)T+\mu$ according to the following table:
$$
\begin{array}{lllll }
\toprule
		& \phantom{123}	& \text{$P(T)$ irreducible}	& \text{$P(T)$ semi-split} 	& \text{$P(T)$ split}\\
\midrule
\text{$Q(T)$ split}	& 		& G_2			& C_3			& D_4\\
\bottomrule
\end{array}
$$
\end{proposition}

\proof 
This  relies on Lipman's explicit description of rational double points (\cite{Lipman 1969}, \S 24).
First note that $P(u,v)$ is square-free, as one sees by computing partial derivatives.
Also note that 
the exceptional divisor $E=\Proj\bigoplus\maxid^n/\maxid^{n+1}$ 
for the blowing-up $Y'=\Bl_{y_1}(Y)\ra Y$ can be seen as  the quadratic curve defined by $Q(u,y)=0$ inside the homogeneous spectrum of $k[u,v,y]$.

We start with the case that $Q(u,y)$ is irreducible, in other words, $\mu+\omega^3\in k$ is not a square.
Then $E$ has a unique singularity, located at the origin of the $v$-chart of the blowing-up. One easily checks that this is not a singularity on $Y'$,
hence $Y'$ is regular. It follows that $\O_{Y,y_1}$ is a rational double point of type $A_1$.
It remains to treat the case that  $Q(u,y)$ is reducible.  In the notation of loc.\ cit., page 265, 
our polynomial $P(u,v)$ corresponds to Lipman's $\overline{G}(U,V)$, and our three table entries are his  cases III a--c.
\qed

\medskip
It remains to understand the singularities $y_0,y_\infty\in Y$.
First note that the situation is symmetric, because the change of variables with $u=t^2$ gives $a'_6=\mu(t^{-5}+t^{-7})$
and $a'_4=\lambda_6t^{-2}+\lambda_2t^{-6}$, the latter with switched coefficients. 
It thus suffices to   concentrate on $y_0\in Y$, which has coordinates $t=x=y=0$. The defining polynomial becomes
\begin{equation}
\label{eq:singularity over origin}
y^2+x(x^2+\lambda_2t^2) +   \mu t^5  + \lambda_6 t^6x+ \mu t^7.
\end{equation}

\begin{proposition}
\mylabel{prop:rdp e8}
The  local ring  $\O_{Y,y_0}$ is a rational double point of type 
$$
C_3,C_5 ,C_7 \quad\text{or}\quad D_8\quad\text{or}\quad E_8.
$$ 
In the special case $\lambda_2=0$ the type is $E_8$. For $\lambda_2\in k^\times$ the 
 precise dependence is    given by the following cumulative conditions:
\begin{enumerate} 
\item 
If $ \lambda_2$ is not a square the type is $C_3$.  Otherwise suppose $ \lambda_2^{1/2}\in k^\times$.
\item 
If $\mu\lambda_2^{-5/2}$ is not a square the type is $C_5$.
Otherwise suppose  $ \mu^{1/2}\lambda_2^{-5/4}\in k^\times$. 
\item 
If the separable quadric $Q(T)=T^2 + T + (\lambda_6\lambda_2^{-6/2} + \mu\lambda_2^{-7/2})$ is irreducible the type is $C_7$. Otherwise the type is $D_8$.
\end{enumerate}
\end{proposition}

\proof
Again we follow Lipman's analysis in \cite{Lipman 1969}, \S 24.
Assume first that $\lambda_2\neq 0$. This will be handled with loc.\ cit., pages 264--267. If $\lambda_2$ is not a square, 
the quadratic part of \eqref{eq:singularity over origin} is semi-split. We are thus  in Lipman's case IIIb,
so $\O_{Y,y_0}$ is a rational double point of type $C_3$.
Suppose now that $\lambda_2^{1/2}\in k^\times$. Expanding \eqref{eq:singularity over origin} in the variables $x,y$ and $s=x+\lambda_2^{1/2}t$ we obtain
\begin{equation*}
\begin{split}
y^2+xs^2	& + \mu\lambda_2^{-5/2}(s^5+s^4x+sx^4+x^5) \\
	& + \lambda_6\lambda_2^{-6/2}(s^6x+s^4x^3+x^2x^5 + x^7)\\
	& + \mu\lambda_2^{-7/2}(s^7+s^6x+s^5x^2+\ldots+sx^6 + x^7).
\end{split}
\end{equation*}
The blowing-up $Y'=\Bl_{y_0}(Y)\ra Y$ contains two singularities, the first being a rational double point of type $A_1$. The second singularity $y'\in Y'$ 
is located on the $x$-chart,
with defining polynomial
\begin{equation*}
\begin{split}
\left(\frac{y}{x}\right)^2+x\left(\left(\frac{s}{x}\right)^2+\mu\lambda_2^{-5/2}x^2\right)	
	& + \mu\lambda_2^{-5/2}(\left(\frac{s}{x}\right)^5 +\left(\frac{s}{x}\right)^4 +\left(\frac{s}{x}\right))x^3 \\
	& + \lambda_6\lambda_2^{-6/2}(\left(\frac{s}{x}\right)^6 +\left(\frac{s}{x}\right)^4 +x^2  + 1)x^5\\
	& + \mu\lambda_2^{-7/2}(\left(\frac{s}{x}\right)^7+\left(\frac{s}{x}\right)^6 +\ldots +\left(\frac{s}{x}\right)+\left(\frac{s}{x}\right)  + 1)x^5.
\end{split}
\end{equation*}
The cubic part is semi-split if  $\mu\lambda_2^{-5/2}$ is not a square, and then $y'\in Y'$ has type $C_3$, consequently $y_0\in Y$ has type $C_5$.
Suppose now   $ \mu^{1/2}\lambda_2^{-5/4}\in k^\times$. To proceed we have to work in the indeterminates $x$ and $y'= y/x+x^2$ and $s'=(s/x + \mu^{1/2}\lambda_2^{-5/4}x)$.
On the $x$-chart of the blowing-up $\Bl_{y'}(Y')\ra Y'$, the defining polynomial takes the form
$$
y'^2 + x \left(\frac{s'}{x}\right)^2 + x^2 \left(\frac{s'}{x}\right)  + (\lambda_6\lambda_2^{-6/2} + \mu\lambda_2^{-7/2})x^ 3 + \ldots
$$
where   the omitted terms have degree   four or higher. The cubic part  contains the factor $x$, and the remaining
quadratic factor is separable. The latter is  irreducible if and only if this holds for the dehomogenization
$Q(T)=T^2 + T + (\lambda_6\lambda_2^{-6/2} + \mu\lambda_2^{-7/2})$.
These are   Lipman's cases IIIb and IIIc.
So if  $Q(T)$ splits, the singularity $y'\in Y'$ is of type $D_4$, and otherwise of type $C_3$. In turn, $y_0\in Y$ is of type $D_8$ or $C_7$, respectively.

It remains to treat the case $\lambda_2=0$. This can be handled with  loc.\ cit., page 267--268: In Lipman's notation, our
equation \eqref{eq:singularity over origin} yields $\alpha=\beta=\rho=0$ and $\sigma=1$. In turn,
we are in Lipman's case IVd, and  the singularity $y_0\in Y$ is a rational double point of type $E_8$.
\qed

\medskip\noindent
\emph{Proof of Theorem \ref{thm:characterization k3}.}
Let $k$ be a ground field, and $X$ be a $K$-trivial surface with an snc-configuration $\sum_{\nu\in\Gamma}R_\nu$ 
as in the theorem.   We already observed that $X$ is a   K3 surface.  In Proposition \ref{prop:supersingular}
and \ref{prop:quasi-elliptic} we saw that  it comes with a quasielliptic fibration $f:X\ra\PP^1$
whose  geometrically reducible fibers occur over $0, \infty,1\in\PP^1$, with respective Kodaira symbols $\II^*+\II^*+\I_0^*$.
Furthermore, the characteristic must be $p=2$, and  the K3 surface $X$ is supersingular with Artin invariant $\sigma=1$.
This settles (i) and (ii). 
 
Let $Y\ra\PP^1$ be the Weierstra\ss{} model for the quasi-elliptic fibration, given by some Weierstra\ss{} equations \eqref{eq:weierstrass equations}.
In light of  to  Proposition \ref{prop:vanishing left side} we have $a_1=a_3=0$, and by \eqref{eq:change of coordinates} may furthermore assume $a_2=0$.
According to \cite{Ito 1994}, Proposition 4.2, the quasi-discriminant $\Psi$ satisfies the conditions of Proposition \ref{prop:classification by value},
hence our Weierstra\ss{} equations takes the form
$$
y^2=x^3+(\lambda_2t^2+\lambda_6t^6)x +\mu(t^6+t^7)\quadand y'^2=x'^3+(\lambda_6t^{-2} +\lambda_2t^{-6})x + \mu(t^{-5}-t^{-7}).
$$
Using that the fibers over $t=0,\infty$ have Kodaira symbol $\II^*$, we obtain $\lambda_2=\lambda_6=0$ from Proposition \ref{prop:rdp e8}.
Let $y_1\in Y$ be the singular point in the fiber over $t=1$. The local ring $\O_{Y,y_1}$ is a rational double point, 
of type either $C_3$ or $D_4$, because the fiber $f^{-1}(1)\subset X$ contains a chain of three projective lines.
According to Proposition \ref{prop:rdp d4}, the polynomial $P(T)=T^3+\mu$ has  a root, and thus $\mu=\alpha^3$ for some $\alpha\in k^\times$.
This gives assertion (iv).

It remains to verify   (iii). Let $F$ be the local system defined by the exact sequence
$$
0\lra \operatorname{Rad}(N)_k\lra (\bigoplus_{\nu\in \Gamma}\ZZ)_k\lra\Pic_{X/k}\lra F\lra 0.
$$
It follows from Proposition \ref{prop:signature} that its stalk $F(k^\alg)$ is a cyclic group.
Let $N_0$ be the sublattice in $N=\bigoplus_{\nu\in \Gamma}\ZZ$ generated by $\gamma_1,\ldots,\gamma_8$ and 
$\delta_1,\ldots,\delta_8$ and $\gamma_0,\rho$.
Being isomorphic to $-(E_8\oplus E_8)\oplus U$, this sublattice is unimodular.
Its orthogonal complement $N_1$ has basis vectors  $\epsilon_1,\epsilon_2$ and $\epsilon_0-a$ and $a-b$.
Here we use the notation in and below  \eqref{eq:graph with 22 vertices}. We see that $N_1/\Rad(N_1)$ is a root lattice of type $-A_3$, with discriminant $d=-4$.
If the image of $N$   in $L=\Pic(X\otimes k^\alg)$ would be imprimitive, 
the primitive closure would be unimodular, so the discriminant group of $L$ would by cyclic,
contradiction.
It follows that   $F(k^\alg)$ is free of rank one.

Let $\xi\in k^\sep$ be a primitive third root of unity. There are two cases:
Suppose first that $\xi\in k$. Then the local system $\det(\ZZ\mu_{3,k})$ is isomorphic to $\ZZ_k$
and the cubic polynomial $P(T)=T^3-\mu$  splits. By Proposition~\ref{prop:rdp d4}, 
the Picard group $\Pic(X)$ has rank $\rho=22$, hence $F$ is isomorphic to $\ZZ_k$.
Finally, assume $\xi\not\in k$. Then $\det(\ZZ\mu_{3,k})$ is  the twisted system $\tilde{\ZZ}_k$ 
corresponding to the quadratic extension $k'=k(\xi)$ and $P(T)$ is semisplit.
Again by Proposition~\ref{prop:rdp d4}, the rational double point $y\in Y$ over $t=1$ has type $C_3$.
In other words, the Galois group $\Gal(k'/k)$ permutes two exceptional curves on the base change $Y\otimes k'$,
and it follows that $F$ is  isomorphic to the    twisted system $\tilde{\ZZ}_k$ corresponding to the quadratic extension $k'=k(\xi)$.
This establishes (iv).

Finally, let $Y_\alpha\ra\PP^1$ be the family of cubic curves defined by the Weierstra\ss{} equation 
$y^2=x^3+\alpha^3(t^5+t^7)$, for some non-zero scalar $\alpha$ from a field $k$ of characteristic $p=2$. 
As observed before Proposition \ref{prop:rdp regular}, the locus $\Sing(Y/k)$ consists of three rational  points $y_0,y_1,y_\infty$,
and we saw in Propositions~\ref{prop:rdp d4} and \ref{prop:rdp e8} that this are rational double points of respective types $E_8,D_4,E_8$.
In turn, the minimal resolution of singularities $X_\alpha\ra Y_\alpha$ is a K3 surface. It inherits a jacobian quasi-elliptic fibration.
The section $R$, together with the exceptional divisors $C_i,D_i,E_i$ form an snc-configuration of projective lines
with dual graph $\Gamma$. Finally,   if $\alpha/\alpha'=u^2$ for some $u\in k^\times$,  a change of coordinate of coordinates
gives  $Y_\alpha\simeq Y_{\alpha'}$, hence also $X_\alpha\simeq X_{\alpha'}$
\qed

\begin{figure}
\begin{tikzpicture}[line width=.4mm]
  \def\numNodes{30}
  
  \foreach \n in {0,...,\numexpr\numNodes-1\relax} {
    \pgfmathsetmacro{\angle}{360/\numNodes * (\n + 7)}
    \node[draw, circle, minimum size=1.5ex, inner sep=0pt] (\n) at (\angle:4cm) { };
     }
    \node at (12*20:4.5cm) {$\gamma_0$}; 
    \node at (12*9:4.5cm) {$\gamma_1$};      
    \node at (12*36:4.5cm) {$\gamma_2$}; 
    \node at (12*8:4.5cm) {$\gamma_3$};      
    \node at (12*7:4.5cm) {$\gamma_4$}; 
    \node at (12*24:4.5cm) {$\gamma_5$};      
    \node at (12*25:4.5cm) {$\gamma_6$}; 
    \node at (12*12:4.5cm) {$\gamma_7$};      
    \node at (12*11:4.5cm) {$\gamma_8$};    	

    \node at (12*28:4.5cm) {$\delta_0$}; 
    \node at (12*17:4.5cm) {$\delta_1$};      
    \node at (12*30:4.5cm) {$\delta_2$}; 
    \node at (12*18:4.5cm) {$\delta_3$};      
    \node at (12*31:4.5cm) {$\delta_4$}; 
    \node at (12*32:4.5cm) {$\delta_5$};      
    \node at (12*33:4.5cm) {$\delta_6$}; 
    \node at (12*34:4.5cm) {$\delta_7$};      
    \node at (12*27:4.5cm) {$\delta_8$};  

    \node at (12*21:4.5cm) {$\rho$}; 

    \node at (12*22:4.5cm) {$\epsilon_0$};      
    \node at (12*15:4.5cm) {$\epsilon_1$}; 
    \node at (12*14:4.5cm) {$\epsilon_2$}; 

  \foreach \n in {0,...,\numexpr\numNodes-2\relax}{
    \pgfmathtruncatemacro{\nextn}{\n+1}
    \draw (\n) -- (\nextn); }
  \draw (29) -- (0);

        	\draw (0) -- (17);
        	\draw (1) -- (22);
    	\draw (2) -- (9);
	\draw (3) -- (26);
	\draw (4) -- (13);
        	\draw (5) -- (18);
    	\draw (6) -- (23);
	\draw (7) -- (28);     	
	\draw (8) -- (15);
    	\draw (10) -- (19);
    	\draw (11) -- (24);
	\draw (12) -- (29);
    	\draw (14) -- (21);
	\draw (16) -- (25);
	\draw (18) -- (17);
    	\draw (20) -- (27);
\end{tikzpicture}
\caption{The Tutte--Coxeter graph}
\label{fig:tutte-coxeter}
\end{figure}
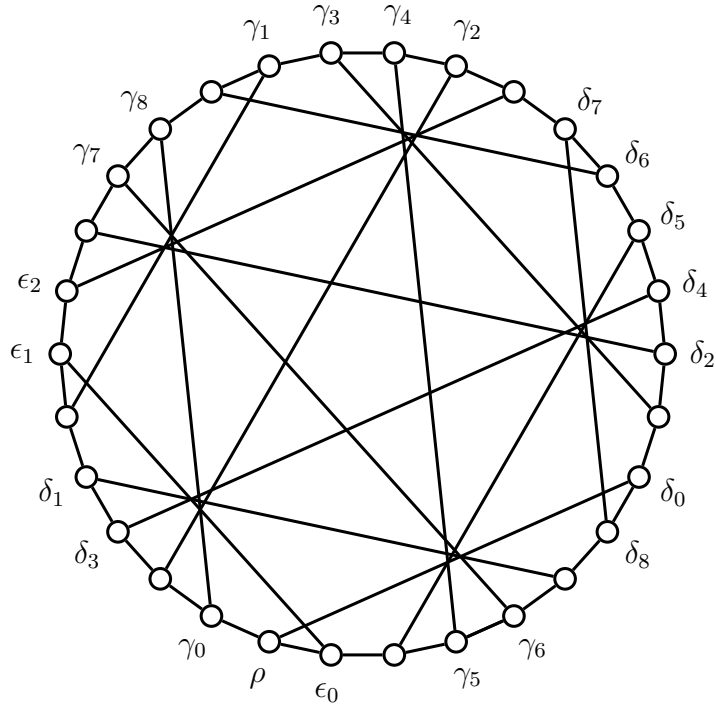

\medskip\noindent
\emph{Proof of Theorem \ref{thm:structure c2}.}
Let $X=\bX(w)$ be the Deligne--Lusztig surface formed with  the Deligne--Lusztig datum \eqref{eq:dld c2} and the Weyl group element
$w=s_2s_1$.
This is a smooth proper surface over the ground field $k=\FF_2$ that is geometrically connected. By the table in Section~\ref{subsec:tables}, the dualizing sheaf
$\omega_X$ is numerically trivial.
Furthermore, the irreducible components of the 
divisors $\bX(s_1)$ and $\bX(s_2)$ form an snc configuration of projective lines  whose dual graph stems from 
the building for the symplectic group $G=\Sp_4$, cf.~Proposition~\ref{prop:stratification and building}.
This is depicted in Figure \ref{fig:tutte-coxeter},  
and coincides with  the  Tutte--Coxeter graph, see Example~\ref{ex:sp4 tutte-coxeter}.
The indicated vertices specify a subgraph 
$\Gamma=\{\gamma_0,\ldots,\delta_0,\ldots,\rho,\epsilon_0,\epsilon_1,\epsilon_2\}$.
Its edges  in the subgraph are as in \eqref{eq:graph with 22 vertices}, hence Theorem \ref{thm:characterization k3} applies.
\qed

\section{The twisted case \texorpdfstring{${^2}A_2$}{2A2}}
\mylabel{sec:case a2twist}

The goal of this section is to understand in  the twisted case ${}^2 A_2$  both  geometry and arithmetic over the prime field
of the  Deligne--Lusztig surfaces whose canonical divisor is negative. We start by briefly reviewing the Deligne--Lusztig datum (of Frobenius type) attached to Dynkin type ${}^2 A_{n-1}$, and then specialize to $n=3$.

\subsection{Formulation of structure result}
\mylabel{subsec:formulation a2twist}
Let us first review the root system: Let $\epsilon_1, \dots,\epsilon_n\in\RR^n$ be the standard orthogonal basis, and set
$$
\alpha_i=\alpha_i^\vee=\epsilon_i-\epsilon_{i+1},\quad i=1, \dots, n-1,
$$
as in the Bourbaki tables \cite{LIE 4-6}.
Let $\Phi=\Phi^\vee$ be the resulting root system of Dynkin type $A_{n-1}$ in the coordinate-sum-zero hyperplane $V\subset\RR^3$,
and write $\Delta=\{\alpha_1, \dots, \alpha_{n-1}\}$ for the system of simple roots. The Weyl group is isomorphic to the group of $n\times n$ permutation matrices, hence $W=\cong S_n$, and we denote by $s_1, \dots, s_{n-1}$ the generating system of simple reflections.
Let $X_*\subset V$ be the lattice generated by the simple coroots $\alpha_1^\vee, \alpha_2^\vee$.
Then the  dual lattice $X^*\subset V$ is generated by $X_*$ together with $\frac{1}{n}((n-1)\epsilon_1-\epsilon_2-\cdots -\epsilon_n)$.  

Fix a prime $p>0$, an integer $d\geq 0$, and set $r=2$ and $s=2d$.
With respect to the     basis $\alpha_1,\dots, \alpha_{n-1}\in V$, the symmetric matrix
$$
p^d\begin{pmatrix}
    0 & 0 & 1 \\ 
    0 & \iddots & 0 \\
    1 & 0 & 0
\end{pmatrix}\in\GL(V)
$$
stabilizes the lattices $X^*\subset X_*$ inside $ V$. The induced maps $\varphi^*$ and $\varphi_*$ satisfy
$$
\varphi^*(\alpha_i)=p^d\alpha_{n-i+1},\quadand 
\varphi_*(\alpha_i^\vee)=p^d\alpha^\vee_{n-i+1},
$$
and thus define a $p$-isogeny of the pinned root datum $(X^*,\Phi,X_*,\Phi^\vee, \Delta)$, obviously satisfying   $(\varphi^*)^r=p^s$.
Summing up, we get a  Deligne--Lusztig datum
\begin{equation}
\label{dld 2A2}
\DLD=(X^*,\Phi,X_*,\Phi^\vee,\Delta,p,\varphi^*).
\end{equation} 
 
This is a Deligne--Lusztig datum of Frobenius type. It is associated, in the sense of the definition of Deligne and Lusztig, to the unitary group over $\FF_{p^d}$, cf.~Example~\ref{ex:unitary DL datum}.
Let us make this explicit.
The root datum $(X^*,\Phi,X_*,\Phi^\vee)$ corresponds to the split reductive group $G=\SL_{n,k}$ over the prime field $k=\FF_2$,
endowed with the   torus $T$ of diagonal matrices and   the Borel group $B$ of upper triangular matrices.
%
%
Now let $q=p^{\frac{s}{r}}=p^d$ and consider
\[
    J = \begin{pmatrix}
        & & 1 \\ & \iddots \\ 1
    \end{pmatrix}\qquad\in M_n(\FF_{q^2}).
\]
Then the morphism
\[
    \Frob\colon \SL_n\to \SL_n,\quad g\mapsto J\, {}^t (g^{(q)})^{-1} J
\]
realizes the isogeny $\varphi$ (for a suitable choice of pinning). In fact, it clearly preserves $T$ and $B$ and has the desired effect on the sets of simple roots and of simple coroots.

\subsection{Explicit description in terms of flags}
\mylabel{subsec:description flags}
Now consider the flag variety $\Flag = \SL_n/B$ of full flags in an $n$-dimensional vector space (over the prime field $\FF_p$). As before, we set $q=p^{s/r} = p^d$. To describe the effect of $\varphi$ on $\Flag$, we define, for $u, v\in R^n$ where $R$ is any $\FF_p$-algebra,
\[
    \langle u, v\rangle = {}^t u^{(q)} J v.
\]
For a direct summand $U\subseteq R^n$ we let
\[
    U^\perp = \{v\in V;\ \forall u\in U: \langle u, v\rangle = 0 \},
\]
a direct summand of $R^n$ of rank $n-{\rm rk}(U)$.

\begin{lemma}\mylabel{lem:2A2 phi on flags}
    The morphism $\varphi\colon \Flag\to \Flag$
    is given, in terms of flags, by 
    \[
        \varphi\colon\quad (0\subset U_1\subset\cdots U_{n-1}\subset V) \ \longmapsto\  (0 \subset U_{n-1}^\perp \subset \cdots \subset U_1^\perp \subset V).
    \]
\end{lemma}

\begin{proof}
    The individual steps of the flags are direct summands of the pertaining rank, thus it is enough to show the inclusion $\subseteq$ for each step of the two flags. After a suitable fpqc base change, we may assume that the flag $(U_i)_i$ is described by a matrix $g\in \GL_n(R)$. Then $\Frob((U_i)_i)$ is given by $\Frob(g)=J\,{}^t (g^{(q)})^{-1} J$, and we need to show that each of the first $i$ columns of $g$ pairs to $0$ with each of the first $n-i$ columns of $\Frob(g)$. But
    \[
        {}^t g^{(q)}\, J\, \Frob(g) = J
    \]
    which shows that the desired property holds.
\end{proof}

\subsection{The case of \texorpdfstring{$\U(3)$}{U(3)}}
\mylabel{subsec:unitary group}

We now specialize to the case $n=3$.
We are interested in the surfaces $\bar{X}(s_1 s_2)$ and $\bar{X}(s_2 s_1)$. Since the situation is symmetric in $s_1$ and $s_2$, they are isomorphic, and we therefore restrict our attention to $\bar{X}(s_1s_2)$. The curves $\bar{X}(s_1)$ and $\bar{X}(s_2)$ embed into $\bar{X}(s_1s_2)$ as divisors intersecting in $X(\id)$. Furthermore, by Deligne--Lusztig reduction, $\bar{X}(s_1 s_2)$ is a ruled surface over $\bar{X}(s_2)$; see Corollary~\ref{cor:structure of ruled surface}.

Let us give an explicit description of the situation.
To this end, we view the flag variety $\Flag = \{ (L\subset U)\}$ with $L$ a line, $U$ a plane as a closed subscheme of $\PP^2\times \PP^2$ via the closed embedding
\begin{equation}
\label{embedding   flag variety}
G/B\lra \PP^2\times\PP^2,\quad (L\subset U)\longmapsto (L,U),
\end{equation} 
where the respective factors parameterize lines and planes. (We identify the projective plane with its dual.) In view of Lemma~\ref{lem:2A2 phi on flags} the embedding is $\varphi$-equivariant, if we let $\varphi$ act on $\PP^2\times \PP^2$ by $(L, U)\mapsto (U^\perp, L^\perp)$. In particular, $\varphi^2 = F^s$ stabilizes each factor and $L^{\perp\perp} = \varphi^2(L)$, and similarly for $U$.

Using Lemma~\ref{lem:relative position explicit}, let us make this explicit for $w =\id, s_1, s_2, s_1s_2$. We work over the prime field $\FF_p$.
We have
\[
    X(\id) = \Flag^\varphi = \{ (L\subset U);\ U= L^\perp,\ L=L^{\perp\perp} \}
\]
which we can rewrite as
\[
    X(\id) = \{ v\in \PP^2;\ \varphi^2(v)=v,\ \langle v, v\rangle =0\}
    = \{ v\in (\PP^2)^{\varphi^2};\ \langle v, v\rangle =0\}
\]
The fix point scheme $(\PP^2)^{\varphi^2}$ is the closed subscheme of $\PP^2_{\FF_2}$ consisting of all points whose residue class field admits an embedding into $\FF_{q^2}$, a finite étale $\FF_p$-scheme.
The set $X(\id)(\FF_{q^2})$ has $q^3+1$ points. This may be seen from the explicit description above or from group theory (cf.~Lemma~\ref{lem:rational Bruhat decomposition}), since the elements in the Weyl group $W \cong S_3$ fixed by $\Frob$ are $\id$ and $s_1s_2s_1$.

For the Deligne--Lusztig curves in this setting, we have
\[
    \overline{X(s_1)} = \{ (L\subset U);\ U=L^\perp\} \isomarrow \{ L\in \PP^2; L\subset L^\perp \},
\]
the curve in $\PP^2$ (the space of lines in $V$) given by the homogeneous equation
\[
    X^qZ + Y^{q+1} + Z^qX = 0.
\]

Similarly,
\[
    \overline{X(s_2)} = \{ (L\subset U);\ L = U^\perp\} \isomarrow \{ U\in \PP^{2};\ U^\perp\subset U\},
\]
a curve in the projective plane of $2$-planes in $3$-dimensional space.

The genus of $\bar{X}(s_1)$ and of $\bar{X}(s_2)$ is equal to $q(q-1)/2$, as follows from the genus-degree formula or from the calculation of the canonical divisor, see the table in Section~\ref{subsec:tables}.

Finally, flags $(L,U)$ and $(L',U')$ have relative position  $w=s_1s_2=(1,2,3)$ if
and only if $L\not\subset U'$ and $L'\subset U$. This yields the description   
$$
\bX(s_1s_2)=\overline{X(s_1s_2)}=\{(L, U)\mid \text{$L\subset U$ and $U^\perp\subset U$}\}
$$
of the functor  of points, embedded into  $\PP^2\times\PP^2$. We see here that the second projection $\bX(s_1s_2)\to \PP^2$ factors through $\bX(s_2)$, giving $\bX(s_1s_2)$ the structure of ruled surface over $\bX(s_2)$. Compare Corollary~\ref{cor:structure of ruled surface}.

\begin{proposition}
\mylabel{prop:inside product planes}
The  closed subscheme $\bX(s_1s_2)\subset\PP^2\times\PP^2$ is  given by the bihomogeneous equations
\begin{equation}
\label{bihomogeneous equations}
x_0y_0+x_1y_1+x_2y_2=0\quadand y_0y_2^{q} - y_1^{q+1} + y_2y_0^{q}=0,
\end{equation} 
where the two    factors in $\PP^2\times\PP^2$ are regarded as   homogeneous spectra of the polynomial rings
$ k[x_0,x_1,x_2]$ and  $k[y_0,y_1,y_2]$, respectively.
\end{proposition}

\proof

First note that in terms of matrices and homogeneous coordinates, the   morphism \eqref{embedding   flag variety} 
sends a point $(\lambda_{ij})B\in G/B$ to the pair with entries
$$
(\lambda_{11}:\lambda_{21}:\lambda_{31})\quadand 
(\mu_{1}:\mu_{2}:\mu_{3}).
$$
Here the former are the coordinates for the line $L$, and $\mu_\bullet$ is the cross product of the first two columns of the matrix $(\lambda_{ij})$,
so that $\mu_{\bullet}$ pairs to $0$ with each of the first two columns of $(\lambda_{ij})$ with respect to the standard dot product.
The first equation in \eqref{bihomogeneous equations} corresponds to the condition $L\subset U$. The second equation then is equivalent to the condition $U^\perp \subset U$. To see this, it is enough to show that $U^\perp$ is the line generated by the vector $J\mu_\bullet^{(q)}$. Indeed, we have $J\mu_\bullet^{(q)} \in U^\perp$, since
\[
    \langle \lambda_{\bullet, j}, J\mu_\bullet^{(q)} \rangle =
    {}^t \lambda_{\bullet, j}^{(q)} J^2 \mu_\bullet^{(q)} = 
    ({}^t \lambda_{\bullet, j} \mu_\bullet)^{(q)} = 0
\]
for $j = 1, 2$, where $\lambda_{\bullet, j}$ is the $j$-th column of the matrix $(\lambda_{ij})$.
\qed

\medskip
From the equations we see once again that the second projection $\bX(s_1s_2)\ra \PP^2$ factors through the plane curve 
$$
C:\quad y_0y_2^{q} - y_1^{q+1} + y_2y_0^{q}=0
$$
of degree $q+1$, and $\bX(s_1s_2)=\PP(\shF)$ for some
locally free sheaf $\shF$ of rank two on $C$. So the latter is the image of the Albanese map for $\bX(s_1s_2)$.
Furthermore, the  Adjunction Formula gives 
$$
\omega_{\bX(s_1s_2)}=\O_{\PP^2\times\PP^2}(-2,q-2).
$$
We are particularly interested in the case $q=2$, i.e., $p=2$ and $d=1$, when the canonical divisor of $\bar{X}(s_1s_2)$ is negative (cf.~Section~\ref{subsec:tables}), and
the plane curves $\bar{X}(s_i)$ have degree $3$ and hence genus $1$.

As mentioned above, we know from group theory that $X(\id)(\FF_2^\alg) = \overline{X(s_i)}(\FF_{4})$ consists of $9$ points. Therefore $\overline{X(s_i)}_{\FF_4}$ is isomorphic to the unique elliptic curve over $\FF_4$ with $9$ rational points (see~\cite{Schoof 1987} Example~5.3). The rational points form the $3$-torsion subgroup and the curve is supersingular.

We can also read off the cardinality of $\bar{X}(s_i)(\FF_2)$ from the description in terms of equations above: It is the number of solutions of the homogeneous equation $X^2Z + Y^{3} + Z^2X = 0$ in $\PP^2(\FF_2)$, and thus equal to $3$. Thus $\bar{X}(s_i)\cong E_3$, the unique supersingular elliptic curve over $\FF_2$ with $3$ rational points; it has $j$-invariant $0$. See Section~\ref{subsec:elliptic curves F2}.

As the following theorem shows, we can describe $\bar{X}(s_1s_2)$ precisely as a specific ruled surface over this elliptic curve.

\begin{theorem}
\mylabel{thm:structure a2twist}
Suppose  $p=2$ and $n=1$. Then  the Deligne--Lusztig variety $X=\bX(s_1s_2)$ has the elliptic curve $E=E_3$ with three rational points as Albanese variety,
and is isomorphic to the symmetric power $\Sym^2(E)$.
\end{theorem}

The symmetric power $\Sym^2(E)$ of an elliptic curve
is isomorphic to $X=\PP(\shF)$ for some locally free sheaf $\shF$ that is indecomposable of rank $r=2$
and degree $d=1$, see Section~\ref{subsec:ruled and albanese}.
Such sheaves are given by the non-split extension $0\ra\O_E\ra \shF\ra \O_E(e)\ra 0$
for some rational point $e\in E$. The group of translations $E(k)$ acts transitively on $E(k)$,
so up to isomorphism the surface $X$ does not depend on the point $e\in E$.
In contrast to $\PP(\shE)$, however, the symmetric power $\Sym^2(E)$ makes sense for para-elliptic curves,
which may lack rational points.

The proof  for the above theorem  requires some preparation, and will be given at the end of Section \ref{subsec:characterization a2twist},
as a consequence of an abstract     characterization of   ruled surfaces containing certain configurations of curves.

\begin{remark}
    The map $X\to \PP^2$, $(L\subset U)\mapsto L$, is finite surjective of degree $q+1$.
\end{remark}

\subsection{Characterization via symmetries and curves}
\mylabel{subsec:characterization a2twist}

We now give a  characterization, in terms of configuration of curves,
of certain ruled surfaces that arise as Deligne--Lusztig surfaces:

\begin{theorem}
\mylabel{thm:characterization symmetric product}
Let  $X$ be ruled over its Albanese image, and write $f:X\ra E$ for the resulting ruling.
Suppose there are para-elliptic divisors $D,D'\subset X$ with 
$$
K_X\equiv -D \quadand (D'\cdot D)\equiv 1 \quad \text{\rm modulo $2$.}
$$
Then the following holds:
\begin{enumerate}
\item
The curve $E$ is para-elliptic and  coincides with $\Alb_{X/k}$.
\item
The  induced  map $D'\ra E$ is an isomorphism, whereas $\deg(D/E)=2$.
\item
The surface $X$ is isomorphic to $\PP(\shE)$ for some   locally free sheaf $\shE$ on $E$
that is indecomposable of rank $r=2$ and degree $d>0$.
\end{enumerate}
Moreover, if $E$ admits  an invertible sheaf of degree two, then there is a rational point $e\in E$, 
one may choose $\shE$ as the  non-split extension of $\O_E(e)$ by the structure sheaf, and furthermore $X\simeq\Sym^2(E)$.
\end{theorem}

\proof
By assumption, the geometric fibers of the Albanese map $f:X\ra E$ are projective lines, hence both para-elliptic curves
$D$ and $D'$ are finite over $E$.  Since $D$ is para-abelian, its image in $\Alb_{X/k}$ must be para-abelian,
by  \cite{Laurent; Schroeer 2024}, Proposition 5.4, and it follows that $E$ is para-abelian.
It thus coincides with $\Alb_{X/k}$, according to the universal property of Albanese maps. 
This already gives (i).

By the Adjunction Formula, $\omega_X$ has degree $d=-2$ on the  geometric  fibers of $f:X\ra E$,
so by our assumption $\deg(D/E)=2$.
Seeking a contradiction, we assume that $D'$ is not a section.
The numerical class of $D'$ takes the form $aD+F$ for some   $F$ that is the preimage of
some divisor on $E$, and the coefficient is some half-integer $a=m/2\geq 1$.
Using  $D^2=F^2=0$  and the Adjunction Formula,  we get  
$$
0=(D'+K_X)\cdot D' = ((a-1)D +F)\cdot(aD+F)= (2a-1) D\cdot F 
$$
So $D\cdot F=0$ because $2a-1\geq 1$.
We get $D'\equiv aD$, which gives $0=aD^2= D'\cdot D$. By assumption, the latter is odd, contradiction.
Thus (ii) holds.

We come to (iii). For this we claim that 
every section $C\subset X$ for the Albanese map $f:X\ra E$ has self-intersection $C^2>0$.
Seeking a contradiction, we assume that $C^2\leq 0$. According to \cite{Kollar 1995}, Chapter II, Lemma 4.12,
the class $[C]\in\Num(X)_\RR$ lies in the boundary of the cone of curves $ \NE(X)$. This has exactly two
extremal rays, which are generated by $D$ and $F$.
We thus have $\RR_{\geq 0}[C]=\RR_{\geq 0}[D]$, and get $D\equiv 2C$.
Then $D\cdot D'=2C\cdot D'$ is even, contradiction.

Since $f:X\ra E$ has a section, we can write $X=\PP(\shE)$ with some locally free sheaf $\shE$
of rank two, which is unique up to tensoring with invertible sheaves. Without loss of generality 
 $0\leq \deg(\shE)\leq 2i-1$ where $i\geq 1$ the positive generator of the degree map $\Pic(E)\ra\ZZ$.
Recall that 
the sections $C\subset X$ correspond to the invertible quotients $\shL=\shE/\shN$   via $C=\PP(\shL)$.
Furthermore, in the exact sequence $0\ra \shN\ra\shE\ra \shL\ra 0$ the degrees $n=\deg(\shN)$ and $l=\deg(\shL)$
are subject to the conditions $n+l=\deg(\shE)$ and $n-l=C^2$, see  \cite{Fanelli; Schroeer 2020}, Lemma 6.1 for the latter.
Summing up, we have 
$$
0\leq n+l\leq 2i-1 \quadand n-l>0.
$$
It follows that $\shE$ is indecomposable: If $\shE=\shN\oplus\shL$ we may swap the summands if necessary, and get $n-l\leq 0$,
contradiction.  
We can also rule out the non-split extension of $\O_A$ by itself: then both degrees $n$ and $l$ vanish, in contradiction to $n-l>0$.
This establishes (iii).

Finally, suppose $E$ admits an invertible sheaf of degree two. Since $(D\cdot D')$ is odd, there is actually some
$\shL$ of degree one. Riemann--Roch and Serre Duality gives $h^0(\shL)=1$, and the effective divisor coming
from a non-zero global section yields a rational point. 
We thus may choose our indecomposable sheaf $\shE$ having degree one. Then $\det(\shE)=\O_E(e)$ for some
unique rational point $e\in E$.
As explained in Section \ref{subsec:ruled and albanese},
the sheaf $\shE$ is isomorphic
to the non-split extension of $\O_E(e)$ by $\O_E$ and the ruled surface $X=\PP(\shE)$ 
is isomorphic to the symmetric product $\Sym^2(E)$.
\qed

\medskip
\emph{Proof of Theorem \ref{thm:structure a2twist}.}
We have already seen that $\bX(s_2)$ is the unique supersingular elliptic curve over $\FF_2$ with $3$ rational points, and that $\bX(s_1s_2)$ is a ruled surface over this curve. The divisors $D = \bX(s_1)$ and $D' = \bX(s_2)$ satisfy the assumptions of Theorem~\ref{thm:characterization symmetric product} in view of the computation of the canonical divisor in Section~\ref{subsec:unitary group} (or in Section~\ref{subsec:tables}) and since $(D'\cdot D) = \lvert X(\id)\rvert = 3$.
\qed

\section{The Suzuki case \texorpdfstring{${^2}C_2$}{twistC2}} 
\mylabel{sec:case c2twist}

\newcommand{\liesp}{\mathfrak{sp}}
\newcommand{\quat}{\text{\rm quat}}
\newcommand{\cent}{\text{\rm sign}}
The goal of this section is to unravel for the root systems ${}^2C_2={}^2B_2$ 
both  geometry and arithmetic of the  Deligne--Lusztig surfaces  that have negativity in the canonical divisor.

\subsection{Formulation of structure result}
\mylabel{subsec:formulation c2twist}

We start with the root system of type $C_2$ as in Section~\ref{subsec:formulation c2}:
Working inside $V=\RR^2$ with the standard orthogonal basis $\epsilon_1,\epsilon_2$ we set
\begin{equation}
\label{eq:root systems c2twist}
\alpha_1=\epsilon_1-\epsilon_2 ,\quad \alpha_2=2\epsilon_2\quadand
\alpha_1^\vee= \epsilon_1-\epsilon_2,\quad \alpha_2^\vee= \epsilon_2
\end{equation} 
and $\Delta=\{\alpha_1,\alpha_2\}$, the system of simple roots. Denote by $\Phi$ and $\Phi^\vee$ the sets of roots and coroots and by $W$ the Weyl group.
Let $X_*\subset V$ be the lattice generated by the simple coroots $\alpha_1^\vee,\alpha_2^\vee$. Then the dual lattice 
$X^*\subset V$ is generated   $\alpha_1,\frac{1}{2}\alpha_2$.

In order to define the Deligne--Lusztig datum of Suzuki--Ree type, we now set
$$
p=2\quadand r=2\quadand s=2n+1 
$$
for some given integer $n\geq 0$. With respect to the  standard basis $\epsilon_1,\epsilon_2\in V$, the symmetric matrix
\begin{equation}
\label{eq:matrix giving isogeny c2twist}
p^n\begin{pmatrix}
1 &  1 \\ 
1 & -1
\end{pmatrix}\in\GL(V)
\end{equation} 
stabilizes the lattices $X_*\subset X^*$. The induced maps $\varphi_*=\varphi^*$ satisfy
$$
\varphi^*(\alpha_1)=p^{n}\alpha_2,\;\varphi^*(\alpha_2)=p^{n+1}\alpha_1\quadand 
\varphi_*(\alpha_1^\vee)=p^{n+1}\alpha^\vee_2,\;\varphi_*(\alpha^\vee_2)=p^{n}\alpha^\vee_1,
$$
and thus define a $p$-isogeny of the pinned root datum $(X^*,\Phi,X_*,\Phi^\vee, \Delta)$, obviously satisfying   $\varphi_*^r=p^s\cdot\id_{X_*}$.
Summing up, we have a  Deligne--Lusztig datum
\begin{equation}
\label{eq:dld c2twist}
\DLD=(X^*,\Phi,X_*,\Phi^\vee,\Delta,p,\varphi^*)
\end{equation} 
for the characteristic $p=2$. 

Fix a field $k$ of characteristic $p=2$  and form the ensuing pinned reductive group $G$,
with maximal torus $T$, Borel group $B$, Weyl group $W=N_G(T)/T$,  and isogeny $\varphi:G\ra G$.
For each $w\in W$ and reduced expression $w=s_{i_1}\ldots s_{i_\ell}$ we then obtain an $\ell$-dimensional Deligne--Lusztig variety $X(w)\subset G/B$ 
with smooth compactification $\bX(w)=\bX(s_{i_1},\ldots, s_{i_\ell})$.
For $\ell=2$ the only possibilities are $w=s_1s_2$ and $w=s_2s_1$; in both cases the canonical divisor has negativity if and only if  $n=0$,
according to the table in Section~\ref{subsec:tables}.

From now on  we assume $n=0$, in other words, the isogeny $\varphi:G\ra G$ satisfies $\varphi^2=F$.
It follows that the finite group scheme $G^\varphi$ is constant, and that 
 $\bX(s_1)$ and $\bX(s_2)$ are isomorphic   curves of genus $g=1$.
The two surfaces are related by a commutative diagram
\begin{equation}
\label{eq:diagram dl surfaces}
\begin{tikzcd} 
\bX(s_1s_2)\ar[r,"\varphi"]\ar[d]\ar[rr,bend left=20,"F"]	& \bX(s_2s_1)\ar[r,"\varphi"]\ar[d]		& \bX(s_1s_2)\ar[d]\\
\bX(s_2)\ar[r,"\varphi"']           				& \bX(s_1)\ar[r,"\varphi"']\ar[u,bend left=40]	&  \bX(s_2)\ar[u,bend right=40],                             
\end{tikzcd}
\end{equation} 
where the horizontal arrows are universal homeomorphisms,   the vertical arrows are rulings stemming from Deligne--Lusztig reductions,
and the sections are given by the inclusions   $\bX(s_1)\subset \bX(s_2s_1)$
and $\bX(s_2)\subset\bX(s_1s_2)$. On all this, the finite group   $G^\varphi(k)$  acts    in a compatible way.

We now seek to describe the  schemes, maps and symmetries in \eqref{eq:diagram dl surfaces} in a completely different
and entirely  explicit way. To this end, let  $E$ be an   elliptic curve, for the moment in arbitrary characteristic $p>0$,
and write $e\in E$ for the origin.
The multiplication map $p:E\ra E$ factors over $E^{(p)}=E/E[F]$, and
the induced projection   $V:E^{(p)}\ra E$, an    isogeny of degree $p$, is  called \emph{Verschiebung}.
We now form the symmetric product $\Sym^p_{E/k}=(E\times\ldots\times E)/S_p$, which by the
Hilbert--Chow map can be identified with   $\Hilb^p_{E/k}$.
The group law induces a morphism
\begin{equation}
\label{eq:addition map}
\Sym^p_{E/k}\lra E,\quad \{a_1,\ldots,a_p\}\longmapsto a_1+\ldots+a_p.
\end{equation} 
Note that the image of the diagonal $\Delta\subset E^p$ under the quotient map is $E^{(p)}\subset\Sym^p_{E/k}$, and the restricting the addition morphism to $E^{(p)}$ gives the Verschiebung.
We now form the commutative diagram
\begin{equation}
\label{eq:diagram symmetric products}
\begin{tikzcd} 
\Sym^p_{E^{(p)}/k}\ar[r,"\can"]\ar[d]\ar[rr,bend left=20, "V"]	& \Sym^p_{E/k}\times_EE^{(p)}\ar[r,"\pr"]\ar[d]		& \Sym^p_{E/k}\ar[d]\\
E^{(p)}\ar[r,"\id"']           				& E^{(p)}\ar[r,"V"']\ar[u,bend left=40]		& E,                   
\end{tikzcd}
\end{equation} 
where the vertical arrows are induced by addition, the upper curved arrow stems from   Verschiebung, and the central curved
arrow is the section induced by the embedding $E^{(p)}\subset\Sym^p_{E/k}$.
While it is also easy to write down sections of the vertical arrow on the right, there does not seem to be a distinguished section. For instance, the map $a\mapsto (a, e, \dots, e)$ is not stable under translations by rational points of $E$.
Note that via the identification $\Sym^p_{E/k} = \Hilb^p_{E/k}$, such sections are nothing but  closed subschemes $C\subset E\times E$ such that
$\pr_1\colon C\ra E$ is finite and flat with $\deg(C/E)=p$, and the geometric fibers $\pr_1^{-1}(y)=\{x_1,\ldots,x_p\}$
satisfy $x_1+\ldots+x_p=y$.

The addition map \eqref{eq:addition map} is 
a bundle of projective $(p-1)$-spaces and
more precisely, $\Sym^p_{E/k}\simeq\PP(\shE)$ 
for each indecomposable sheaf $\shE$ of  rank $r=p$ and degree $d=p-1$,
cf.~Section~\ref{subsec:ruled and albanese}.
From this point of view,    sections can be interpreted as  invertible quotients $\shN=\shE/\shN_0$.

The vertical arrows are Albanese map, and therefore compatible
with the formation of automorphism group schemes (\cite{Laurent; Schroeer 2024}, Corollary 10.3).
Recall that $E\rtimes\Aut_{E/k}$ is the automorphism group scheme of $E$ without the group laws, that is,    as para-elliptic curve.
By functoriality we have an induced action on $\Sym^p_{E/k}=\Hilb^p_{E/k}$. 
The Albanese map is equivariant with respect to $\Aut_{E/k}$, while for the translations  it is equivariant via the multiplication map $p:E\ra E$.
The situation for the base-change  is even more complicated: 
Now the group scheme $\Kernel(p-V)$, defined with the homomorphism $E\times E^{(p)}\stackrel{p-V}{\ra}E$,
acts   on   $\Sym^p_{E/k}\times_EE^{(p)}$ by functoriality. This action is faithful,  and the Albanese map is equivariant with respect
to  $\pr_2:\Kernel(p-V)\ra E^{(p)}$.
Since the kernel of Verschiebung is normalized by automorphisms, we get an identification
$\Aut_{E^{(p)}/k}=\Aut_{E/k}$, and see that this group acts diagonally on the fiber product.

Let us now restrict to the special case that $E$ is a supersingular elliptic curve in characteristic $p=2$.
First note that $E$ being supersingular implies that the Verschiebung is  purely inseparable.
Next  observe  that for each rational point $a\in E$, 
the extension group
$\Ext^1(\O_E(a),\O_E)$ is one-dimensional, and the indecomposable sheaves of rank $r=2$ and degree $d=1$ arise  from the non-split extensions
\begin{equation}
\label{eq:indecomposable r=2 and d=1}
0\lra \O_E\lra \shE\lra \O_E(a)\lra 0.
\end{equation} 
Here the sheaf $\shE$ depends on $a\in E$, but the scheme $\PP(\shE)$ does not.
To simplify notation write $E'=E^{(p)}$ for   Frobenius pullback and $\psi:E'\ra E$ for   Verschiebung, and set 
$$
X=\Sym^p_{E/k}\quadand X'=X\times_E (E',\psi)\quadand X''=X\times_E(E,p).
$$
These   can also be interpreted as    $X=\PP(\shE)$ and $X'=\PP(\psi^*\shE)$ and $X''=\PP(p^*\shE)$. 
What makes $p=2$ so  special is the following:
As explained in \cite{Togashi; Uehara 2022}, Lemma 2.12, the pullback of \eqref{eq:indecomposable r=2 and d=1} under Verschiebung   
also sits in  a non-split extension
\begin{equation}
\label{eq:first  pullback}
0\lra \O_{E'}(a')\lra \psi^*\shE\lra \O_{E'}(a')\lra 0
\end{equation} 
where   $a' =F_{E/k}(a)$ is the Frobenius image. In turn, the ruled surface $X'$  comes with two distinguished sections, namely
$$
E^{(p)}\subset\Sym^p_{E/k}\times_E E^{(p)}\quadand \PP(\O_{E'}(a)),
$$
the first already present in \eqref{eq:diagram symmetric products}, 
the second defined by the invertible quotient \eqref{eq:first  pullback}.
These actually  coincide,  and     have an intrinsic meaning:

\begin{proposition}
\mylabel{prop:distinguished curve}
The above define the same curve $R'\subset X'$, which has the properties
$\O_{R'}(R')=\O_{R'}$   and $\omega_{X'}=\O_{X'}(-2R')$. Moreover, $R'$ is not numerically equivalent to 
another curve on $X'$.
\end{proposition}

\proof
Write  $A,B\subset X'$ for the two sections in question. We first observe $\O_B(B)=\O_{E'}(a')\otimes \O_{E'}(-a')=\O_{E'}$.
The quotient map $E\times E\ra\Sym^2_{E/k}=X$ is unramified outside the diagonal, hence $\omega_X=\O_X(nE^{(p)})$ for some integer $n$.
Actually $n=-1$, because both $K_X$ and $-E^{(p)}$ have degree $d=-2$ on the fibers of the ruling.
The projection $\psi:X'\ra X$ restricts to a closed embedding on $A$, hence $\psi^{-1}(\psi(A))=2A$.
Using $\omega_{X'}=\omega_{X'/E'}=\psi^*\omega_{X/E}=\psi^*\omega_X$ we obtain  $K_{X'}=-2A$.
Seeking a contradiction, we assume $A\neq B$. The Adjunction Formula gives $0=\deg(\omega_B)=(-2A+B)\cdot B= -2A\cdot B$,
thus $A,B\subset X$ are disjoint sections. Hence the extension \eqref{eq:first  pullback} splits, contradiction.
Likewise we see that  $R'=A=B$ is not numerically equivalent to  a curve $C\neq R'$: From $(C\cdot R')=0$
we see that $C$ contains no fibers, hence must be a section disjoint from $R'$, again a contradiction.
\qed

\medskip
In light of the Canonical Bundle Formula $K_{X'}=-2R'$ we call $2R'$ the \emph{canonical curve}, and 
$R' $ the \emph{distinguished section}. The  existence of these curves  puts some constraints on symmetries of the surface $X'$:

\begin{corollary}
\mylabel{cor:distinguished curve stable}
The  subscheme $2R'\subset X'$ is stable with respect to the group scheme $\Aut_{X'/k}$, and the reduction $R'$
is stable with respect to the group $\Aut(X')$. 
\end{corollary}

\proof
The dualizing sheaf $\omega_{X'}=\det(\Omega^1_{X'/k})$ carries a canonical linearization with respect to $\Aut_{X'/k}$,
and the same goes for the dual sheaf.
For each non-zero vector $s\in H^0(X',\omega_{X'}^{\otimes-1})$, the zero-scheme is a double section for the ruling.
By the proposition, from these double sections only $2R'$ is everywhere non-smooth.
In turn, the sheaf of ideals  $\O_{X'}(-2R')\subset \O_{X'}$ and equivalently the closed subscheme  $2R'\subset X'$  are stable.
 The elements $\sigma\in\Aut(X')$ stabilize   also the reduction $R'$.
\qed

\medskip
As explained in \cite{Togashi; Uehara 2022}, Lemma 2.12, the pullback of \eqref{eq:first pullback}
under the relative Frobenius  splits, and the choice of a splitting for the extension
\begin{equation}
\label{eq:second  pullback}
0\lra \O_E(pe)\lra p^*\shE\lra \O_E(pe)\lra 0
\end{equation} 
gives an identification $X''=E\times\PP^1$. 
On this product surface, the preimage of $R'\subset X'$ becomes a fiber, and without  loss of generality we may assume that this fiber
equals $E\times\{(1:0)\}$. Summing up, we have a cartesian diagram
\begin{equation}
\begin{tikzcd} 
E\times\PP^1\ar[r]\ar[d]		& X'\ar[r,"\psi"]\ar[d]	& X\ar[d]\\
E\ar[r,"F_{E/k}"']\ar[u,bend left=40]	& E'\ar[r,"\psi"']\ar[u,bend left=40]	& E,                   
\end{tikzcd}
\end{equation} 
where the respective curved arrows stem from   $E\times\{(0:1)\}$ and the distinguished section  $R'$.
By abuse of notation we use the letter $\psi$ also to denote the first projection on $X'=X\times_EE'$.

The second projection $X\ra E$ is the Albanese map, and thus induces a homomorphism $\Aut_{X/k}\ra\Aut_{E/k}$.
Let us write   $\Aut_{X/k,E}$ for its kernel, and $\Aut(X/E)$ for the group of rational points. Expressed differently, $\Aut_{X/k,E}$ is the $k$-scheme obtained from the relative automorphism scheme $\Aut_{X/E}$ via the composition $\Aut_{X/E}\to E\to \Spec(k)$.
We likewise  form such groups for $X'$ and $X''$. Taking pullbacks gives inclusions
$$
\Aut_{X/k,E}\subset\Aut_{X'/k,E'}\subset\Aut_{X''/k,E}=\PGL_2.
$$
The stabilizer for  $(1:0)\in\PP^1$ is the standard Borel group $\GG_a\rtimes \GG_m\subset\PGL_2$,
and the additive part is the    kernel for the tangent representation. This shows:

\begin{proposition}
\mylabel{prop:relative automorphism group}
In the above inclusions we actually have $\Aut_{X'/k,E'}\subset \GG_a$.
\end{proposition}
 
As already discussed earlier, we also have a commutative diagram
\begin{equation}
\label{eq:all automorphisms}
\begin{tikzcd}[column sep=small]
								& \Aut_{X'/k}\ar[d]				& \Aut_{X/k}\ar[d]\\	
\Kernel(p-V)\rtimes\Aut_{E/k}\ar[r,"\pr_2\rtimes\id"']\ar[ur] 	& E'\rtimes\Aut_{E/k}\ar[r,"\psi\rtimes\id"'] 	& E\rtimes\Aut_{E/k}	& E\rtimes\Aut_{E/k}\ar[l,"p\rtimes\id"]\ar[ul].
\end{tikzcd}
\end{equation} 
Note that the term on the left  is one-dimensional, with two-dimensional Lie algebra.  
By construction, the projection $X''\ra X$ is a torsor with respect to the infinitesimal group scheme $H=E[p]$.
Then $X''\ra X'$ is the partial quotient by the Frobenius kernel, and $X'\ra X$ is the induced quotient by the Frobenius image.
Here $H$  embeds into  $\Aut_{X''/k}=(E\rtimes\Aut_{E/k})\times\PGL_2$ via $h\mapsto (h,\id_E,\iota(h))$
for some monomorphism $\iota:H\ra\PGL_2$.

It is rather surprising that embeddings of $E[p]$ into $\PGL_2$ exist,  compare \cite{Gouthier; Tossici 2024}:
The image is not contained in any Borel group, because the Verschiebung precludes
embedding into the additive group.
However, the assignment $x\mapsto x+ t + sx^2$, where $t$ is arbitrary and $s^2=0$ defines an action of a semidirect product 
$\GG_a\rtimes\alpha_p$ on the affine line $\AA^1=\Spec k[x]$ that extends to the projective line $\PP^1$,
and $E[2]$ is obtained  by setting $t^4=0$ and $s=t^2$. Compare \cite{Hilario; Schroeer 2023} for related situations.

In light of the diagram \eqref{eq:diagram dl surfaces} we expect that in 
certain situations not only $X'\ra E'$, but also    $X\ra E$ should have a distinguished section,
and our next major task is to understand this. To this end consider the set 
$$
\Upsilon=\{A\subset X\mid\text{$A$ is section for $X\ra E$ having  $A^2=1$}\}.
$$
It comes with a map $\Upsilon\ra E(k)$,   $A\mapsto a$   defined by  $\O_A(A)=\O_E(a)$.

\begin{proposition}
\mylabel{prop:sections and rational points}
The above map is bijective, and equivariant with respect to the action of the group $\Aut(X)$.
\end{proposition}
 
\proof
Equivariance is obvious. For surjectivity we start with a rational point $a\in E$, form the
non-split extension $0\ra \O_E\ra\shE\ra\O_E(a)\ra 0$, and make an identification $X=\PP(\shE)$.
Then $A=\PP(\O_E(a))$ defines a section with the property $\O_A(A)=\O_E(a)\otimes\O_E^{\otimes-1}=\O_E(a)$.

For injectivity, suppose we have  $A,B$ giving the same rational point $a\in E$.
Taking the direct images  of the short exact sequence $0\ra\O_X\ra\O_X(A)\ra\O_A(A)\ra 0$,
one sees that $A$ arises as in the preceding paragraph, and the other section becomes
an invertible quotient $\shN=\shE/\shN_0$, say with degrees $d=\deg(\shN)$ and $d_0=\deg(\shN_0)$.
We have $1=B^2=d-d_0$, while additivity of degrees gives $d+d_0=1$, and thus $d=1$ and $d_0=0$.
Write $\shN=\O_E(b)$ and $\shN_0=\O_E(a-b)$. From $\O_E(a)=\O_B(B)= \O_E(b)\otimes\O_E(b-a)$ we see
that $2b$ and $2a$ are linearly equivalent. This implies $b=a$, because  $\Pic(E) $ has no two-torsion,
and therefore $\shN_0=\O_B$. Using $h^0(\shE)=1$ we finally conclude that $\O_E=\shN_0$ as subsheaves
inside $\shE$, and thus $A=B$.
\qed

\medskip 
Let $A\subset X$ be the section  with $\O_A(A)=\O_E(a)$ for the origin $a=e$.  We now 
consider sections $R\subset X$ with $R\cap A=3e$.
Each such section corresponds to a commutative diagram
\begin{equation}
\label{eq:two invertible quotients}
\begin{tikzcd}[row sep=tiny]
0\ar[r]		& \O_E(-2e)\ar[rd] 		& 		& \O_E(e)\ar[r] 	& 0 \\		
		&						& \shE\ar[ru]\ar[rd]\\
0\ar[r]		& \O_E\ar[ur]\ar[rr,dashed]	&  		& \O_E(3e)\ar[r]		& 0
\end{tikzcd}
\end{equation}
where the two sequences from the upper left to the lower right, and the lower right to the upper left are exact, and the dashed arrow is the canonical inclusion.

\begin{proposition}
\mylabel{prop:only two extensions}
Surjective extensions $\shE\ra\O_E(3e)$ of the  injection $\O_E\ra\O_E(3e)$ do exist. For the prime field $k=\FF_2$,
there are exactly two such extensions.
\end{proposition}

\proof
Set $\shL=\O_E(3e)$. Consider the long exact sequence
$$
0\ra\Hom(\O_E(e),\shL)\ra\Hom(\shE,\shL)\ra\Hom(\O_E,\shL)\ra\Ext^1(\O_E(e),\shL).
$$
The term on the right vanishes, hence extensions of any given map $\O_E\ra\O_E(3e)$ to a homomorphism $\shE\ra\O_E(3e)$ exist. They are unique up to maps that factor over $\O_E(e)$, 
in other words, elements of $H^0(E,\O_E(2e))$.
Since $\O_E(2e)$ is globally generated,
we indeed find surjective extensions $h:\shE\ra \shL$. For $k=\FF_2$, the   space $H^0(E,\O_E(2e))$ contains four vectors $s$,
and the extension $h'=h+s$  is non-surjective if and only if $s(e)=1$. Since there are exactly two vectors with $s(e)=0$,
we see that there are exactly two surjective extensions.
\qed

\begin{proposition}
\mylabel{prop:intersection distinguished section}
In the above situation, the ruling gives an identification
$$
\psi(R')\cap R= e+D
$$for some geometrically reduced 
divisor $D\subset E$ that is disjoint from the point $e\in E$, and linearly equivalent to   $4e$.
\end{proposition}

\proof
First note that $\psi:X'\ra X$ restricts to a closed embedding on $R'$, giving an identification $\psi(R')\cap Z=R'\cap\psi^{-1}(Z)$
for each closed subscheme $Z\subset X$.
Pulling back \eqref{eq:two invertible quotients}, we obtain a commutative diagram 
\begin{equation}
\label{eq:three invertible quotients}
\begin{tikzcd}[row sep = small]
0\ar[r]		& \O_{E'}(-4e')\ar[rd]		& 		& \O_{E'}(2e')\ar[r]		& 0 \\		
0		& \O_{E'}(e')\ar[l]					& \psi^*\shE\ar[ru]\ar[rd]\ar[l]	& 	& \O_{E'}(e')\ar[ll]\ar[dl,dashed]\ar[ul,dashed]	& 0\ar[l]\\
0\ar[r]		& \O_{E'}\ar[ur]\ar[rr,dashed]	&  		& \O_{E'}(6e')\ar[r]		& 0
\end{tikzcd}
\end{equation}
with three invertible quotients, where the middle row corresponds to the section $R'$. The ruling identifies the pairwise intersections for the   sections  with the schemes of zeros
for the respective dashed arrows.  
The   arrow $ \O_{E'}(e')\ra \O_{E'}(2e')$ immediately  gives $R'\cap \psi^{-1}(A)=e'$. By construction we   have $\psi^{-1}(R)\cap \psi^{-1}(A)=6e'$,
and thus $e'\in R'\cap \psi^{-1}(R)$. 
The   arrow $\O_{E'}(e')\ra\O_{E'}(6e')$ now reveals     $R'\cap \psi^{-1}(R)=e'+D'$, where the second summand is linearly equivalent to $4e'$. 
To show that the divisor $e'+D'$ is geometrically reduced, we may assume that the base field is algebraically closed. Let $x'$ be a point in the intersection $R'\cap \psi^{-1}(R)$. The map induced by $\psi$ on the tangent space $T_{X', x'}$ has one-dimensional kernel $K$, because $\psi$ is the quotient by a free $\alpha_p$-action. Since $\psi^{-1}(R)$ is a smooth curve, we have $T_{\psi^{-1}(R), x'} = K$. On the other hand, $T_{R', x'}\cap K = 0$, since $R'$ embeds into $X$ under $\psi$. Thus the intersection $R'\cap \psi^{-1}(R)$ is transversal, and $e'+D'$ is geometrically reduced. 
\qed

\medskip
Since $p=2$, our supersingular elliptic curve $E$ has invariant $j=0$, and thus stems from 
a Weierstra\ss{} equation of the form $E:\, y^2+a_3y=x^3+a_4x+a_6$ with $a_3\neq 0$, compare \cite{Tate 1975}.
The quotient by the sign involution can be seen as the morphism 
$$
\epsilon:E\lra E/\{\pm 1\} = \Spec k[x]\cup\Spec k[1/x] =\PP^1,
$$
This is also given by the linear system $H^0(E,\O_E(2e))$, and thus   $\O_E(4e)=\epsilon^*(\O_{\PP^1}(2))$.

If the intersection $\psi(R')\cap R=e+D$ is stable under the sign involution, the same holds for $D$,
and we see that this  geometrically reduced divisor  is the pullback of a divisor defined by an inhomogeneous equation of the form 
$x^2+\lambda x+\mu=0$ with $\lambda\neq 0$.
Note that the vectors in $H^0(E,\O_E(2e))$ that vanish at the origin are precisely the scalars $s\in k$,
and changing the surjection $\shE\ra\O_E(3e)$ by such a vector   changes the constant term $\mu$ 
in the equation  to $\mu'=\mu+s^2$. Over the prime field $k=\FF_2$, this means that for the two sections $R,R'$
with $R\cap A=3e=R'\cap A$, the resulting schemes $D,D'\subset E$ are defined, possibly after swapping them, by
\begin{equation}
\label{eq:equations residual divisor}
\begin{gathered}
D:\, x^2+x=0,\quad y^2+y=(a_4+1)x+a_6, \\
D':\, x^2+x+1=0,\quad y^2+y=a_4x+(a_6+1).
\end{gathered}
\end{equation}
Note that the above discussion applies, regardless of which of the rational points of the para-elliptic curve $E$ is chosen as the neutral element for the group law.

We now exploit that for the elliptic curves over $\FF_2$ one may explicitly compute their automorphism group and that for $E=E_5$, the unique elliptic curve with $5$ rational points, 
the automorphism group is isomorphic to $C_5\rtimes \Aut(C_5)$ and in particular the sign involution does not generate a 2-Sylow group. See Section~\ref{subsec:elliptic curves F2}.

\begin{proposition}
\mylabel{prop:second section}
Suppose $E:y^2+y=x^3+x$ is the elliptic curve with five rational points over  
$k=\FF_2$. Then the canonical map $\Aut(X)\ra E(k)\rtimes\Aut(E)$ is bijective, 
and there  is a unique section $R\subset X$ such that the ruling
gives an identification  $\psi(R')\cap R=E(k)$. Moreover, the section is stable with respect to $\Aut(X)$. 
\end{proposition}

\proof
We start with the existence of the section.
Write $a_1,\ldots,a_5\in E$ for the rational points, and $A_1,\ldots,A_5\subset X$ for the corresponding sections with self-intersection $n=1$.
By Proposition \ref{prop:only two extensions},  there are exactly two sections $R_i\neq R'_i$ satisfying $R_i\cap A_i=3a_i=R'_i\cap A_i$.
In turn, the group $\Aut(X)$ acts via permutations on the sets 
$$
\{A_1,\ldots, A_5\}\quadand \{R_1,R'_1,\ldots, R_5,R'_5\}.
$$
Set $\Gamma=E(k)\rtimes\Aut(E)$, seen as a subgroup of $\Aut(X)$, and write $\Gamma_i$ for the stabilizer   of $a_i\in E$.
Note   that   this is the automorphism group of $E$ regarded as elliptic curve with origin $a_i$,
and we write $\sigma_i\in \Gamma_i$ for the ensuing sign involution.
Since $a_i$ and hence also $A_i$ are stabilized by $\Gamma_i$, the group $\Gamma_i$ acts via permutations on the set $\{R_i,R'_i\}$.
Since the group is cyclic of order four,  both $R_i$ and $R'_i$ must be $\sigma_i$-stable.
From \eqref{eq:all automorphisms} we infer that the 
whole group $\Gamma$ stabilizes the image   of the distinguished section $R'\subset X'$. Writing
$$
\psi(R')\cap R_i = a_i+D_i\quadand \psi(R')\cap R'_i = a_i+D'_i.
$$
we see that the divisors $D_i, D'_i\subset E$ are $\sigma_i$-stable. In turn, the descriptions
\eqref{eq:equations residual divisor} apply, hence the above intersections have coordinate rings $(\FF_2)^5 $ and $\FF_2\times\FF_{16}$,
respectively.
Summing up, we have constructed sections $R_i$ such that the ruling gives identifications $\psi(R')\cap R_i=E(k)$.

We next verify uniqueness.  Suppose we have a section $R$ with $\psi(R')\cap R=E(k)$.
Fix an index $1\leq j\leq 5$,   regard $E$ as elliptic curve with origin $e=a_j$,
write   $X=\PP(\shE)$ for the non-split extension $0\ra\O_E\ra\shE\ra\O_E(e)\ra 0$,
and let $\shN=\shE/\shN_0$ be the invertible quotient corresponding to $R$. Then $\shN=\O_E(Z)$ for some effective divisor $Z$.
The pullback $\psi^*\shE$ now comes with three invertible quotients. This is depicted in  
\begin{equation} 
\begin{tikzcd}[row sep=small]
0\ar[r]		& \shN_0^{\otimes 2}\ar[rd]		& 		& \O_{E}(2e)\ar[r]		& 0 \\		
0		& \O_E(e)\ar[l]					& \psi^*\shE\ar[ru]\ar[rd]\ar[l]	& 	& \O_E(e)\ar[ll]\ar[dl,dashed]\ar[ul,dashed]	& 0\ar[l]\\
0\ar[r]		& \O_{E}\ar[ur]\ar[rr,dashed]	&  		& \O_E(2Z)\ar[r]		& 0
\end{tikzcd}
\end{equation}
where the pairwise  intersections of the sections are the zero schemes of the dashed arrows.
By assumption, the dashed arrow $\O_E(e) \ra \O_E(2Z)$ vanishes precisely along $a_1+\ldots+a_5$.
Then $2Z$ is linearly equivalent to $2a_j+ \sum_{i\ne j} a_i$. The latter is linearly equivalent to $6e$,
as one sees by looking at the quotient   $E\ra E/\{\pm 1\}=\PP^1$.

The composite map $\O_E(e)\ra\psi^*\shE\ra\O_E(2e)$ vanishes precisely at  $e$, and it follows that the curves
$R,A,\psi(R')$ have a common intersection over   $e\in E$. In turn,
the composite map $\O_E\ra\shE\ra\O_E(Z)$ vanishes at $e=a_j$,   but not at the other rational points.
Since $\PP^1$ contains exactly one point with residue field isomorphic to $\FF_4$,
the curve $E$ contains no such point, the map vanishes precisely at $3e$,
and therefore  $R=R_j$. Since the index $j$ is arbitrary, we conclude  $R_1=\ldots=R_5$, and our section $R$ coincides with
this.

It remains to verify the statements on the automorphism group. By uniqueness,  the section $R$ must be stable
with respect to  $\Aut(X)$, and is therefore fixed by $\Aut(X/E)$.
According to Proposition \ref{prop:relative automorphism group}, 
the latter acts freely on the open set $X\smallsetminus\psi(R')$, hence $\Aut(X/E)$ must be  trivial.
In turn, the canonical surjection $\Aut(X)\ra E(k)\rtimes\Aut(E)$ is bijective.
\qed
 
\medskip
Setting   $E :y^2+y=x^3+x$ and    $X=\Sym^2_{E/k}$ over the prime field $k=\FF_2$, we finally arrive at the  commutative diagram
\begin{equation}
\label{eq:diagram ruled surfaces}
\begin{tikzcd} 
X\ar[r,"F_{X/E}"]\ar[d]\ar[rr,bend left=40, "F"]	& X'\ar[r,"\psi"]\ar[d]    		& X\ar[d]\\
E\ar[r,"\id"']           				& E'\ar[r,"\psi"']\ar[u,bend left=40]	& E\ar[u,bend right=40],                             
\end{tikzcd}
\end{equation} 
where $X'$ is the base-change along   Frobenius    $E'=E\stackrel{\psi}{\ra} E$, 
the resulting projection is also denoted by  $\psi:X'\ra X$, and
$F_{X/E}:X\ra X'$ is the relative Frobenius map. The section $R'$  on the left 
stems from the inclusion $E'=E^{(p)}\subset\Sym^2_{E/k}$, while the section   on the right is given by the
unique $R\subset X$ with $\psi(R')\cap R=E(k)$, as established in Proposition \ref{prop:second section}. The group 
$$
C_5\rtimes\Aut(C_5) =E(k)\rtimes\Aut(E)=\Aut(X)
$$
acts on all terms, and all maps are equivariant.
This diagram with the ruled surfaces $X$ and $X'$ indeed has the same form as the original diagram with the  Deligne--Lusztig surfaces 
$\bX(s_1s_2)$ and $\bX(s_2s_1)$, and we  now can formulate the main result of this section:

\begin{theorem} 
\mylabel{thm:c2twist}
Over $k=\FF_2$, the diagram \eqref{eq:diagram   dl surfaces} of Deligne--Lusztig surfaces  
is isomorphic to the diagram \eqref{eq:diagram ruled surfaces}
of ruled surfaces.
In particular, one has 
\[
\bX(s_1s_2)\simeq\PP(\shE)\quadand \bX(s_2s_1)\simeq\PP(\shF),
\]
for the non-split extensions 
\[
0\ra\O_E\ra\shE\ra\O_E(e)\ra 0\quadand 0\ra \O_E(e)\ra \shF\ra \O_E(e)\ra 0.
\]
\end{theorem}

The sheaves $\shE$ and $\shF$ are precisely the indecomposable sheaves of rank
$r=2$, with respective determinant  $\O_E(e)$ and $\O_E(2e)$. Also note that the
two   surfaces are not isomorphic: Otherwise $\shE$ would be isomorphic to
$\shE'=\sigma^*(\shF)\otimes\shN$ for some $\sigma\in\Aut(E)$ and
$\shN\in\Pic(E)$. In contrast to the former, the Frobenius pullback of the
latter becomes decomposable.

The proof requires some preparation, and will be given at the end of  Section
\ref{subsec:characterization b2twist}. We also will  see that  $G^\varphi(k)$ is
a group of order twenty, and identify it with $C_5\rtimes\Aut(C_5)$ via
symplectic matrices.

\subsection{Matrix interpretation of the \texorpdfstring{$p$}{p}-isogeny}
\mylabel{subsec:matrix interpretation}

In this section we take a little detour  and use symplectic matrices to   
make the $p$-isogeny $\varphi:G\ra G$ in the Deligne--Lusztig datum \eqref{eq:dld c2twist} explicit.
This seems  to be of independent interest, and also clarifies the connection to Suzuki groups. 
As before, we set $p=2$, and we fix $n \ge 0$ and write $s = 2n+1$, $q=\sqrt{p^s} = \sqrt{2^s}$.

Write $J$ for the    block matrix $\left(\begin{smallmatrix}0&R\\-R&0\end{smallmatrix}\right)$ 
of size $4\times 4$, for the reverse diagonal matrix
$R= \left(\begin{smallmatrix}0&1\\1&0\end{smallmatrix}\right)$, and recall that the  symplectic group 
$\Sp_{4}\subset\GL_{4}$ for the alternating pairing defined by $J$ comprises the matrices $S$ satisfying 
the equivalent conditions
\begin{equation}
\label{eq:defining conditions symplectic}
{}^tSJS=J \quadand   J\cdot ({}^tS^{-1})\cdot J^{-1} = S.
\end{equation} 
Note that the latter exhibits the symplectic group as a scheme of fixed points under an involution. Also note that
our choice of $R= \left(\begin{smallmatrix}0&1\\1&0\end{smallmatrix}\right)$ follows the Bourbaki convention
\cite{LIE 7-8}, Chapter VIII, \S 13, No.\ 3).
The Lie algebra
$$
\liesp_4(k)=\{T  \mid {}^tTJ+JT=0\}   \subset\gl_4(k).
$$
is 10-dimensional, and comprises the block matrices $T=\left(\begin{smallmatrix}A&B\\C&D\end{smallmatrix}\right)$
with blocks
\begin{equation}
\label{eq:blocks in symplectic lie algebra}
A=\begin{pmatrix}
a&b\\
c&d
\end{pmatrix},\quad
B=\begin{pmatrix}
u&v\\
w&u
\end{pmatrix},\quad
C=\begin{pmatrix}
x&y\\
z&x
\end{pmatrix},\quad
D=\begin{pmatrix}
-d&-b\\
-c&-a
\end{pmatrix}.
\end{equation} 
The standard Borel group $B$ is given by the symplectic upper triangular matrices,
and the standard maximal torus $T$ is formed by the   symplectic diagonal  matrices,
which have the form $S(\lambda_1,\lambda_2) = {\rm diag}(\lambda_1, \lambda_2, \lambda^{-1}_{2}, \lambda^{-1}_{1})$.
Consider the characters $\epsilon_i:T\ra \GG$  given by $S(\lambda_1,\lambda_2)\mapsto\lambda_i$,
and the resulting real vector space $V=\RR\epsilon_1\oplus \RR\epsilon_2$.  As explained in \cite{LIE 7-8}, Chapter VIII \S 13, No.~3,
the vectors $\alpha_1=\epsilon_1-\epsilon_2$ and $\alpha_2=2\epsilon_2$
generate the root system for $\liesp_4(k)$, and coincide with the root system introduced in Section~\ref{subsec:formulation c2twist}.
Moreover, the two copies of $\GG_a$ inside $\Sp_4$ stemming from the matrices 
$$
\begin{pmatrix}
1	& \mu_{1}	&	&\\
	& 1		&  	& \\
	&		& 1	& -\mu_{1}  \\
	&		&	& 1\\
\end{pmatrix}
\quad\quadand\quad
\begin{pmatrix}
1	& 	&		&\\
	& 1	& \mu_{2}	&\\
	&	& 1		&\\
	&	&  		&1
\end{pmatrix}
$$
are the root subgroups for the simple roots $\alpha_{1}$ and $\alpha_{2}$, and  yield a pinning.

So far everything is valid   for the Chevalley group over the ring of integers. To exploit the special features of $p=2$,  
consider  the mapping that sends a $4\times 4$-matrix $S=(\sigma_{ij})$ to the $4\times 4$-matrix
\begin{equation}
\label{eq:explicit isogeny}
\begin{pmatrix}
\sigma_{11}\sigma_{22}-\sigma_{12}\sigma_{21}	& \sigma_{11}\sigma_{23}-\sigma_{13}\sigma_{21}	& \sigma_{12}\sigma_{24}-\sigma_{14}\sigma_{22}	& \sigma_{13}\sigma_{24}-\sigma_{14}\sigma_{23}\\
\sigma_{11}\sigma_{32}-\sigma_{12}\sigma_{31}	& \sigma_{11}\sigma_{33}-\sigma_{13}\sigma_{31}	& \sigma_{12}\sigma_{34}-\sigma_{14}\sigma_{32}	& \sigma_{13}\sigma_{34}-\sigma_{14}\sigma_{33}\\
\sigma_{21}\sigma_{42}-\sigma_{22}\sigma_{41}	& \sigma_{21}\sigma_{43}-\sigma_{23}\sigma_{41}	& \sigma_{22}\sigma_{44}-\sigma_{24}\sigma_{42}	& \sigma_{23}\sigma_{44}-\sigma_{24}\sigma_{43}\\
\sigma_{31}\sigma_{42}-\sigma_{32}\sigma_{41}	& \sigma_{31}\sigma_{43}-\sigma_{33}\sigma_{41}	& \sigma_{32}\sigma_{44}-\sigma_{34}\sigma_{42}	& \sigma_{33}\sigma_{44}-\sigma_{34}\sigma_{43}\\
\end{pmatrix},
\end{equation} 
which comprises sixteen of the thirty-six 2-minors of $S$.

\begin{proposition}
\mylabel{prop:properties isogeny}
Over the ground field $k=\FF_2$, the above map defines an isogeny $\varphi:\Sp_4\ra\Sp_4$ of group schemes
satisfying $\varphi^2=F$. The kernel is the height-one group scheme with 5-dimensional Lie algebra
comprising the block matrices $T=\left(\begin{smallmatrix}A&B\\C&D\end{smallmatrix}\right)$
whose blocks \eqref{eq:blocks in symplectic lie algebra} satisfy
$$
a=d\quadand w=v=0\quadand y=z=0.
$$
It respects the pinning, and induces $\varphi_*(\alpha_1^\vee)=2\alpha_2^\vee$ and $\varphi_*(\alpha_2^\vee)=\alpha_1^\vee$.
Moreover, for each non-zero perfect $k$-algebra $R$ the induced  map $\Sp_4(R)\ra\Sp_4(R)$ is an  automorphism  that is not inner.
\end{proposition}

\proof
A  brute force computation reveals that on the symplectic group, the map $\varphi$ respects multiplication,
factors over $\Sp_4$,  and satisfies $\varphi^2=F$. 
The latter ensures that $\Kernel(\varphi)$ is annihilated by the Frobenius. 
Applying $\varphi$ to symplectic matrices of the form $S=E_{4} + \epsilon\cdot T$ gives the statement on the Lie algebra.

Of course,  there is an intrinsic explanation in terms of algebraic geometry. This seems to be well-known, but we had
difficulties to find conclusive references
(compare \cite{Hua 1948}, \cite{Tits 1962}, \cite{Duncan 1968}, \cite{Todd 1970}, \cite{Dieudonne 1971}).
One may proceed as follows:

Let $(E,\Psi)$ be a four-dimensional symplectic vector space, 
with dual space $E^*$. Using the Grothendieck convention, we 
write $\Grass^2(E^*)$ for the Grassmann variety of 2-dimensional quotients $E^*\ra k^2$.
Sending such quotients to  the induced one-dimensional quotient $\Lambda^2E^*\ra\Lambda^2k^2$ defines the Pl\"ucker embedding
$$
\Grass^2(E^*)\subset\Grass^1(\Lambda^2E^*)=\PP(\Lambda^2E^*).
$$
This is an effective divisor of degree two, hence defined by some $Q\in\Sym^2(\Lambda^2E^*)$, which  is unique up to invertible scalars. 
Consider the vector subspace
$$
\Kernel(\Psi)_0=\{x  \mid \text{$\Psi(x)=0$ and $Q(x,x)=0$}\} 
$$ 
inside $\Lambda^2E$, on which $Q$ becomes a non-degenerate symplectic form.
The standard representation of the group scheme $\Sp_{(E,\Psi)/k}$ on $E$ fixes the form $\Psi$ and stabilizes the subspace $kQ$.
Since this group scheme has no non-trivial characters, $Q$ is actually fixed.
In turn, we get an induced symplectic representation of $\Sp_{(E,\Psi)/k}$  on  $\Kernel(\Psi)_0$.  
To compute it, we choose a   symplectic basis $e_i\in E$, $1\leq i\leq 4$, here in the sense that  the Gram matrix takes the form $(\Psi(e_i,e_j))=J$.
Write  $e_i^*\in E^*$ for the dual basis.
The 2-vectors   $e^*_{ij}=e^*_i\wedge e^*_j$, $1\leq i<j\leq 4$, with lexicographic ordering form a basis for $\Lambda^2E^*$, and we have 
$$
\Psi=e_{14}^*+e_{23}^* \quadand Q=e^*_{12}e^*_{34} -  e^*_{13}e^*_{24} + e^*_{14}e^*_{23}.
$$
Then $\Kernel(\Psi)_0$ is generated by the 2-vectors $e_{12},e_{13}, e_{24},e_{34}$, which actually  form  a symplectic basis. 
The sixteen 2-minors in \eqref{eq:explicit isogeny} simply record the effect of an element $S=(\sigma_{ij})$ in $\Sp_4(K)=\Sp(E,J)$ on these basis vectors. 
This proves   that the map $\varphi$ respects the group law and factors over the symplectic group.

Suppose now that the ring $R\neq 0$ is perfect, and in particular reduced. For the height-one group scheme $H=\Kernel(\varphi)$,
it follows that the  group $H(R)$ is trivial, and that every $H$-torsor over $\Spec(R)$ admits a section.
Consequently, the map $\varphi:\Sp_4(R)\ra\Sp_4(R)$  is bijective. By definition, it effects 
$$ 
\begin{pmatrix}
1	& \lambda	&	&\\
	& 1	&	&\\
	&	& 1	&\lambda\\
	&	&  	&1
\end{pmatrix}
\mapsto
\begin{pmatrix}
1	& 	&	&\\
	& 1	& \lambda^2 	& \\
	&	& 1	&  \\
	&	&	& 1\\
\end{pmatrix}
\quadand
\begin{pmatrix}
1	& 	&	&\\
	& 1	& \mu 	&\\
	&	& 1	&  \\
	&	&	& 1\\
\end{pmatrix}
\mapsto 
\begin{pmatrix}
1	& \mu	&	&\\
	& 1		&	&\\
	&		& 1	&\mu\\
	&		&  	&1
\end{pmatrix}.
$$
For $\lambda=\mu=1$ these matrices have different Jordan normal forms over each residue field $K=R/\maxid$, hence the automorphism is not inner.
This description also reveals that $\varphi$ respects the pinning, with the stated effect on the coroots.
\qed

\medskip
Consequently, our explicit description 
\eqref{eq:explicit isogeny} defines a purely inseparable  isogeny $\varphi:\Sp_4\ra \Sp_4$  of degree $p^5=32$,
where the induced map on the underlying topological space is a homeomorphism of order two.
By looking at the effect on the pinning, we see that it 
coincides with the abstract description  \eqref{eq:matrix giving isogeny c2twist} from the Isogeny Theorem in the case $n=0$. Similarly, for general $n$, $\varphi^n$ is the isogeny corresponding to~\eqref{eq:matrix giving isogeny c2twist}. 

Let $M$ be the set of all subsets 
$\{v_1,\ldots,v_5\}\subset(\FF_2)^4$ with ${}^tv_i J v_j=1$ for $i\neq j$.
Then $|M|=6$ and the induced permutation representation $\Sp_4(\FF_2)\ra S_M$  is bijective,
as explained in \cite{Grove 2002}, Proposition 3.13.
This gives, up to inner automorphisms, an   identification $\Sp_4(\FF_2)=S_6$. The outer automorphism group  has order two,
a result due to H\"older (\cite{Hoelder 1895}, \S 8, compare also \cite{Janusz; Rotman 1982}), so our $\varphi$
yields the generator.

Recall that the \emph{Suzuki groups} \cite{Suzuki 1960} were originally defined for $q=\sqrt{2^s}$, $s$ odd, as   subgroup $\Sz(q^2)\subset \Sp_4(\FF_q)$
generated by certain explicit matrices, with   
$$
|\Sz(q^2)|=q^4(q^2-1)(q^4+1).
$$
Following ideas of Ree   \cite{Ree 1961}, this was reinterpreted as the fixed group with respect to a certain involution on 
$\Sp_4(\FF_{q^2})$ by Ono  (\cite{Ono 1962}, \cite{Ono 1963}), compare also  \cite{Dieudonne 1971}, Chapter IV, \S 3). From our perspective, it is the group of
$\FF_q$-valued points of the group scheme $(\Sp_4)^{\varphi^s}$, which is already defined over the prime field $k=\FF_2$.

The Suzuki groups $\Sz(2^s)$ are simple, except for $s=1$. In fact, the group $\Sz(2)$  has order twenty, and
by the Sylow Theorems  the 5-Sylow group must be normal. Consider the following two $\varphi$-fixed symplectic matrices
\begin{equation}
\label{generators suzuki group}
A=
\begin{pmatrix}
1&0&1&1\\
1&0&1&0\\
0&1&0&0\\
1&0&0&0
\end{pmatrix}
\quadand
S=\begin{pmatrix}
1	& 1	& 1	& 1\\
	& 1	& 1	& 0\\
	&	& 1	& 1\\
	&	& 	& 1
\end{pmatrix}.
\end{equation} 
The former has order five, the latter has order four,  and $SAS^{-1}=A^2$, hence
\begin{equation}
\label{eq:identification suzuki group}
\operatorname{Sz}(2)=\langle A,S\rangle =  C_5\rtimes\Aut(C_5).
\end{equation} 
Note that $A$ is not an upper triangular matrix, and thus does not belong to the standard Borel group $B\subset G$.
It follows that the cosets  $A^iB$, $0\leq i\leq 4$ are pairwise distinct  elements in the flag variety.
We shall see below that the inclusion 
\begin{equation}
\label{eq:c2twist zero stratum}
\{A^0B, \ldots, A^4B\}\subset (G/B)^\varphi=X(e)
\end{equation} 
is actually an equality.

\subsection{Effect on isotropic flags}
\mylabel{subsec:effect isotropic}

\newcommand\longmapsfrom{\mathrel{\reflectbox{\ensuremath{\longmapsto}}}}
We continue to work with the symplectic group $G=\Sp_4$ over the ground field $k=\FF_2$.
Recall that the   standard Borel group $B$ consists of the symplectic matrices 
that are  upper triangular.
The elements $(\sigma_{ij})B$ of the flag variety
$G/B$ can be viewed as  \emph{isotropic  flags} $(L\subset U)$, with  defining condition   $U=U^\perp$. Here the line
$L$ is generated by the first column,  and the plane $U$ is generated by the first two columns in $S=(\sigma_{ij})$.
Throughout, we write
$$
\varphi(L\subset U) = (L'\subset U')
$$
for the effect of the isogeny $\varphi:G\ra G$ on isotropic flags.

\medskip
The description of the Lie algebras in \eqref{eq:explicit isogeny} shows that $\dim(G/B)=4$.
For each    integral closed subscheme $Z\subset G/B$  with image $Z'=\varphi(Z)$ and dimension $d\geq 0$, 
the relative Frobenius of $G/B$ restricted to $Z$ factors as $Z\stackrel{\varphi}{\ra} Z' \stackrel{\varphi}{\ra} Z$,
giving 
\begin{equation}
\label{eq:factorization of degrees}
p^d= \deg(Z/Z')\cdot \deg(Z'/Z).
\end{equation} 
For the relevant Deligne--Lusztig curves and surfaces, we compute:

\begin{proposition}
\mylabel{prop:degrees inseparable maps}
Both $\bX(s_1)$ and $\bX(s_2)$ are isomorphic to the elliptic curve  with five rational points,
and $X(e)$ is their common scheme of rational points. Furthermore, the 
 morphism  $\varphi:G/B\ra G/B$ induces universal homeomorphisms
\begin{gather*}
\bX(s_1)\stackrel{2:1}{\lra} \bX(s_2) \quadand \bX(s_2)\stackrel{1:1}{\lra} \bX(s_1),\\
\bX(s_1s_2)\stackrel{2:1}{\lra} \bX(s_2s_1) \quadand \bX(s_2s_1)\stackrel{2:1}{\lra} \bX(s_1s_2) 
\end{gather*}
with indicated  degrees.
\end{proposition}

\proof
The curve $\bX(s_1)$ and $\bX(s_2)$ are $K$-trivial, hence are para-elliptic. 
By the classification of elliptic curves over $k=\FF_2$, there at most five rational points. 
On the other hand,  \eqref{eq:c2twist zero stratum} gives inclusions
$\{A^0B, \ldots, A^4B\}\subset X(e)\subset \bX(s_i)$, and the first assertion follows.

If two cosets $S_1B,S_2B$ in the flag variety $G/B$ are in relative position $w\in W$,  their images under the isogeny $\varphi$
are in relative position $\varphi(w)$. From $\varphi(s_1)=s_2$ and $\varphi(s_2)=s_1$ we get
the statements on the $\varphi$-images of our Deligne--Lusztig varieties.
To compute the relative $\varphi$-degrees, consider the symplectic matrix
$$
S=\begin{pmatrix}
1	&  	&  	&  \\
\epsilon	& 1	&  	&  \\
	&	& 1	&  \\
	&	& \epsilon	& 1
\end{pmatrix}\in\Sp_4(k[\epsilon])
$$
over the ring of dual numbers $R=k[\epsilon]$, whose image  $\varphi(S)$ is the identity matrix.
For the resulting isotropic flags $(L'\subset U')=\varphi(L\subset U)$, we see $L'\neq L$ and $U'=U$.
In view of Lemma~\ref{lem:position symplectic flags}, both cosets $SB\neq B$ belong to $\bX(s_1)$ and have the same image in $\bX(s_2)$.
Using \eqref{eq:factorization of degrees}, we infer the statements  about the    $\varphi$-degrees of the Deligne--Lusztig curves.
In light of the inclusions of $\bX(s_1)$ into both $\bX(s_1s_2)$ and $\bX(s_2s_1)$, the statements about the $\varphi$-degrees
of the surfaces follow as well. 

Set $X=\bX(s_2s_1)$ and $R=\bX(s_2)$. We observed in Section~\ref{subsec:tables} that $K_X\equiv -R$. The Adjunction Formula gives $K_R=0$, hence 
$R$ is a para-elliptic curve.  By the classification of para-elliptic curves, it contains at most five rational points.
From $\varphi^2=F$ we infer that every point in the finite smooth scheme $\bX(e)$ is rational.  
Consider the matrix $A$ defined in  \eqref{generators suzuki group}, which belongs to the Suzuki group 
$G^\varphi(k)=\Sz(2)=C_5\rtimes\Aut(C_5)$ and generates the $5$-Sylow group. 
This matrix is lower-triangular, whereas the standard Borel group comprises only upper triangular matrices.
So the cosets $A^iB$ show that $\bX(e)$ contains at least five rational points.  
From the inclusion   $\bX(e)\subset\bX(s_2)$ we see that both $\bX(e)$ and $\bX(s_2)$ have exactly five rational points.
We also see that $C_5$ acts faithfully on $\bX(e)$, and conclude with Lemma \ref{lem:holomorph cyclic group} below that
the same holds for $C_5\rtimes\Aut(C_5)$.
\qed

\medskip
In the above proof we have used a group-theoretical observation:

\begin{lemma}
\mylabel{lem:holomorph cyclic group}
Let $N$ be a cyclic group of prime order $\ell>0$. Then the non-trivial normal subgroups in the holomorph $H=\operatorname{Hol}(N):= N \rtimes \Aut(N)$
are precisely the $H'=N\rtimes A'$, for some (normal) subgroup $A'$ inside $A=\Aut(N)$.
\end{lemma}

\proof
Being the kernels for the composite maps $H\ra   A/A'$, the subgroups $H'=N\rtimes A'$ are normal.
Conversely, suppose we have a non-trivial normal subgroup $H'\subset H$.
If $H'$ is contained in $N$ it must be normal in $N$. Since the latter is simple we obtain $H'=N\rtimes A'$
for the trivial subgroup $A'=\{\id_A\}$.
Now write  $N=\ZZ/\ell\ZZ$ and $\Aut(N)=(\ZZ/\ell\ZZ)^\times$
and suppose that there is some $(a,u)\in H'$ with $u\neq 1$.
The conjugate  
$$
(x,1)\ast (a,u)\ast (-x,1) = (x+a,u)\ast (-x,1)= (x+a-ux,u)
$$
also belongs to $H'$. The element $1-u\in\ZZ/\ell\ZZ$ is non-zero, hence a unit,
so the elements $x+a-ux$ for varying $x$ exhaust the whole group $N$.
In other words, $N\times\{u\}\subset H'$, and the assertion follows.
\qed

\subsection{Proof of the structure result}
\mylabel{subsec:proof main b2twist}

The goal of this section is to prove that the diagrams~\eqref{eq:diagram dl surfaces} obtained from Deligne--Lusztig theory and~\eqref{eq:diagram ruled surfaces} obtained by the explicit construction above, coincide. In particular, this pins down the description of the two Deligne--Lusztig surfaces as ruled surfaces.

\medskip
\emph{Proof of Theorem \ref{thm:c2twist}.}
Recall that we work over the prime field  $k=\FF_p$ in characteristic $p=2$, and that our Deligne--Lusztig varieties sit in 
a commutative diagram
$$
\begin{CD}
\bX(s_2s_1) 	@>\varphi>>	\bX(s_1s_2)\\
@VVV				@VVV\\
\bX(s_1)	@>>\varphi>	\bX(s_2).
\end{CD}
$$
The vertical arrows are rulings, by Corollary~\ref{cor:structure of ruled surface}, and the horizontal maps are universal homeomorphisms of degree $p$,
according to Proposition \ref{prop:degrees inseparable maps}. 
Since $k$ is perfect, the function field $K$ of $\bX(s_1)$ has only one subextension over which it is purely inseparable of degree $p$,
namely $k\subset K^p\subset K$. Hence  the lower horizontal arrow can be interpreted as the 
Frobenius map for $\bX(s_2)$, and these  curves have genus $g=1$; see Section~\ref{subsec:tables}.
The induced map $\bX(s_2s_1)\ra \bX(s_1s_2)\times_{\bX(s_2)}\bX(s_1)$ of ruled surfaces is  birational and finite,
hence an  isomorphism, because both schemes are regular. We have $\varphi^2=F$,
and $\varphi:\bX(s_2)\ra\bX(s_1)$ is an isomorphism by Proposition \ref{prop:degrees inseparable maps}. It  follows that $\varphi:\bX(s_1s_2)\ra \bX(s_2s_1)$
is the relative Frobenius map. 

To see that the diagrams \eqref{eq:diagram dl surfaces} and \eqref{eq:diagram ruled surfaces} coincide, it thus suffices to verify that $\bX(s_1s_2)$
is isomorphic to $X=\Sym^2_{E/k}$ for the elliptic curve $E:y^2+y=x^3+x$ with five rational points.
To this end we may apply Theorem~\ref{thm:characterization symmetric product} with $D = \bX(s_1)$ and $D' = \bX(s_2)$. It follows from the table in Section~\ref{subsec:tables} that $K_X \equiv -D$, and furthermore $(D'\cdot D) = 5$ is odd.
\qed

\subsection{Characterization via symmetries}
\mylabel{subsec:characterization b2twist}

In the final section of the chapter on the Suzuki case ${}^2 C_2$, we prove that similarly to the other Deligne--Lusztig surfaces, the surface $\bX(s_2 s_1)$ admits a characterization in terms of a certain divisor configuration.

Let $E$ be an elliptic curve over the prime field $k=\FF_p$ of characteristic $p=2$,
and $f:Y\ra E$ be a ruled surface. Since this is the Albanese map, the ruling induces a homomorphism
$\Aut_{Y/k}\ra E \rtimes\Aut_{E/k}$.
 
\begin{theorem}
\mylabel{thm:characterization ruled via symmetries}
Suppose  $Y$ admits a faithful action of the group  $\Gamma=C_5\rtimes\Aut(C_5)$, 
and that there are  $\Gamma$-stable sections $A,B\subset Y$ such that 
$$
K_Y\equiv -2A\quadand (A\cdot B)=5.
$$
Then the following holds:
\begin{enumerate}
\item 
The base is the elliptic curve $E:y^2+y=x^3+x$ with five rational points.
\item 
We have  $Y=\PP(\shF)$ for the non-split extension  $0\ra\O_E\ra\shF\ra\O_E\ra 0$.
\item
The ruling identifies $A\cap B$ with the set of rational points $e_1,\ldots,e_5\in E$.
\item
The   section $A$ corresponds to the surjection $\shF\ra \O_E$, whereas $B$
stems from a surjective extension $\shF\ra\O_E(e_1+\ldots + e_5)$ of  the standard inclusion 
$\O_E\ra \O_E(e_1+\ldots + e_5)$. 
\item
The canonical map $\Gamma\ra E(k)\rtimes \Aut(E)$ is bijective, and $\omega_Y=\O_Y(-2A)$.
\end{enumerate}
\end{theorem}

\proof
We start with (i) and the first part of (v).
Seeking a contradiction,   assume that $C_5\subset\Gamma$ fixes the origin $e\in E$.
According to \cite{Deligne 1975}, Proposition 5.9,
the  order of $\Aut(E)$ divides $n=24$, so the $C_5$-action on $E$ and thus also on  $A,B$ are trivial. 
Write $F=k(E)$ for the function field,  and  choose an identification $f^{-1}(\eta)=\PP^1_F$ so that
$A$ corresponds to $(0:1)$. The inclusion of $C_5$ into $\PGL_2(F)$ then factors over the standard Borel group 
$ F\rtimes F^\times$.  Using $p=2$, we see
that the projection $C_5\ra F^\times$ remains injective, giving
a primitive $5$-th root of unity $\zeta\in F$. We have $\zeta\in \Gamma(E,\O_E)=k$  because $E$ is normal,
in contradiction to $k=\FF_2$.
So $C_5$ does not fix the rational point $e\in E$, and its orbits contributes five rational points.
The classification of elliptic curves over $k=\FF_2$ yields (i).
We also see that the canonical map $\Gamma\ra E(k)\rtimes\Aut(E)$ is injective on $C_5$,
hence injective by  Lemma \ref{lem:holomorph cyclic group}. Since both groups have the same order,
the map is  bijective, which establishes the first part of (v).
 
We next check (ii). The image $Z=f(R\cap R')$ is a finite $\Gamma$-stable scheme of degree five.
If   $Z$ contains one rational point, it contains all of them, and  $\{e_1,\ldots,e_5\}\subset Z$ is an equality.
Seeking a contradiction, we now assume that $Z(k)$ is empty. Write $Z=m_1z_1+\ldots+m_rz_r$, such that 
$5=\sum_{i=1}^r m_i\cdot[\kappa(z_i):k]$. Up to order, the only possibilities are
$$
r=1,\  [\kappa(z_1):k]=5\quadand r=2,   [\kappa(z_1):k]=3,\ [\kappa(z_2):k]=2,
$$
with multiplicities $m_i=1$. In the second case, the asymmetry of degrees ensures that both $z_1,z_2\in E$ are $\Gamma$-stable.
In all cases, we thus have some   $\Gamma\ra\Gal(\kappa(z_i)/k)$. These have a non-trivial kernel, but  must be injective on $C_5$,
because the translations act freely, in contradiction to Lemma \ref{lem:holomorph cyclic group}.
This settles (ii).

We next verify (iv) and  (ii). Using \cite{Hartshorne 1977}, Chapter III, Corollary 12.9 one sees that the direct images
\begin{equation}
\label{eq:direct images}
\shF=f_*(\O_Y(A))\quadand \shN=f_*(\O_A(A)) \quadand \shL=f_*(\O_{B}(A))
\end{equation} 
are locally free. The sheaf $\shF$ has  rank two and we obtain an identification $Y=\PP(\shF)$.
The others have rank one, with $A=\PP(\shN)$ and $B=\PP(\shL)$.
We have 
$$
8\chi(\O_E) = K_Y^2=4A^2 = 4 \deg(\shN),
$$
by \cite{Hartshorne 1977}, Chapter V, Corollary 2.11, and conclude $\deg(\shN)=0$.
The other invertible sheaf has  $\deg(\shL)=(A\cdot B)=5$.
The  coherent sheaves $\O_Y(A)$ and $\O_A(A)$ and $\O_A(B)$ carry   $\Gamma$-linearizations,
hence the same holds for the direct images  in \eqref{eq:direct images}.
In particular, the classes of $\shN$ and $\shL$ in the Picard group are fixed by $\Aut(E)$. Using that the bijections
$$
E(k)\lra \Pic^0(E)  \quadand \Pic^0(E)\lra \Pic^5(E) 
$$
given by $b\mapsto \O_E(b-e)$ and $\shI\mapsto \shI(e_1+\ldots+e_5)$ are   equivariant with respect to $\Aut(E)$,
we infer that $\shN=\O_E$ and $\shL=\O_E(e_1+\ldots+e_5)$.
The sheaf $\shF$ thus sits in two extensions
\begin{equation}
\label{eq:two short exact sequences}
\begin{tikzcd}[row sep=tiny]
0\ar[r]		& \O_E 	\ar[rd] 		& 		& \shL\ar[r] 	& 0 \\		
		&					& \shF\ar[ru]\ar[rd]\\
0\ar[r]		& \shL^{\otimes-1}\ar[ur]	&  		& \O_E\ar[r]		& 0,
\end{tikzcd}
\end{equation}
the first  coming from  $0\ra \O_Y\ra\O_Y(A)\ra\O_A(A)\ra 0$, the second stemming from $0\ra \O_Y(A-B)\ra\O_Y(A)\ra\O_B(A)\ra 0$.

View the $5$-dimensional vector space   $V=H^0(E,\shL)$ as a   $\Gamma$-representation, and fix a generator $\sigma\in C_5 $.
For the ensuing automorphism $\sigma:V\ra V$, we now determine the rational normal form and the corresponding
similarity invariants, or equivalently the invariant factors (compare \cite{A 4-7}, Chapter VII, \S5).
Testing with the root-less polynomials of degree at most two, one easily sees 
that $T^5-1=(T-1)(T^4+T^3+\ldots + 1)$ is the prime decomposition  over $k=\FF_2$.
If follows that the similarity invariants  are either $(T-1)^5$ or $T^5-1$. 
Using that $\shL$ is very ample one sees 
that $\sigma$ is not the identity on $V$. Thus the similarity invariant must be  $T^5-1$,
hence  the eigenspace $V'\subset V$ for the eigenvalue $\lambda=1$ is one-dimensional.

Next regard the vector space $U=H^0(E,\shF)$ as $\Gamma$-representation.
From \eqref{eq:two short exact sequences} we see $\dim(U)\leq 2$, with equality if and only if the extension splits.
Arguing with rational normal forms and similarity invariants as in the preceding paragraph, we see that
$C_5$ acts trivially on $U$.
On the other hand,   \eqref{eq:two short exact sequences} gives  an injective map  $U\ra V$, hence the eigenspace $U'\subset U$ for the eigenvalue
$\lambda=1$ is at most one-dimensional. Thus $\dim(U)=1$, and both extensions in  \eqref{eq:two short exact sequences} are non-split.
Moreover, the inclusion $k\subset H^0(E,\O_E(e_1+\ldots+e_5))$ induced by the composite map  $\O_E\subset\shF\ra\shL$ is
the eigenspace for the eigenvalue $\lambda=1$, hence the   map is indeed the canonical inclusion.
This establishes (ii) and  (iv). 

It remains to verify the assertion on $\omega_Y$. As discussed in Section \ref{subsec:formulation c2twist}, the ruled surface
can also be seen as the  Frobenius  base-change of $\Sym^2_{E/k}$, and we write $A'\subset X$ for
the curve  corresponding to  $E^{(p)}\subset \Sym^2_{E/k}$. Then  $-2A'=K_{X'}\equiv -2A$,
and thus $A'\cdot A=0$. So if $A'\neq A$ the two sections must be disjoint, hence the extension $\shF$ splits, contradiction.
\qed

\medskip
Applying the theorem to $Y = \bX(s_2 s_1)$ with its divisors $A = \bX(s_2)$ and $B= \bX(s_1)$ reproves the description of $Y$ given in Theorem~\ref{thm:c2twist}.

\section{The Ree case \texorpdfstring{${^2}G_2$}{2G2}}
\mylabel{sec:case g2twist}

\newcommand{\Ree}{\operatorname{Ree}}
The goal of this section  is to understand for  ${}^2 G_2$  both  geometry and arithmetic of the  Deligne--Lusztig surfaces  
over the prime field with negative canonical divisor.

\subsection{Formulation of structure result}
\mylabel{subsec:formulation g2twist}

We start by reviewing the root system: Let $\epsilon_1,\epsilon_2,\epsilon_3\in\RR^3$ be the standard orthogonal basis, and set
$$
\alpha_1= \epsilon_1-\epsilon_2,\quad\alpha_2= -2\epsilon_1+\epsilon_2+\epsilon_3\quadand 
\alpha_1^\vee= \alpha_1,\quad \alpha_2^\vee=\frac{1}{3}\alpha_2,
$$
as in the Bourbaki tables \cite{LIE 4-6}.
Let $\Phi,\Phi^\vee$ be the resulting root system in the coordinate-zero hyperplane $V\subset\RR^3$,
and write $\Delta=\{\alpha_1,\alpha_2\}$ for the system of simple roots. The Weyl group takes the form 
$W=\langle s_1,s_2\mid s_1^2, s_2^2, (s_1s_2)^6 \rangle =D_6$, where
$s_1,s_2\in\GL(V)$   are the reflections corresponding  to the simple roots $\alpha_1,\alpha_2$.
Let $X_*\subset V$ be the lattice generated by the simple coroots $\alpha_1^\vee, \alpha_2^\vee$.
The  dual lattice   $X^*\subset V$ is generated by the simple roots  $\alpha_1,\alpha_2$.
 
Fix some integer $n\geq 0$ and set 
$$
p=3\quadand r=2\quadand s=2n+1\quadand q=\sqrt{p^s}.
$$
With respect to the  standard basis $\epsilon_1,\epsilon_2,\epsilon_3$, the symmetric matrix
$$
p^n\begin{pmatrix}
0 &2 &1\\
2 &1 &0\\
1 &0 &2\\
\end{pmatrix}\in\GL_3(\RR)
$$
stabilizes both lattices $X^*\subset X_*$ inside $\RR^3$. The induced maps $\varphi^*$ and $\varphi_*$ satisfy
$$
\varphi^*(\alpha_1)=p^n\alpha_2,\quad \varphi^*(\alpha_2)=p^{n+1}\alpha_1\quadand
\varphi_*(\alpha_1^\vee)=p^{n+1}\alpha^\vee_2,\;\varphi_*(\alpha^\vee_2)=p^n\alpha^\vee_1,
$$
and thus define a $p$-isogeny of the root datum $(X^*,\Phi,X_*,\Phi^\vee)$, obviously satisfying   $(\varphi^*)^r=p^s$.
Summing up, we have a  Deligne--Lusztig datum
\begin{equation}
\label{dld 2G2}
\DLD=(X^*,\Phi,X_*,\Phi^\vee,\Delta,p,\varphi_*).
\end{equation} 

The root datum $(X^*,\Phi,X_*,\Phi^\vee)$ corresponds to a split reductive group $(G,T)$ over the prime field $k=\FF_3$,
with Borel group $B$ given by the simple system $\Delta$. Choosing a pinning for $G$, we obtain an isogeny $\varphi:G\ra G$ that induces $\varphi^*$
and satisfies $\varphi^r=F^s$. The datum \eqref{dld 2G2} defines for every $w\in W$ a Deligne--Lusztig variety $X(w)$, coming with an action of the group scheme $G^\varphi$,
and smooth compactifications $\bX(w)$ for each simple expression $w=s_{i_1}\ldots s_{i_d}$.

The \emph{Ree groups}   ${}^2G_2(3^s)$ can be seen as the $\FF_{3^s}$-valued points of the group scheme $G^\varphi$, compare \cite{Ree 1961}.
We use the notation 
$$
\Ree(3^s)= {}^2G_2(3^s) = G^\varphi(\FF_{3^s}).
$$
With $q=p^{s/r} = 3^{s/2}$, the order takes the form $|\Ree(3^s)|=(q^2-1)q^6(q^6+1)$.
Alternative descriptions are due to Tits \cite{Tits 1995} and Wilson \cite{Wilson 2010a}, \cite{Wilson 2010b}.

The Deligne--Lusztig surface $\bX(w)$ has negative canonical divisor if and only if $w=s_2s_1$ and $\varphi^2=F$, in other words $n=0$; see Section~\ref{subsec:tables}.
We can now unravel its  geometry and arithmetic.
From Deligne--Lusztig theory we know the following: The Deligne--Lusztig curves $\overline{X}(s_1)$ and $\overline{X}(s_2)$ are smooth projective curves of genus $15$ over $\FF_3$. They are naturally contained in $X=\bar{X}(s_2s_1)$.

Furthermore by Corollary~\ref{cor:structure of ruled surface} the surface $X$ has the structure of a ruled surface over $C:=\overline{X}(s_1)$, and $\overline{X}(s_1) \subset X$ is a section. In particular the Albanese map of $X$ factors through the morphism $X\to C$.

The main result of this section is the following theorem which explicitly determines the structure of $X$ as a ruled surface over the Deligne--Lusztig curve.

\begin{theorem}
\mylabel{thm:structure g2twist}
The ruled surface $X = \bar{X}(s_2s_1)$ over $C = \bar{X}(s_1)$ is isomorphic to
$X=\PP(\shE)$ for the canonical extension  $0\ra\O_C\ra\shE\ra\omega^{\otimes-1}_C\ra 0$.
\end{theorem}

The proof requires some preparation, and will be given in Section~\ref{subsec:conclusion ree}.
The symmetries afforded by the Ree group and the sections given by the Deligne--Lusztig curves will allow us to apply Theorem~\ref{thm:characterization ruled ree} below, which characterizes surfaces over $\FF_3$ with an action by the Ree group, ruled over a curve of genus $15$, and equipped with a certain configuration of divisors.

\subsection{Ree curves}
\mylabel{subsec:ree curves}

\newcommand{\wild}{\text{\rm wild}}
We start by taking a closer look at the Ree curve $C = \bar{X}(s_1)$.
As before, let $p=3$, $n\geq 0$ a natural number and set  $q=p^{\frac{2n+1}{2}}$ and $q_0=p^n$. According to the table in Section~\ref{subsec:tables} the curve $C$ has genus
\[
    h^1(\O_C) = \frac{3}{2}q_0(q^2-1)(q^2+q_0+1).
\]
We will use the following unique characterization of this curve that we use in the form given in~\cite{Hirschfeld; Korchmaros; Torres 2008}, Theorem~12.31; see also~\cite{Pedersen 1992}, \cite{Hansen; Pedersen 1993}.

\begin{proposition}[\cite{Hirschfeld; Korchmaros; Torres 2008}, Theorem~12.31]\mylabel{prop:characterization ree curve}
Let $p=3$, $n \ge 0$, $q = p^{\frac{2n+1}{2}}$ and $q_0=p^n$. Let $C$ be a curve over $\FF_{q^2}$ of genus $g = \frac 32 q_0 (q^2-1) (q^2+q_0+1)$ whose automorphism group $\Aut_{C/\FF_{q^2}}(\FF_{q^2})$ contains the Ree group $\Ree(q^2)$.
    Then $C$ is isomorphic to the Ree curve.
\end{proposition}

Furthermore, the following properties of $C$ are known.

\begin{proposition}[\cite{Hirschfeld; Korchmaros; Torres 2008} Theorem~12.29]\mylabel{prop:properties ree curve}
    The Ree curve $C$ has the following properties.
    \begin{enumerate}
        \item The automorphism group of $C$ is the Ree group $\Ree(3^{2n+1})$.
        \item The curve $C\otimes_{\FF_p} \FF_{q^2}$ is the regular proper model of the curve given by the affine equations
            \[
                y^{q^2}-y = x^{q_0}(x^{q^2}-x)\quadand z^{q^2}-z = x^{q_0}(y^{q^2}-y)
            \]
            over the field $\FF_{q^2}$.
    \end{enumerate}
\end{proposition}

Consider the affine equations given in the proposition
over the prime field $k=\FF_p$. (The following discussion is valid for any prime $p > 0$.) They define a finite   field extension $k(x)\subset k(x,y,z)$, which is separable of degree $q^4$,
and thus corresponds to a  regular proper curve $C$ coming with a branched covering $h:C\ra\PP^1$. The latter is   \'etale  over
the open set $\AA^1=\Spec k[x]$. Using that the projective line becomes  simply-connected over $k^\alg$, one sees
that $h:C\ra\PP^1$ is totally ramified over $\infty\in\PP^1$. The morphism is actually an iterated
Artin--Schreier covering, and one infers 
$$
|C(\FF_{q^2})| = \deg(C/\PP^1)\cdot |\AA^1(\FF_{q^2})| +1  = q^6+1.
$$
Moreover, with $u\in \FF_{q^2}^\times $ and $r,s,t\in \FF_{q^2}$  the coordinate changes
\begin{align*}
x & = ux'+r,\\
y & = u^{q_0+1}y' + ur^{q_0}x' + s \\
z & = u^{2q_0 +1}z' - u^{q_0+1}r^{q_0}y + u^{q_0}r^{2q_0}x + t
\end{align*}
define a  group  of order $(q^2-1)q^6$ in $\Aut_{C/\FF_{q^2}}(\FF_{q^2})$. Note the analogy to  
elliptic curves (see the coordinate change formulas in~\cite{Tate 1975} and~\cite{Deligne 1975}).

We now return to the case $p=3$, where these equations describe the \emph{Ree curve} $\bar{X}(s_1)$, and restrict to the simplest case $n=0$, in other words $p=3$, $q = \sqrt{3}$ and $q_0=1$. The number of rational points is $|C(\FF_p)|=28$,
the genus is $h^1(\O_C)=15$,   and  the automorphism group scheme  is constant,  given by the semidirect product
$$
\Aut_{C/k}(k)= \Ree(3) \cong \PGL_2(\FF_8)\rtimes\Aut(\FF_8) 
$$
of order $2^3\cdot 3^3\cdot 7$.  Note that both factors are simple, and the term on the right is cyclic of order three.
The following canonical bundle formula will play a certain role:

\begin{proposition}
\mylabel{prop:canonical bundle ree}
For $n=0$, the  dualizing sheaf of the Ree curve $C$  over $k=\FF_3$ has the property  $\omega_C=\O_C(c_1+\ldots+c_{28})$, where
$c_1,\ldots,c_{28} \in C$ are the rational points of $C$.
\end{proposition}

\proof
The  dualizing sheaf $\omega_C$ is trivial over $h^{-1}(\AA^1)=C\smallsetminus\{c_\infty\}$, 
where $c_\infty\in C$ is the point over $\infty\in\PP^1$. In turn, we have
the canonical bundle formula $K_C=(2g-2)\cdot c_\infty$, valid for any curve of genus $g=h^1(\O_C)$.
In our specific situation, this gives
$$
\omega_C=\O_C(28\cdot c_\infty) = f^*(\O_{\PP^1}(3) )\otimes\O_C(c_\infty).
$$
Let $s\in\Gamma(\PP^1,\O_{\PP^1}(3))$ be a global section that vanishes at the three rational points $\neq\infty$.
Then $f^*(s)$ vanishes exactly at the $27$ rational points on $C$ that differ from $ c_\infty$,
and the desired formula $\omega_C=\O_C(c_1+\ldots+c_{28})$ follows.
\qed

\medskip
In other words, the sum of rational points $\sum_{i=1}^{28} c_i$ corresponds to a distinguished non-zero  section $s\in H^0(C,\omega_C)$,
which is unique up to sign.

\subsection{Characterization via symmetries and curves}
\mylabel{subsec:characterization g2twist}

Our next goal is to give a characterization of the Ree Deligne--Lusztig surface $\bar{X}(s_2s_1)$ in terms of the action of $H$ and a configuration of divisors. To this end, we start out over an arbitrary ground field $k$ of characteristic $p\geq 0$.
Let $f\colon X\ra C$ be a ruled surface over a curve $C$ of genus $g=15$,
and suppose that the  finite group $H=\Ree(3)=\PGL_2(\FF_8)\rtimes\Aut(\FF_8)$ acts on the surface. 
The curve
$C$ is the schematic image of the Albanese map $X\ra \Alb_{B/k}$. The latter is functorial,
thus $C$ inherits a unique $H$-action for which $f\colon X\ra C$ is equivariant  (\cite{Laurent; Schroeer 2024}, Corollary 10.3).

\begin{proposition}
\mylabel{prop:ree group on ruled}
Suppose the action of the subgroup $\PGL_2(\FF_8)$ on the surface $X$ is non-trivial.
The canonical map $H\ra\Aut(C)$ is injective, the scheme of $H$-fixed points on $C$ is empty,
and we have $p\in\{2,3,7\}$.
\end{proposition}

\proof
Seeking a contradiction, we assume that the action of the simple group $\PGL_2(\FF_8)$ on $C$ is trivial.
We then get an induced action on the generic fiber $f^{-1}(\eta)=\PP^1_F$, over the function field $F=\FF_3(T)$.
This action is faithful, and fixes both $R_\eta$ and $R'_\eta$. This yields an inclusion of $\PGL_2(\FF_8)$
into the commutative group $\GG_m(F)= F^\times$, contradiction. Thus the action of $\PGL_2(\FF_8)$ on $C$ is non-trivial.
Observe that both factors of $H=\PGL_2(\FF_8)\rtimes\Aut(\FF_8)$ are simple groups. Thus by Lemmas~\ref{lem:simple groups} and~\ref{lem:ree not product} below, the $H$-action on $C$ must be  faithful, giving an  inclusion $H\subset\Aut(C)$.

The assertion on the characteristic follows, because $1512 = |H|$ exceeds the Hurwitz bound $84\cdot (g-1)= 1176$.
By Lemma \ref{lem:fixed scheme empty} below, the 
scheme of fixed points for the  $\PGL_2(\FF_8)$-action on  $C$ is empty.
Consequently, the  same holds for the $H$-action.
\qed

\medskip
\begin{theorem}
\mylabel{thm:characterization ruled ree}
Let $f\colon X\to C$ be a ruled surface over $k=\FF_3$, where $C$ has genus $15$, and where $H=\PGL_2(\FF_8)\rtimes\Aut(\FF_8)$ acts on $X$.
Suppose that the action of $\PGL_2(\FF_8)$ on the surface $X$ is non-trivial
and that there are   $H$-stable sections $R,R'\subset X$ with
$$
K_X\equiv -2R\quadand (R\cdot R')=28.
$$
Then the following holds:
\begin{enumerate}
\item 
The base $C$ is the Ree curve defined by the affine equations
$y^3-y = x(x^3-x)$ and $z^3-z = x(y^3-y)$, and the inclusion $H\subset\Aut(C)$ is an equality.
\item 
We have  $X=\PP(\shE)$ for the canonical extension $0\ra\O_C\ra\shE\ra\omega^{\otimes-1}_C\ra 0$.
\item
The  section $R$ corresponds to the given surjection $\shE\ra \omega^{\otimes-1}_C$, whereas the second section $R'$
is given by a surjective extension $\shE\ra\omega_C$ of  a distinguished section $s:\O_C\ra \omega_C$
that vanishes at the rational points $c_1,\ldots,c_{28}\in C$.
\end{enumerate}
\end{theorem}

\proof
As noted above, the group $H$ acts on $C$. By Proposition~\ref{prop:characterization ree curve} we infer that $C$ is the Ree curve for $n=0$ over $\FF_3$. The further claims in~(i) then follow from Proposition~\ref{prop:properties ree curve}.

To see  (ii) and (iii) we consider the coherent sheaves  
\begin{equation}
\label{direct images rees}
\shE=f_*(\O_X(R))\quadand \shN=f_*(\O_R(R)) \quadand \shL=f_*(\O_{R'}(R)) 
\end{equation} 
on $C$. By cohomology and base change (e.g., \cite{Hartshorne 1977}, Chapter III, Corollary 12.9) they  are locally free, the former of rank two,
the latter of rank one. 
Moreover, we have $X=\PP(\shE)$ and $R=\PP(\shN)$ and $\PP(\shL)$, corresponding to two short exact sequences
\begin{equation}
\label{two extensions}
0\ra \O_C\xrightarrow{\varphi}\shE\ra\shN\ra 0\quadand 0\ra \shL^{\otimes-1}\otimes\shN\ra\shE\xrightarrow{\psi}\shL\ra0.
\end{equation} 
We have $\deg(\shL)=(R\cdot R')=28$ and $\deg(\shN)=-28$, the latter from 
$8\chi(\O_C) = K_X^2=4R^2 = 4 \deg(\shN)$. 

The sheaves in \eqref{direct images rees}  carry    canonical $H$-linearizations.
Now recall that the Ree curve $C$ comes with an iterated Artin--Schreier covering
$C\ra\PP^1$ of degree $q^2=9$, which is \'etale over the open set $\AA^1$,
and totally ramified over $\infty\in\PP^1$. Using that $\Pic(\AA^1)=0$, we infer that every $H$-linearized invertible
sheaf on $C$  is a multiple of $\O_C(c_\infty)$, where $c_\infty\in C$ is the preimage of $\infty\in\PP^1$. This yields
$$
\shN=\omega_C^{\otimes-1}\quadand \shL=\omega_C\quadand \shL^{\otimes-1}\otimes\shN=\omega_C^{\otimes -2}.
$$
We already observed above that $R\cap R'$ corresponds to the sum of rational points $c_1,\ldots,c_{28}\in C$.
In turn, the composition $s=\psi\circ\varphi$ formed with the maps from \eqref{two extensions} yields
a distinguished section $s\colon \O_C\ra\omega_C$, which establishes (iii).

It remains to verify that the first extension in \eqref{two extensions}  is non-split. To this end we first examine
the distinguished section $s\in H^0(B,\omega_C)$. The resulting line $ks$ is $H$-stable, and thus defines
a character 
$\chi:\PGL_2(\FF_8)\rtimes\Aut(\FF_8)=H\ra k^\times=\{\pm 1\}$.
It clearly vanishes on the left, and thus also on the right factor of the semidirect product. 
Thus the section $s\in H^0(C,\omega_C)$ is $H$-fixed. Viewing it as homomorphism
$s:\O_C\ra\omega_C$ between $H$-linearized sheaves, we see that the graph $\Gamma_s\in H^0(C,\O_C\oplus\omega_C)$ is $H$-fixed,
hence the homomorphism is $H$-equivariant.

We next exploit the short exact sequence 
$0\ra\omega_C\xrightarrow{s}\omega_C^{\otimes 2}\ra \bigoplus_{i=1}^{28} \kappa(z_i)\ra0$, which yields an exact sequence 
$$
0\lra H^0(B,\omega_C)\lra H^0(B,\omega_C^{\otimes 2})\lra \bigoplus_{i=1}^{28} k \lra H^1(C,\omega_C)\lra 0
$$
of $H$-representations. The term on the right is the trivial representation, because all homomorphisms $H\ra \{\pm 1\}=k^\times $ are trivial.
The direct sum is a permutation representation given by $k[H/H_z]$, where $H_z$ is the stabilizer of a chosen rational point $z\in C$.
So  the map on the right can be seen as   $(\lambda_{\sigma H_z})\mapsto \sum \lambda_{\sigma H_z}$.
Its kernel $U\subset k[H/H_z]$ is a subrepresentation which does not contain a copy of the trivial representation  $k$.
Seeking a contradiction, we   assume that the extension $0\ra\O_C\ra\shE\ra\shN\ra 0$ splits.
Then $\shE\otimes\omega_C=(\O_C\oplus\shN)\otimes\omega_C=\omega_C\oplus\O_C$, and the surjective mapping 
$$
\psi:\omega_C\oplus\O_C=\shE\otimes\omega_C\lra \shL\otimes\omega_C=\omega_C^{\otimes 2}
$$ 
yields 
an inclusion of the trivial representation $H^0(B,\O_C)=k$ into our subrepresentation $U\subset k[H/H_z]$, contradiction.
\qed

\medskip
The above arguments used the following auxiliary statements, which we establish now.
Let $C$ be a regular proper curve with $h^0(\O_C)=1$ over an arbitrary ground field $k$  of characteristic $p>0$,
endowed with a faithful action of a finite group $G$.
For each closed point $z\in Z$, the \emph{stabilizer subgroup} 
$G_z=\{\sigma\in G\mid \sigma\cdot x=x\}$  acts on the residue field,  and the exact sequence
$$
1\lra I_z\lra G_z\lra \Aut(\kappa(z))
$$
defines the \emph{inertia subgroup} $I_z$. The latter comes with a $\kappa(z)$-linear 
representation on the cotangent space $\maxid_z/\maxid_z^2$, and the exact sequence
$$
0\lra I_z^\wild\lra I_z\lra \GL(\maxid_z/\maxid_z^2)=\kappa(z)^\times
$$
defines the \emph{wild part of the inertia}. Note that the latter is a $p$-group, with cyclic quotient $I_z/I_z^\wild$,
as discussed in \cite{Serre 1979}, Chapter IV, \S 2.
 
\begin{lemma}
\mylabel{lem:fixed scheme empty}
In the above situation, suppose that  $G$ is a non-abelian simple group. Then the scheme of fixed points $C^G$ is empty. 
\end{lemma}

\proof
Suppose there is a  point $z\in C^G$, in other words  $G=G_z=I_z$. Since $G$ is non-abelian and simple,
the cotangent representation $G\ra \GL(\maxid_z/\maxid_z^2)$ is trivial, hence $G$ is a $p$-group,
and therefore isomorphic to the cyclic group $C_p$, contradiction.
\qed

\begin{lemma}
\mylabel{lem:only rational points}
In the above situation, suppose $k=\FF_3$ and $G=\PGL_2(\FF_8)$. 
If there is a   $G$-stable effective divisor  $Z\subset C$ of degree $\deg(Z)=28$, then $Z$ is reduced, every point $z\in \Supp(Z)$ is rational and the $G$-action
on $Z$ is transitive.
\end{lemma}

\proof
Let $z_1,\ldots,z_r\in \Supp(Z)$ be representatives for the $G$-orbits. Write $k_i=\kappa(z_i)$ for the residue fields and
$G_i\subset G$ for the   stabilizer subgroups, such that
\begin{equation}
\label{orbit degrees}
28=\deg(Z) =\sum_{i=1}^r [G:G_i] \cdot [k_i:k] \cdot m_i,
\end{equation} 
where $m_i$ is the multiplicity of $z_i$ in $Z$.
Our task is to verify $r=1$ and $[k_i:k]=m_i=1$ for all $i$. We do this by checking that 
 the only possibility for a stabilizer subgroup    is to have index  $[G:G_i]=28$. 
To start with, we list the isomorphism classes occurring among  the subgroups $H\subset G$.  
According to \cite{Faber 2023}, this are
$$
\PGL_2(\FF_{2^i})\quadand \FF_8\rtimes \FF_8^\times\quadand \FF_2^{\oplus j}\quadand D_n\quadand C_n,
$$
with $i\mid 3$ and $0\leq j\leq 3$ and $n\mid 3^2\cdot 7$. Here $C_n$ denotes the cyclic group of order $n$,
and $D_n=C_n\rtimes\{\pm 1\}$ the dihedral group of order $2n$. By \eqref{orbit degrees}, only those with $[G:H]\leq 28$ may
occur as stabilizer groups, tabulated as follows:
$$
\begin{array}{lllllll}
\toprule
H	& \PGL_2(\FF_8)	& \FF_8\rtimes\FF_8^\times 	& D_9		& D_{21}		& D_{63}		& C_{63}	\\
\toprule
|H|	& 2^3\cdot 3^2\cdot 7	& 2^3\cdot 7		& 2\cdot 3^2	& 2\cdot 3\cdot 7	& 2\cdot 3^2\cdot 7	& 2\cdot 7	\\
\midrule
{[G:H]}	& 1		& 9			& 28		& 12		& 4		& 8\\
\bottomrule
\end{array}
$$
Seeking a contradiction, we assume that  $[G:H]\neq 28$, in other words $H\neq D_9$, for all $z\in Z$.
Fix some $z\in Z$, and write $H=G_z$ for the stabilizer subgroup,   $I\subset H$ for the inertia subgroup, and $\FF_{3^\nu}=\kappa(z)$
for the residue field. Then $\Aut(\FF_{3^\nu})$ and  $\FF_{3^\nu}^\times$ are cyclic, 
of respective order $\nu$ and $3^\nu-1$.   
By the above table,  $d=7$ divides the order of $H$. 
Since the wild part of the inertia is a $p$-group, we see that $d=7$   divides the order of the image of $\phi:H\ra\Aut(\FF_{3^\nu})$
or the image of $\psi:I\ra \FF_{3^\nu}^\times$. 
The former gives $7\mid \nu$, the latter means $7\mid 3^\nu-1$, and in both cases we have $[\kappa(z):k]=\nu\geq 6$. 
In light of   \eqref{orbit degrees} the only possibilities are  $H=\PGL_2(\FF_8)$ or   $H=D_{63}$. 
In the former case, the group is simple, and it follows that both $\phi$ and $\psi$ are trivial, contradiction.
Thus $H=D_{63}$, or equivalently $[G:H]=4$. The abelianization of $H$ has order two, and we infer that $I=C_{63}$ and $\nu\geq 6$ is even.
In light of \eqref{orbit degrees} we actually have  $\nu=6$. 
Summing up,  in the right-hand side of \eqref{orbit degrees}, for every $z=z_i$ the  summand takes the form $[G:G_i]\cdot [k_i:k]\cdot m_i =24 m_i$,
contradiction.
\qed

\medskip
We also have   used a  purely   group-theoretic fact:

\begin{lemma}
\mylabel{lem:simple groups}
Let $H=H_1\rtimes H_2$ be a semidirect product of simple groups. Assume that $H$ is not isomorphic to the direct product $H_1\times H_2$.
Then $H_1$ is the only normal subgroup in $H$ different from $\{e\}$ and $H$.
\end{lemma}

\proof
Let $N\subset H$ be a normal subgroup. Then   $N_1=H_1\cap N$ is a normal subgroup in $H$,
and $N_2=N/N_1$ is a normal subgroup in $H_2$. If both of them are trivial, or both of them are non-trivial,
we have $N=\{e\}$ or $N=H$.  We have $N=H_1$ if $N_2$ is is trivial and $N_1$ is not. It remains
to exclude the case that $N_1$ is trivial and $N_2$ is not. Then $N $ is a complement for $H_1\subset H$, so $H\cong H_1\rtimes N$. Using that $N$ is normal, we see that this decomposition is actually a direct product, i.e., $H = H_1\times N \cong H_1\times H_2$.
This is a contradiction to our assumption.
\qed

\begin{lemma}\mylabel{lem:ree not product}
    The Ree group ${\rm Ree}(3)$ is not isomorphic to the product $\PGL_2(\FF_8)\times \Aut(\FF_8)$.
\end{lemma}

\proof
According to the list in the proof of Lemma~\ref{lem:only rational points}, $\PGL_2(\FF_8)$ has a cyclic subgroup of order $9$, and such subgroups are $3$-Sylow subgroups. But $\PGL_2(\FF_2)$ does not have an element of order $9$, so $\Aut(\FF_8)$ acts non-trivially on any of these. Therefore the $3$-Sylow subgroups of ${\rm Ree}(3)$ are non-commutative, and the lemma follows.
\qed

\subsection{Conclusion}
\mylabel{subsec:conclusion ree}

We finish this section by proving the description of the Deligne--Lusztig surface $\bar{X}(s_2s_1)$ as a specific ruled surface.

\emph{Proof for Theorem \ref{thm:structure g2twist}.}
The Deligne--Lusztig curves $\bar{X}(s_1)$ and $\bar{X}(s_2)$ are smooth projective curves of genus $15$ over $\FF_3$. They are naturally contained in $X=\bar{X}(s_2s_1)$, and we write $D_i = \bar{X}(s_i)$ when considering them as divisors on $X$. This is the same numbering as in the table in Section~\ref{subsec:tables}, and we see that $K_X = -2 D_2$ is negative.

Recall that the Deligne--Lusztig reduction method (see Corollary~\ref{cor:structure of ruled surface}) gives us (since $s_2 w \Frob(s_2) = \id$ has length $0=\ell(w)-2$) a morphism $f\colon X\to \bar{X}(s_1)$ that is a $\PP^1$-bundle which restricts to the identity morphism on $\bar{X}(s_1)$, i.e., $\bar{X}(s_1)$ is a section. So $X$ is a ruled surface over a curve of genus $15$.

Furthermore, we know that
\[
    (D_1D_2) = \# X(\id) = 28,
\]
see Lemma~\ref{lem:rational Bruhat decomposition}.

Since $K_X = -2 D_2$ and $K_X$ has degree $-2$ when restricted to the fibers, it follows that $\bar{X}(s_2)$ likewise is a section of $f$.

The Ree group $H$ acts faithfully on $X$ (Proposition~\ref{prop:kernel of action}) and compatibly on the Deligne--Lusztig curves. 
Thus for $R = \bar{X}(s_2)$ and $R'=\bar{X}(s_1)$ the assumptions of the Theorem~\ref{thm:characterization ruled ree} are satisfied.
\qed

\section{Appendix: Further tables}
\mylabel{sec:appendix}

We record here the coefficients of the expression of the canonical divisor as in Proposition~\ref{prop:canonical divisor DL} for those connected Deligne--Lusztig surfaces where for all $q$ at least one of the coefficients is strictly positive. (We list the cases up to isomorphism, e.g., for ${}^2 A_3$ the surfaces $\bX(s_1 s_2)$ and $\bX(s_3 s_2)$ are isomorphic because of the symmetry of the Dynkin diagram, and below we list only the data for the first one.)

In the Suzuki--Ree case ${}^2G_2$ and ${}^2 F_4$ we have $p=3$ and $p=2$, respectively; we then have $q = p^{\frac{2n+1}{2}}$ for some $n \ge 0$ and we write $q_{0} = p^n$.

{
\renewcommand{\arraystretch}{1.5}
\begin{longtable}{p{1.3cm}|C{1.5cm}C{1.5cm}|C{1.5cm}p{3.2cm}p{2.5cm}}
\toprule
\makecell{Type} 	& $r$, $s$	& $q = p^{\frac{s}{r}}$	& $w$ 	& \multicolumn{2}{l}{Coefficients in $K_X$}\\
\toprule
$C_2$	
        & $1$, $s$ & $p^s$ & $s_1s_2$			& $\dfrac{2 q^{2} - 2 q - 2}{q^{2} + 1}$, & $\dfrac{q - 3}{q^{2} + 1}$\\
\midrule
$G_2$	
        & $1$, $s$ & $p^s$ & $s_1s_2$			&  $\dfrac{3 q^{2} - 2 q - 2}{q^{2} - q + 1}$, & $\dfrac{2 q - 3}{q^{2} - q + 1}$\\[.4cm]
        & & & $s_2s_1$			& 		$\dfrac{q^{2} - 2}{q^{2} - q + 1}$, & $\dfrac{4 q - 5}{q^{2} - q + 1}$\\
\midrule
${}^2G_2$
        & $2$, $2n+1$ & $p^{\frac{2n+1}{2}}$ & $s_1s_2$	& $\dfrac{9q_{0}^{2} - 9 q_{0} + 2}{3q_{0}^{2} - 1}$, & $\dfrac{-5 q_{0} +3}{3q_{0}^{2} - 1}$\\
\midrule
${}^2A_3$	&
$2$, $2s_0$ & $p^{s_0}$ & $s_{1} s_{2}$ & $\dfrac{q^{2} - 2}{q^{2} - q + 1}$, & $\dfrac{q - 3}{q^{3} + 1}$\\[.4cm]
            & & & $s_{2} s_{1}$ & $ \dfrac{q^{3} - q^{2} - 2}{q^{3} + 1}$, & $\dfrac{2 q - 3}{q^{2} - q + 1}$\\[.4cm]
            & & & $s_{2} s_{3}$ & $\dfrac{q^{3} - q^{2} - 2}{q^{3} + 1}$, & $\dfrac{2 q - 3}{q^{2} - q + 1}$\\
\midrule
${}^2A_4$	& $2$, $2s_0$ & $p^{s_{0}}$	& 
$s_{1} s_{2}$ & $\frac{q^{4} - 2 q^{2} + 2 q - 2}{q^{4} - q^{3} + q^{2} - q + 1}$, & $ \frac{q^{3} + q - 3}{q^{4} - q^{3} + q^{2} - q + 1}$\\[.4cm]
              & & & $s_{2} s_{1}$ & $ \frac{q^{4} + q^{3} - 2 q^{2} + q - 2}{q^{4} - q^{3} + q^{2} - q + 1}$, & $ \frac{2 q - 3}{q^{4} - q^{3} + q^{2} - q + 1}$\\[.4cm]
              & & & $s_{1} s_{3}$ & $ \frac{q^{3} - q^{2} + q - 2}{q^{4} - q^{3} + q^{2} - q + 1}$, & $ \frac{2 q^{3} - q^{2} - 2}{q^{4} - q^{3} + q^{2} - q + 1}$\\[.4cm]
\midrule
${}^3D_4$ & 
$3$, $3s_0$ & $p^{s_{0}}$ & $s_1s_2$	& $\dfrac{q^{4} + q^{3} + q^{2} - 2 q - 2}{q^{4} - q^{2} + 1},\quad \dfrac{q^{2} + q - 3}{q^{4} - q^{2} + 1}$\\[.4cm]
           & & & $s_2s_1$	& $\dfrac{q^{4} - q^{3} + q^{2} - 2}{q^{4} - q^{2} + 1},\quad \dfrac{2 q^{3} + q^{2} - q - 3}{q^{4} - q^{2} + 1}$\\
\midrule
${}^2F_4$	& $2$, $2n+1$ & $p^{\frac{2n+1}{2}}$ &
$s_{1} s_{2}$ & $\frac{4q_{0}^{4} - 4q_{0}^{2} + 4q_{0} - 2}{4q_{0}^{4} - 4q_{0}^{3} + 2 q_{0}^{2} - 2 q_{0} + 1}$, & $\frac{4q_{0}^{3} + 2 q_{0} - 3}{4q_{0}^{4} - 4q_{0}^{3} + 2 q_{0}^{2} - 2 q_{0} + 1}$  \\[.4cm]
              & & & $s_{1} s_{3}$ & $ \frac{2 q_{0}^{3} - 2 q_{0}^{2} + 3 q_{0} - 2}{4q_{0}^{4} - 4q_{0}^{3} + 2 q_{0}^{2} - 2 q_{0} + 1}$, & $\frac{8 q_{0}^{3} - 2 q_{0}^{2} - 2}{4q_{0}^{4} - 4q_{0}^{3} + 2 q_{0}^{2} - 2 q_{0} + 1}$  \\[.4cm]
              & & & $s_{2} s_{1}$ & $ \frac{4q_{0}^{4} + 4q_{0}^{3} - 4q_{0}^{2} + 2 q_{0} - 2}{4q_{0}^{4} - 4q_{0}^{3} + 2 q_{0}^{2} - 2 q_{0} + 1}$, & $\frac{4q_{0} - 3}{4q_{0}^{4} - 4q_{0}^{3} + 2 q_{0}^{2} - 2 q_{0} + 1}$  \\[.4cm]
              & & & $s_{2} s_{4}$ & $ \frac{6 q_{0}^{3} - 2 q_{0}^{2} + q_{0} - 2}{4q_{0}^{4} - 4q_{0}^{3} + 2 q_{0}^{2} - 2 q_{0} + 1}$, & $\frac{4q_{0}^{3} - 2 q_{0}^{2} + 2 q_{0} - 2}{4q_{0}^{4} - 4q_{0}^{3} + 2 q_{0}^{2} - 2 q_{0} + 1}$  \\[.4cm]
              & & & $s_{3} s_{4}$ & $ \frac{4q_{0}^{4} + 4q_{0}^{3} - 4q_{0}^{2} + 2 q_{0} - 2}{4q_{0}^{4} - 4q_{0}^{3} + 2 q_{0}^{2} - 2 q_{0} + 1}$, & $\frac{4q_{0} - 3}{4q_{0}^{4} - 4q_{0}^{3} + 2 q_{0}^{2} - 2 q_{0} + 1}$  \\[.4cm]
              & & & $s_{4} s_{3}$ &  $\frac{4q_{0}^{4} - 4q_{0}^{2} + 4q_{0} - 2}{4q_{0}^{4} - 4q_{0}^{3} + 2 q_{0}^{2} - 2 q_{0} + 1}$, & $\frac{4q_{0}^{3} + 2 q_{0} - 3}{4q_{0}^{4} - 4q_{0}^{3} + 2 q_{0}^{2} - 2 q_{0} + 1}$  \\
\bottomrule
\end{longtable}
}



\begin{thebibliography}{ccccc}


\bibitem{Arbarello et al 1985}
E.\ Arbarello, M.\ Cornalba, P.\ Griffiths, J.\ Harris:
Geometry of algebraic curves. I. 
Springer, New York, 1985.

\bibitem{Artin 1974}
M.\ Artin:
Supersingular $K3$ surfaces.
Ann.\ Sci.\ \'Ecole Norm.\ Sup.\  7 (1974), 543--567. 

\bibitem{Atiyah 1957}
M.\ Atiyah:
Vector bundles over an elliptic curve.  
Proc.\ London Math.\ Soc.\   7  (1957), 414--452. 

\bibitem{Bjoerner; Brenti 2005}
A.\ Bj\"orner, F.\ Brenti: Combinatorics of Coxeter groups.
Springer, New York, 2005.

\bibitem{Bombieri; Mumford 1977}
E.\ Bombieri, D.\ Mumford:
Enriques' classification of surfaces in char.\ $p$,  II.
In: 
W.\ Baily, T.\ Shioda (eds.), Complex analysis and algebraic geometry, pp.\ 23--42.
Cambridge University Press, London, 1977.

\bibitem{Bonnafe; Dat; Rouquier 2017}
C.\ Bonnaf\'e, J.-F. Dat, R.\  Rouquier: 
Derived categories and Deligne-Lusztig varieties II.  
Ann.\ of Math.\   185 (2017),   609--670. 

\bibitem{Bonnafe; Rouquier 2003}
C.\ Bonnaf\'e, R.\  Rouquier:
Cat\'egories d\'eriv\'ees et vari\'et\'es de Deligne--Lusztig. 
Publ.\ Math.\ Inst.\ Hautes \'Etudes Sci.\   97 (2003), 1--59. 

\bibitem{Bonnafe; Rouquier 2006}
C.\ Bonnafé, R.\  Rouquier:
On the irreducibility of Deligne--Lusztig varieties.  
C.\ R.\ Math.\ Acad.\ Sci.\ Paris 343 (2006),  37--39. 

\bibitem{Bott; Samelson 1958}
R.~Bott, H. Samelson:
Applications of the theory of Morse to symmetric spaces.
Amer.~J.~Math.~80 (1958), 964--1029.

\bibitem{A 4-7}
N.\ Bourbaki:
Algebra II. Chapters 4--7.
Springer, Berlin, 1990.

\bibitem{LIE 4-6}
N.\ Bourbaki:
Groupes et alg\`ebres de Lie. 
Chapitres 4, 5 et 6. 
Masson, Paris, 1981.

\bibitem{LIE 7-8}
N.\ Bourbaki:
Groupes et alg\`ebres de Lie. 
Chapitres 7 et 8. 
Springer, Berlin, 2006.

\bibitem{Brosnan; Hong; Lee 2025}
P.\  Brosnan, J.\  Hong, D.\ Lee:
Geometry of regular semisimple Lusztig varieties.
Preprint, \arXiv{arXiv:2504.15868}.

\bibitem{Carter 1972}
R.\ Carter:
Simple groups of Lie type. 
Wileys, London--New York--Sydney, 1972.


\bibitem{Conrad; Gabber; Prasad 2015}
B.~Conrad, O.~Gabber, G.~Prasad,
Pseudo-reductive Groups.
2nd ed., Cambridge Univ.~Press, 2015.

\bibitem{Cossec; Dolgachev 1989}
F.\ Cossec, I.\ Dolgachev:
Enriques surfaces I.
Birkh\"auser, Boston, MA, 1989.

\bibitem{Coxeter 1958}
H.\ Coxeter:
The chords of the non-ruled quadric in $\operatorname{PG}(3,3)$.
Canadian J.\ Math.\ 10 (1958), 484--488. 

\bibitem{Deligne 1975}
P.\ Deligne:
Courbes elliptiques: formulaire d'apr\`es J.\ Tate. 
In: 
B.\ Birch, W.\ Kuyk (eds.),
Modular functions of one variable IV,  pp.\ 53--73. 
Springer, Berlin, 1975.

\bibitem{Deligne; Lusztig 1976}
P.\ Deligne, G.\ Lusztig:
Representations of reductive groups over finite fields. 
Ann.\ of Math.\ 103 (1976),  103--161.

\bibitem{Demazure 1974}
M.~Demazure:
Désingularisation des variétés de Schubert.
Ann.~Sci.~École Norm.~Sup.~(4), 7 (1974), 53--88.
 

\bibitem{SGA3-3}
M.~Demazure, A.~Grothendieck, \emph{Schémas en groupes (SGA 3) III}, Éd.~recomposée et annotée, SMF (2011).

\bibitem{Dieudonne 1971}
J.\ Dieudonn\'e:
La géométrie des groupes classiques.  
Springer, Berlin,  1971.

\bibitem{Digne; Michel 1991}
F.\ Digne, J.\  Michel:
Representations of finite groups of Lie type. 
Cambridge University Press, Cambridge, 1991.


\bibitem{Dolgachev; Kondo 2003}
I.\ Dolgachev, S.\ Kondo:
A supersingular K3 surface in characteristic 2 and the Leech lattice.
Int.\ Math.\ Res.\ Not.\ 2003,  1--23. 

\bibitem{Duncan 1968}
A.\ Duncan:
An automorphism of the symplectic group $\Sp_4(2n)$.
Proc.\ Cambridge Philos.\ Soc.\ 64 (1968), 5--9. 


\bibitem{Ekedahl 2003}
T.\ Ekedahl:
On non-liftable Calabi--Yau threefolds.
Preprint, \arXiv{math.AG/0306435}.

\bibitem{Everitt 2014}
B.\ Everitt:
A (very short) introduction to buildings. 
Expo.\ Math.\ 32 (2014),  221--247. 

\bibitem{Faber 2023}
X.\ Faber:
Finite $p$-irregular subgroups of $\PGL_2(k)$.  
Matematica 2 (2023),   479--522. 

\bibitem{Fanelli; Schroeer 2020}
A.\ Fanelli, S.\ Schr\"oer:
Del Pezzo surfaces and Mori fiber spaces in positive characteristic.
Trans.\ Amer.\ Math.\ Soc.\  373 (2020),  1775--1843.


\bibitem{Genestier 1996} 
A.\ Genestier:
Espaces sym\'etriques de Drinfeld.  
Ast\'erisque 234 (1996).


%
 

\bibitem{Gouthier; Tossici 2024}
B.\ Gouthier, D.\  Tossici:
Unexpected subgroup schemes of $\PGL_{2,k}$ in characteristic 2.
Preprint, \arXiv{arXiv:2403.09469}.

\bibitem{Grove 2002}
L.\ Grove:
Classical groups and geometric algebra.
American Mathematical Society, Providence, RI, 2002.

\bibitem{Haastert 1986}
B.\ Haastert:
Die Quasiaffinit\"at der Deligne-Lusztig-Variet\"aten. 
J.\ Algebra 102 (1986),  186--193. 

\bibitem{Hansen 1992}
J.\ Hansen:
Deligne--Lusztig varieties and group codes. 
In:
H.\ Stichtenoth, M.\ Tsfasman (eds.), Coding theory and algebraic geometry, pp.\ 63--81.
Springer, Berlin, 1992. 

\bibitem{Hansen; Pedersen 1993}
J.~P.\ Hansen, J.\ Pedersen:
Automorphism groups of Ree type, Deligne-Lusztig curves and function fields.
J.\ Reine Angew.\ Math.\ 440 (1993), 99--109.

\bibitem{Hansen 1999}
S.\ Hansen:
Canonical bundles of Deligne--Lusztig varieties.
Manuscripta Math.\ 98 (1999),  363--375.

\bibitem{Hartshorne 1977}
R.\ Hartshorne:
Algebraic geometry.
Springer, Berlin,  1977.

\bibitem{Henn 1978}
H.-W.\ Henn:
Funktionenkörper mit grosser Automorphismengruppe.  
J.\ Reine Angew.\ Math. 302 (1978), 96--115. 

\bibitem{Hilario; Schroeer 2023}
C.\ Hilario, S.\ Schr\"oer:
Generalizations of quasielliptic curves. 
\'Epijournal Geom.\ Alg\'ebrique 7 (2024), Article 23, 31 pp.

\bibitem{Hirschfeld; Korchmaros; Torres 2008}
J.~Hirschfeld, G.~Korchmáros, F.~Torres:
Algebraic Curves over a Finite Field.
Princeton Series in Appl.~Math. Princeton University Press (2008).

\bibitem{Hoelder 1895}
O.\ H\"older:
Bildung zusammengesetzter Gruppen.  
Math.\ Ann.\ 46 (1895), 321--422. 

\bibitem{Hua 1948}
L.\-K.\ Hua:
On the automorphisms of the symplectic group over any field.
Ann.\ of Math.\  49 (1948), 739--759. 


\bibitem{Husemoeller 1987}
D.\ Husem\"oller: 
Elliptic curves. 
Springer, New York, 1987.

\bibitem{Huybrechts 2016}
D.~Huybrechts:
Lectures on K3 Surfaces.
Cambridge Univ.~Press, 2016.

\bibitem{Ito 1994}
H.\ Ito:
The Mordell--Weil groups of unirational quasi-elliptic surfaces in characteristic $2$.  
Tohoku Math.\ J.\   46  (1994),   221--251.

\bibitem{Jantzen 2003}
J.~C.~Jantzen:
Representations of Algebraic Groups.
2nd ed., AMS, 2003.

\bibitem{Janusz; Rotman 1982}
G.\ Janusz, J.\ Rotman:
Outer automorphisms of S6.
Amer.\ Math.\ Monthly 89 (1982),  407--410. 

\bibitem{Kollar 1995}
J.\ Koll\'ar:
Rational curves on algebraic varieties.
Springer, Berlin, 1995.

\bibitem{Kondo 2020}
S.~Kond\={o}:
K3 surfaces.
Europ.~Math.~Soc., 2020.

\bibitem{Kondo; Schroeer 2021}
S.\ Kond\={o}, S.\ Schr\"oer:
Kummer surfaces associated with group schemes.
Manuscripta Math.\ 166 (2021), 323--342. 

\bibitem{Lang 1956}
S.\ Lang:
Algebraic groups over finite fields.
Amer.\ J.\ Math. 78 (1956), 555–563. 

\bibitem{Langer 2019}
A.~Langer:
Birational geometry of compactifications of Drinfeld half-spaces over a finite field.  
Adv.\ Math.\ 345 (2019), 861--908. 

\bibitem{Laurent; Schroeer 2024}
B.\ Laurent, S.\ Schr\"oer:
Para-abelian varieties and  Albanese maps.
Bull.\ Braz.\ Math.\ Soc.\ 55 (2024), 1--39.

\bibitem{Lipman 1969}
J.\ Lipman:
Rational singularities, with applications to algebraic surfaces and unique factorization.  
Inst.\ Hautes \'Etudes Sci.\ Publ.\ Math.\  36 (1969), 195--279.

\bibitem{Lusztig 1976}
G.\ Lusztig:
Coxeter orbits and eigenspaces of Frobenius.
Invent.\ Math.\ 38 (1976), 101--159.

\bibitem{Lusztig 1977}
G.\ Lusztig:
Irreducible representations of finite classical groups.
Invent.\ Math.\ 43 (1977), 125--175. 

\bibitem{Lusztig 1984}
G.\ Lusztig:
Characters of reductive groups over a finite field.
Princeton University Press, Princeton, NJ, 1984.

\bibitem{Malle; Testerman 2011}
G.~Malle, D.~Testerman:
Linear Algebraic Groups and Finite Groups of Lie Type.
Cambridge Univ.~Press, Cambridge, 2011.

\bibitem{Milne 2017}
J.\ Milne:
Algebraic groups.
Cambridge University Press, Cambridge, 2017.

\bibitem{Miyanishi 1977}
N.\ Miyanishi:
Unirational quasi-elliptic surfaces. 
Japan.\ J.\ Math.\ 3 (1977),   395--416.

\bibitem{Nikulin 1975}
V.\ Nikulin:
On Kummer surfaces.
Math.\ USSR  Izv.\ 9 (1975),  261--275.

\bibitem{Ono 1962}
T.\ Ono:
An Identification of Suzuki Groups With Groups of Generalized Lie Type.
Annals of Math.\ 75, (1962), 251--259.

\bibitem{Ono 1963}
T.\ Ono:
Correction to "An identification of Suzuki groups with groups of generalized Lie type''.
Ann.\ of Math.\  77 (1963), 413. 

\bibitem{Pedersen 1992}
J.\ Pedersen:
A function field related to the Ree group. 
In:
H.\ Stichtenoth and M.\ Tsfasman (eds.), Coding theory and algebraic geometry, pp.\ 122--131.
Springer, Berlin, 1992. 

\bibitem{Ree 1961}
R.\ Ree:
A family of simple groups associated with the simple Lie algebra of type $(G_2)$.
Amer.\ J.\ Math.\ 83  (1961), 432--462.

\bibitem{Rodier 2000}
F.~Rodier:
Nombre de points des surfaces de Deligne et Lusztig.
J.~Alg.~\textbf{227}, 706--766 (2000).

\bibitem{Schoof 1987}
R.~Schoof:
Nonsingular plane cubic curves over finite fields.
J.\ Combin.\ Theory Ser.\ A 46 (1987),  183--211. 


\bibitem{Schroeer 2023}
S.\ Schr\"oer:
There is no Enriques surface over the integers.
Ann.\ of Math.\ 197 (2023), 1--63.

\bibitem{Serre 1973}
J.-P.~Serre:
A course in arithmetic. 
Springer, New York-Heidelberg, 1973.

\bibitem{Serre 1979}
J.-P.\ Serre:
Local fields.
Springer, Berlin, 1979.


\bibitem{Shioda; Inose 1977}
T.\ Shioda, H.\ Inose:
On singular K3 surfaces. 
In:
Complex analysis and algebraic geometry, pp.\ 119--136.
Iwanami Shoten, Tokyo, 1977. 


\bibitem{Steinberg 1968}
R.~Steinberg:
Endomorphisms of linear algebraic groups.
American Mathematical Society, Providence, RI, 1968. 

\bibitem{Steinberg 1999}
R.~Steinberg:
The isomorphism and isogeny theorems for reductive algebraic groups. 
J.\ Algebra 216 (1999),  366--383.

\bibitem{Stichtenoth 1973}
H.\ Stichtenoth:
\"Uber die Automorphismengruppe eines algebraischen Funktionenk\"orpers von Primzahlcharakteristik. I.
Arch.\ Math.\  24 (1973), 527--544. 

\bibitem{Stichtenoth 1973b}
H.\ Stichtenoth:
\"Uber die Automorphismengruppe eines algebraischen Funktionenk\"orpers von Primzahlcharakteristik II.
Arch.\ Math.\ 24 (1973), 615--631.

\bibitem{Suzuki 1960}
M.\ Suzuki:
A new type of simple groups of finite order.
Proc.\ Nat.\ Acad.\ Sci.\ U.S.A. 46 (1960), 868--870. 

\bibitem{Tate 1975}
J.\ Tate:
Algorithm for determining the type of a singular fiber in an elliptic pencil.  
In: 
B.\ Birch, W.\ Kuyk (eds.),
Modular functions of one variable IV,  pp.\ 33--52. 
Springer, Berlin, 1975.

\bibitem{Tits 1962}
J.\ Tits:
Ovoïdes et groupes de Suzuki. 
Arch.\ Math.\ 13 (1962), 187--198. 

\bibitem{Tits 1995}
J.\ Tits:
Les groupes simples de Suzuki et de Ree.
S\'eminaire Bourbaki, Vol.\ 6, Exp.\ No.\ 210, 65–82. 
Soc.\ Math.\ France, Paris, 1995. 

\bibitem{Todd 1970}
J.\ Todd:
As it might have been.
Bull.\ London Math.\ Soc.\ 2 (1970),   1--4. 

\bibitem{Togashi; Uehara 2022} 
T.\ Togashi, H.\ Uehara:
Elliptic ruled surfaces over arbitrary characteristic fields
Preprint, \arXiv{arXiv:2212.00304}.

\bibitem{Tutte 1958}
W.\ Tutte:
The chords of the non-ruled quadric in $\operatorname{PG}(3,3)$.
Canadian J.\ Math.\ 10 (1958), 481--483. 

\bibitem{Wang 2024}
Y.\ Wang:
Cohomology and geometry of Deligne–Lusztig varieties for $\GL_n$.
Math.\ Z.\ 306, 69 (2024), 43 pp.

\bibitem{Wilson 2010a}
R.\ Wilson:
A new construction of the Ree groups of type ${}^2G_2$.  
Proc.\ Edinb.\ Math.\ Soc.\   53 (2010),   531--542. 

\bibitem{Wilson 2010b}
R.\ Wilson:
Another new approach to the small Ree groups.  
Arch.\ Math.\ (Basel) 94 (2010),  501--510. 

\end{thebibliography}
\end{document}